\documentclass{pspum-l} 
\usepackage{amsmath, amssymb, amsthm, amscd, amsfonts} 
\usepackage[all]{xy} 
\numberwithin{equation}{section} 

\newcommand{\teq}{\arabic{section}.\arabic{equation}}
\newcommand{\teql}{\Alph{section}.\arabic{equation}}

\newcommand{\sqr}[2]{{\vcenter{\vbox{\hrule height.#2pt\hbox{\vrule width.#2pt
height#1pt \kern#1pt\vrule width.#2pt}\hrule height.#2pt}}}}

\newcounter{eqcount}

\renewcommand{\labelenumi}{{{\rm (\teq \alph{enumi})}}} 
\newenvironment{edesc}{\refstepcounter{equation}\begin{enumerate}}%
{\end{enumerate}}
\newenvironment{triv}{\refstepcounter{equation}\begin{list}%
{{\hbox{\rm(\teq)\ }}} \item }{\end{list}}
\newenvironment{trivl}{\refstepcounter{equation}\begin{list}%
{{\hbox{\rm(\teql)\ }}} \item }{\end{list}}

\newcommand{\ring}[1]{{\mathbb #1}}
\newcommand\bZ{{\ring{Z}}}

\newcommand\bC{{\ring{C}}} 
\newcommand\bF{{\ring{F}}} \newcommand\bQ{{\ring{Q}}}

\newcommand{\csp}[1]{{\mathbb #1}}
\newcommand{\tsp}[1]{{\mathcal #1}}

\newcommand{\prP}{\csp{P}}
\newcommand{\afA}{\csp{A}}
 
\newcommand{\sC}{{\tsp{C}}} 
\newcommand{\sO}{{\tsp{O}}} 
\newcommand{\sQ}{\tsp{Q}}
\newcommand{\sP}{{\tsp {P}}} 
 
\newcommand{\sT}{{\tsp {T}}} \newcommand{\sH}{{\tsp {H}}}
\newcommand{\sX}{{\tsp {X}}} 
 \newcommand{\sD}{{\tsp {D}}}
\newcommand{\sG}{{\tsp {G}}} 
\newcommand{\sR}{{\tsp {R}}} 

 \newcommand{\bN}{{\csp {N}}}

\newcommand{\eql}[2]{{\rm (\ref{#1}\ref{#2})}} 

\newcommand{\vect}[1]{{\pmb #1}} 
 \newcommand{\bg}{\vect{g}}
 
\newcommand{\bp}{{\vect{p}}} 
\newcommand{\bv}{{\vect{v}}} 
 \newcommand{\bz}{{\vect{z}}}

\newcommand{\row}[2]{{#1_1,\ldots,#1_{#2}}}
\newcommand{\rowb}[3]{{#1_{#2},\ldots,#1_{#3}}}
\newcommand{\smatrix}[4]{{\big(\begin{array}{cc}
\!\lower2pt\hbox{$\scriptstyle#1$} &\lower2pt\hbox{$\scriptstyle#2$}\!
\\\! \raise2pt\hbox{$\scriptstyle#3$} &\raise2pt\hbox{$\scriptstyle#4$}
\!\end{array}\big)}}

\newcommand{\texto}[1]{{\textr{#1}}}
\newcommand{\GL}{\texto{GL}} \newcommand{\SL}{\texto{SL}}
 \newcommand{\ind}{\texto{ind}}
\newcommand{\PSL}{\texto{PSL}} \newcommand{\PGL}{\texto{PGL}}
 
 \renewcommand{\ni}{\texto{Ni}}
 
\newcommand{\textr}[1]{{\text{\rm #1}}}
\newcommand{\tr}{\textr{tr}} 
\newcommand{\abs}{\textr{abs}}  \newcommand{\cyc}{\textr{cyc}}
 
 \newcommand{\inn}{\textr{in}}
 \newcommand{\Aut}{\textr{Aut}}
\newcommand{\pr}{\textr{pr}}

\newcommand{\norm}{{\triangleleft\,}}

\newcommand{\BCL}{{\text{\rm Branch Cycle Lemma}}}

\newcommand{\rd}{\texto{rd}}

\newcommand{\ot} {\otimes}

\newcommand{\textb}[1]{{\text{\bf #1}}}
\newcommand{\bfC}{{\textb{C}}}
\newcommand{\longmapright}[2]{\smash{\mathop{\hbox to
#2pt{\rightarrowfill}}\limits^{#1}}}
\newcommand{\longmapleft}[2]{\smash{\mathop{\hbox to
#2pt{\leftarrowfill}}\limits^{#1}}}

\newcommand{\mapright}[1]{\smash{\mathop{\longrightarrow}\limits^{#1}}}

\newcommand{\np}{{+}}   \newcommand{\nm}{{-}}

\newcommand{\lrang}[1]{{\langle #1\rangle}}

\newcommand{\eqdef}{\stackrel{\text{\rm def}}{=}}

\newfont{\sevenrm}{cmr7}
\newfont{\bsevenrm}{cmbx7}
\newfont{\mathseven}{cmsy7}
\newfont{\bigmath}{cmsy10 scaled 1200}
\newfont{\fiverm}{cmr5}
\newfont{\bfiverm}{cmbx5}
\newfont{\hel}{cmbx10 scaled 1200}
\newfont{\eu}{eufb10}
\newfont{\sseu}{eufm5}
\newfont{\seu}{eufm7}
\newfont{\Cal}{cmmib10}
\newfont{\sCal}{cmmib7}
\newfont{\zch}{eusb10}

\theoremstyle{plain}
\newtheorem{thm}{Theorem}[section] 
\newtheorem{lem}[thm]{Lemma}
\newtheorem{princ}[thm]{Principle}

\newtheorem{prop}[thm]{Proposition}
\newtheorem{cor}[thm]{Corollary}

\theoremstyle{definition}
\newtheorem{defn}[thm]{Definition}
\newtheorem{exmp}[thm]{Example}
\newtheorem{guess}[thm]{Conjecture}
\newtheorem{quest}[thm]{Question}
\newtheorem{prob}[thm]{Problem}

\theoremstyle{remark}
\newtheorem{rem}[thm]{Remark}

\newcommand{\xs}{\times^s\!}
\newcommand{\wsp}{{$\,$---$\,$}} 

\def\pic #1 by #2 (#3){\vbox to #2{\hrule width 
#1 height 0pt depth 0pt\vfill\special{picture #3}}}
\def\scaledpicture#1
by #2 (#3 scaled #4){{\dimen0=#1 \dimen1=#2\divide\dimen0 by
 1000 \multiply\dimen0 by #4\divide\dimen1 by 1000 \multiply\dimen1 
by #4\pic \dimen0 by \dimen1 (#3 scaled #4)}}

\newcommand{\comm}[1]{{}}
\begin{document} 

\newcommand{\sB}{{\tsp {B}}}
\newcommand{\sN}{{\tsp {N}}}
\newcommand{\bga}{{\pmb\gamma}}
\newcommand{\Int}{{\text{\rm Int}}}
\renewcommand{\phi}{\varphi}
\newcommand{\Fr}{\text{Fr}}
\newcommand{\sW}{\tsp W}
\newcommand{\WP}{{\tsp {W}\tsp {P}}} 
\newcommand{\med}{\text{med}}
\newcommand{\pu}{\text{pu}}
\newcommand{\emed}{\text{e-med}}
\newcommand{\D}{{\text{DP}}}
\newcommand{\C}{{\text{C}}}
\newcommand{\CM}{\text{CM}}
\newcommand{\psigma}{{\pmb \sigma}}
\newcommand{\ab}{{{}_{\text{\rm ab}}}}

\title[The place of exceptionality]{The place of exceptional covers \\ among all diophantine relations} 

\author[M.~Fried]{Michael D.~Fried} 
\email{mfri4@aol.com, mfried@math.uci.edu}

\begin{abstract}
Let $\bF_q$ be the order $q$ finite field. An   $\bF_q$  cover $\phi:X\to Y$ of absolutely irreducible
normal varieties has a {\sl nonsingular locus}. Then, $\phi$ is {\sl 
exceptional\/} if it maps one-one on
$\bF_{q^t}$ points for $\infty$-ly many
$t$ over this locus. Lenstra
suggested a curve $Y$  may have an {\sl Exceptional (cover) Tower\/} over 
$\bF_q$ \cite{Lenstra}.   We construct it, and its  
canonical limit group and permutation representation, in general. We know all one-variable
tamely ramified rational function exceptional covers, and much on wildly ramified one variable polynomial exceptional
covers, from \cite{FGS}, \cite{GMS} and \cite{LMT}.  We  use  exceptional towers to form subtowers  
from any  exceptional cover collections. This gives us a language for separating known results from
unsolved problems. 

We generalize exceptionality to p(ossibly)r(educible)-exceptional covers by    
dropping irreducibility of $X$.  {\sl
Davenport pairs\/} (DPs) are significantly different covers of $Y$ with the same ranges (where maps are
non-singular) on 
$\bF_{q^t}$ points  for $\infty$-ly many $t$. If the range values  have the same multiplicities, we have an 
{\sl iDP}.  We show how a  pr-exceptional correspondence on $\bF_q$ covers characterizes a Davenport pair.    

You recognize exceptional covers and iDPs from their {\sl extension of
constants\/}  series. Our topics include some of their dramatic effects:  
\begin{itemize} \item How they produce universal  {\sl relations\/} between Poincar\'e series. 
\item  How  they relate to the Guralnick-Thompson genus 0 problem and to Serre's open image
theorem. \end{itemize} 
Historical sections capture Davenport's  late '60s desire to deepen ties between exceptional covers,  
their related  cryptology, and   the Weil
conjectures. 
\end{abstract}

\thanks{Acknowledgement of support from NSF \#DMS-0202259. These observations revisit 
topics from my first years as a mathematician, even from my graduate school conversations around A.~Brumer,
H.~Davenport, D.J.~Lewis, C.~MacCluer, J.~McLaughlin and A.~Schinzel. Reader: If you have an interest in exceptional
covers, and the use of towers in this paper, please note I gave Lenstra credit for his verbal (only) suggestion of an
exceptional tower in \cite{Lenstra}, and then read Rem.~\ref{toLenstra}. I borrowed the name i(sovalent)DP for the
special Davenport pairs of \S\ref{iDPint} from \cite{ABl}}
\maketitle 

\tableofcontents

\newcommand{\sing}{\text{\rm sing}}

\section{Introduction and historical prelude} The pizzazz in a canonical tower of exceptional
covers comes from group theory. \S\ref{introGp} explains that and my main results. Then,
\S\ref{prehist} uses the history of  exceptional covers to introduce notation 
(\S\ref{notation}). The main topic here is pr-exceptional covers with their 
pure covering space interpretation. I call its encompassing domain the {\sl monodromy method}. Its virtues 
include success with old problems and  interpretative flexibility, through additions to Galois
theory.  

I call zeta
function approaches to diophantine questions the {\sl representation method\/}. They come from  
representations of the Frobenius on cohomology.   In the 70's I connected the
monodromy and representation methods through particular problems (around  \cite{galStrat} based on Galois
stratification and \cite{FrGGCM} based on Hurwitz monodromy). Witness the general zeta function topics of
\cite[Chaps.~30-31]{FrJ4} (\cite[Chaps.~25-26]{FrJ}). Then, both subjects were still formative and used  different
techniques. The former  analyzed  spaces of covers through intricate group theory. The latter used abstract group
theory and mostly eschewed spaces. 

Now we have  {\sl Chow motives\/}, based much on Galois
stratification (\cite{DL}, \cite{nicaise}). These directly connect monodromy and representation methods. 
Worthy monodromy problems help hone topics in  Chow motives.  \cite{exceptTow} extends these to Chow
motives/zeta function problems while keeping us on the mathematical earth of pr-exceptional covers. 

\newcommand{\ns}{\text{\rm ns}}

\subsection{Results of this paper} \label{introGp} Let 
$K$ be any perfect field (usually a finite field or number field). Let 
$\phi: X\to Y$ be a degree $n$ cover (finite flat morphism) of {\sl absolutely irreducible\/} varieties  
(irreducible over the algebraic closure $\bar K$ of $K$) over $K$. They need not be projective; {\sl
quasiprojective\/} (locally open in a projective variety) suffices (see \cite[Part I]{Mum} for basics on
varieties). We assume from here that both are normal: Defined locally by integral domains 
integrally closed in their fractions.  Here is our definition of {\sl exceptionality\/} of $\phi$.
Let
$Y'$ be any Zariski open $K$ subset of $Y$ over which
$\phi$ restricts (call this
$\phi_{Y'}$) to a cover, $\phi^{-1}(Y')\to Y$, of  nonsingular varieties. The maximal 
{\sl nonsingular\/} locus for $\phi$, $Y_\phi^\ns$, is the complement of this set: the image of singular points of
$X$  union with singular points of $Y$. 

\begin{defn} \label{excDef} Call $\phi$ {\sl exceptional\/} if for some  $Y'$,  $\phi_{Y'}$  
is one-one on $\bF_{q^t}$ points for $\infty$-ly many $t$. Cor.~\ref{T2excCh} shows exceptionality is independent of 
$Y'$. For maps of normal curves, no choice of $Y'$ is necessary. \end{defn} 

From a cover of normal varieties we get an  {\/arithmetic\/} Galois
closure (\S\ref{fibprod})  $\hat
\phi:\hat X \to Y$. The geometric Galois closure,  $\ab\phi: \ab \hat X\to Y$, is the same construction done over
$\bar K$.   This gives two groups: Its geometric, $G_\phi=G(\ab \hat X/Y)$,  and
arithmetic, $\hat G_\phi=G(\hat X/ Y)$, {\sl monodromy groups\/}  (\S\ref{galClosure}). The former is a subgroup
of the latter.  The difference between the two groups is the result of {\sl extension of
constants\/}, the algebraic closure of $K$ in the Galois closure over $K$ is larger than $K$. Also, $\hat X$ is
absolutely irreducible if and only if $G_\phi=\hat G_\phi$. 

\cite{FrGGCM} phrased an extension of constants 
problem as generalizing complex multiplication. Several results used that 
formulation (for example, \cite{FrVAnnals} and \cite{GMS}). We 
refine it here to construct from any (degree
$n$)   
$\phi: X\to Y$  an  {\sl extension of
constants\/} series 
$\hat K_\phi(2)\le \hat K_\phi(3) \le \cdots \le \hat K_\phi(n\nm 1)$ (\S
\ref{galClosure}). 

Each $\hat K_\phi(k)$ is Galois over $K$  and its group has a canonical faithful permutation
representation $T_{\phi,k}$. Exceptional covers are at one extreme, dependent only on $\hat
K_\phi(2)/K$.  For $K$ a finite field, Lift Princ.~\ref{liftPrinc} (see Cor.~\ref{T2excCh}),  characterizes
exceptionality:    
$G(\hat K_\phi(2)/K)$ fixes no points under $T_{\phi,2}$. 

Such a  $\phi$ produces a transitive permutation representation $T_\phi: \hat G_\phi\to S_{T_\phi}$ on cosets of
$\hat G_\phi(1)=G(\hat X/X)$ in
$\hat G_\phi$:  $S_{T_\phi}$ denotes all permutations of these cosets.  We can identify $S_{T_\phi}$ 
(noncanonically)  with the symmetric group $S_n$ on $\{1,\dots,n\}$. This paper emphasizes canonical
construction of a certain infinite projective system of absolutely irreducible covers of $Y$ over $K$: $$\{\phi_i:
X_i\to Y\}_{i\in I}.$$ 

Such a projective system  gives projective
completions (limit groups) $\hat G_I \ge G_I$ with an associated (infinite) permutation representation. Essential to
a projective system is that for any two of its covers, another cover in it dominates both. Our absolute
irreducibility constraint  is serious. For two covers $\phi_i: X\to Y$,
$i=1,2$, to fit in any canonical projective system requires their fiber product $X_1\times_YX_2$ 
 have a unique absolutely irreducible factor over $K$ (see
\S\ref{fibProdpts}). 

To be truly canonical,  there should be at most one map between any two covers in the system. So,
such infinite canonical projective systems of absolutely irreducible covers over a field $K$
are rare. Here, though, is one. For $n$ prime to the
characteristic of $K$, and $\zeta_n$ any primitive $n$th root of 1, let
$C_n=\{\zeta_n^j,1\le j\le n\}$. Consider 
$\sT^\cyc_{\prP^1_y,K}\eqdef\{x^n\}_{\{n\mid K\cap C_n=\{1\}\}}$. The corresponding covers are $\prP^1_x\to
\prP^1_y=Y$ by
$x\mapsto x^n$ (notation of \S\ref{notation}). 

For any
finite field, $\bF_q$ this represents the tiny {\sl cyclic subtower\/} of  the whole exceptional tower
$\sT_{\prP^1_y,\bF_q}$ of $(\prP^1_y,\bF_q)$ (Prop.~\ref{excFibProp}). This category with fiber products includes 
all exceptional covers of $\prP^1_y$ over
$\bF_q$. It captures the whole subject of exceptionality, giving 
empyreal drama to a host of new problems. 

If you personally research (or just like) exceptional covers \wsp they are the nub of any public key-like 
cryptography (\S\ref{excScram} and \S\ref{prodExcCov})\wsp your special likes or expertises  will appear as subtowers
of the full tower. Examples, like the Schur and Dickson
subtowers of
\S\ref{SchurTow} and \S\ref{dickson},  clarify definitions of subtowers and their limit groups. 

Exceptional covers have practical uses outside cryptography. Here are three  
using rational function exceptional covers: Respectively in \S\ref{restGen0tame}, \S \ref{indecRat1} and 
\S\ref{cryptexc}.  
\begin{edesc} \item Producing $f\in K(x)$ (rational functions with $K$ a number field or finite
field) indecomposable over $K$, but decomposable over $\bar K$.
\item Interpreting Serre's O(pen)I(mage)T(heorem) as properties of exceptional rational functions. 
\item Creating general  {\sl relations\/} between zeta functions. 
\end{edesc}   

These applications motivate the questions we've posed in 
\S\ref{serreOpenIm}. Classical number theorists answered these  questions for the subtowers of \S\ref{classSubTows}.
So,
\S\ref{serreOpenIm} is an introduction to \cite{exceptTow}  and the full context for problems
posed in   
\S\ref{restGen0tame} (subtowers from modular curves) and \S\ref{wildRamExc} (subtowers with wild ramification).
There are two distinct ways  a given curve over a
number field could produce many tamely ramified exceptional covers of the projective line over finite fields. One is
 from reduction of covers that satisfy an exceptionality criterion according to Chebotarev's density
theorem. Another is less obvious, but it is through the reduction of curves that have the {\sl median value
property\/} (\S\ref{nameExc}). We use \cite{Se-Cheb} and \cite{ Se-La} to tie the correct primes
of reduction to $q$-expansions of automorphic functions (\S\ref{expPrExc}, continued in \cite{exceptTow}).   
 
\S\ref{wildRamExc} outlines how to  describe the limit group of the
subtower
$\WP_{\prP^1_y,\bF_q}$ (of the exceptional tower over
$(\prP^1_y,\bF_q)$, $q=p^u$) that indecomposable polynomials, wildly ramified over $\infty$, generate. This 
suggests how to generalize \wsp even arithmetically \wsp
aspects of Grothendieck's famous theorems on curve fundamental groups. 

\S\ref{excScram} and Ques.~\ref{invFunct2} consider  
exceptional rational functions $\phi:\prP^1_x\to \prP^1_y$ as  {\sl scrambling\/} functions.  The 
combinatorics of Poincar\'e series allow us to ask how the periods of those scramblers vary as the finite field
extension changes.

The full role of exceptionality,  appears in  {\sl p(ossibly)r(educible)-exceptionality\/} 
(starting in
\S\ref{prexceptSt}). Davenport's problem  (\S\ref{davPairs})  is a special case of pr-exceptionality.
 Finally, \S\ref{prehist} and
\S\ref{monConn} take us to the historical topics started by Davenport and Lewis (\S\ref{monConna};  from whence
exceptionality sprang) and by Katz (\S\ref{monConnb}). These motivated our using the extension of constants series
to  put all these exceptional covers together.

\subsection{Primitivity and a prelude to the history of exceptionality} \label{prehist}  Most topics until
\S\ref{classSubTows} work as well for $Y$ of arbitrary dimension. We, however, understand tame exceptional covers of
curves  through the {\sl  branch cycle\/} tools of \S\ref{introBC}. These allow being constructive. 

To shorten the
paper, I limit use of branch cycles and associated {\sl Nielsen classes\/} (a bare
bones  review is in \S\ref{nc1}) to a necessary minimum. \S\ref{dickson} uses branch
cycles to give precise generators of the limit group for the Dickson subtowers. Another example is in the Nielsen
class version setup for Serre's Open Image Theorem (OIT in
\S\ref{serreOpenIm}). This approach to modular curves generalizes to form other
systems of tamely ramified exceptional covers in  \cite{exceptTow}.  App.~\ref{MullerDPs} uses
\cite{thompson} to guide the so-inclined reader to the most modern use of Nielsen classes. These  
example polynomial families from Davenport's problem  seem so explicit, it must be surprising  we can't do them
without some version of branch cycles. Finally, \S\ref{repRET} discusses how \cite{exceptTow} will use \cite{FrMez}
to replace branch cycles (Riemann's Existence Theorem) when covers wildly ramify.  Given the structure of
Prop.~\ref{excFibProp}, unsolved problems on subtowers of wildly ramified covers are a fine test  for this
method.

\subsubsection{Using primitivity in  exceptional covers} 
Let $\phi: X\to Y$ be a cover of absolutely irreducible
(normal) varieties over a field $K$. Call $\phi$ {\sl decomposable\/}
(over $K$) if it decomposes as a chain of $K$ covers $$X\mapright{\phi'} W \mapright{\phi''}
Y \text{ with
$\phi'$ and
$\phi''$ of degree at least 2}.$$ Otherwise it is {\sl indecomposable\/} or {\sl
primitive\/} (over $K$). From
\cite{FrSchur} to
\cite{FGS}, much has come from observing that the arithmetic monodromy group
(in its $\deg(\phi)$ permutation representation) is primitive if and only if the cover is
primitive. 

\begin{lem} \label{decompKtoKbar} Also, 
assume $\phi$ is totally ramified over some absolutely
irreducible $K$ divisor  (for curves a $K$ point) of
$Y$.  Then (if
$(\deg(\phi),\text{char}(K))=1$, necessary  from
\cite[Cor.~11.2]{FGS}): $\phi$ decomposes over $K$ $\Leftrightarrow$ $\phi$ decomposes over $\bar K$.  \end{lem}
The proofs of \cite[Prop.~3, p. 101]{FrPap0} and \cite[Thm.~3.5]{FrPap1} are readily adapted to prove this, and it  a
special case of
\cite[Lem.~4.4]{FGS}. 
  
Suppose $K$ is a number field or finite field. In the former case let $\sO_K$
be its ring of integers.  Let $k_f=k_{f,K}$ be the number of absolutely irreducible $K$ components
of 
$\prP_x^1\times_{\prP^1_z}\prP^1_x\setminus \Delta$ (\S\ref{fibprod}). So, $k_{f,\bar K}$
might be  larger than $k_{f,K}$. 
\cite{dlcong1} used exceptional to mean  $k_{f,K}$ is 0 (\S\ref{monConn}). 

Davenport and Schinzel visited University of Michigan  
in 1965--1966 (see \S\ref{expexcTowers}). They discussed many polynomial
mapping problems. This  included   Schur's 1923 \cite{SchurConj} conjecture, whose hypothesis and
conclusion are the second paragraph of Lem.~\ref{basicSchur} when $\bQ=K$   
(\cite[Thm.~1]{FrSchur}; notation from \S\ref{SchurTow}). Recall the degree $n$ Tchebychev polynomial, $T_n(x)$:  
$T_n(\frac{x+1/x}2)=\frac{x^n+1/x^n}2$ (\S\ref{dickson}). 
\begin{lem} \label{basicSchur} Suppose $f\in K[x]$ is  
indecomposable, $(\deg(f),\text{char}(K))=1$ and $k_{f,\bar K}\ne 1$. Then,   $f$ has prime degree and 

\begin{triv} \label{cycCheb}   either $\lambda_1\circ
f\circ
\lambda_2^{-1}(x)$ is cyclic  
($x^{\deg(f)}$) or Chebychev ($T_{\deg(f)}(x)$)  for some $\lambda_1,\lambda_2\in \afA(\bar K)$ (\S\ref{notation};
Prop.~\ref{ordSchurT} for precision on the $\lambda\,$s).
\end{triv}

\noindent Let  $K$ be a number field, $g\in \sO_K[x]$ (maybe decomposable). 
\begin{triv} \label{schurCov} 
\label{Qexc} Assume     
$g:\sO_K/\eu{p}\to \sO_K/\eu{p}$ is one-one for $\infty$-ly many primes $\eu{p}$. \end{triv}
\noindent Then, $g$ is a composition over $K$ of polynomials $f$ satisfying
\eqref{cycCheb}. 
\end{lem}

\cite{MacCluer} earlier showed that if $f\in \bF_q[x]$ gives a tame 
ramified cover over
$K=\bF_q$ with $k_{f,K}=0$,  then $f:\bF_q\to \bF_q$ is a
one-one map.    
\cite{FrThmMac} quoted \cite{MacCluer} for the name exceptional. It also showed how generally 
MacCluer's conclusion applied, to any finite cover $\phi: X\to Y$ of
absolutely irreducible nonsingular varieties  (any dimension, even if wildly ramified) satisfying the
general condition
$k_{\phi,K}=0$.  

\subsubsection{Primitivity and grabbing a generic group} \label{primGrab} If you have ever done a crossword
puzzle, then you'll recognize this situation. You have a clue for 7 Across, a seven letter word, but you have
only filled in previously the 4th letter: ...E...: Say, the clue is \lq\lq Bicycle stunt.\rq\rq\ You will be happy
for the moment to find one word that fits, even if it isn't the precise fill for the crossword. Shouldn't that be
easier to do than to be given another letter W..E... that constrains you further? 

The lesson  is that you can't seem to \lq\lq grab\rq\rq\ a word at random, but need clues that force you to the
\lq\lq right\rq\rq\ word. That also applies to groups. They are too discrete and too different between
them. If you aren't a group theorist you likely won't easily grab a primitive, not doubly
transitive, group at random.  Exceptional covers and Davenport's problem focused group theory on a set of problems
that were the analog of having to fill a suggestive set of letters in a crossword clue.  

That tantalized John Thompson and Bob Guralnick to push to complete
solutions for a particular problem where the constraints included that the group was the
monodromy of a genus 0 cover over the complexes. \S\ref{davPairProb} and \S\ref{prodExcCov} show why examples that
were telling in the genus 0 problem (over the complexes) applied to produce an understanding of wildly ramified
covers in positive characteristic.  
The Guralnick-Thompson  genus 0 problem succeeded technically and practically. It was propitious: It took group
theory beyond the classification stage that dominated the simple group program;  yet it made much of that
classification work. 

\subsection{Notation} \label{notation}
We denote projective 1-space,  
$\prP^1$, with a specific uniformizing variable
$z$ by   
$\prP^1_z$. This decoration tracks distinct domain and
range copies of $\prP^1$. 

\subsubsection{Group notation} \label{gpNot} We use some classical algebraic groups over a field $K$: especially
affine groups and groups related to them. If $V=K^n$, then the action of $\GL_n(K)$ on $V$ produces a semi-direct
product group
$V\xs
\GL_n(K)$. Represent its elements as pairs $(A,\bv)$ so the multiplication is given by
\begin{equation} \label{leftAct} (A_1,\bv_1)(A_2,\bv_2)=(A_1A_2,(\bv_1)A_2+\bv_2).\end{equation}  Here we use a right
action of matrices on vectors. Regard this whole group as permuting elements of $V$ by the action $(A_1,\bv_1)$ maps
$\bv\in V$ to
$(\bv)A_1+\bv_1$.  If you prefer a left action of matrices on vectors, then it is convenient to write $(A,\bv)$ as
$\smatrix A {\bv} 0 1$. Then, multiplication is that  expected from matrix multiplication: 
\begin{equation} \label{rightAct}\smatrix {A_1} {\bv_1} 0 1  \smatrix {A_2} {\bv_2} 0
1=\smatrix {A_1A_2} {\bv_1+A_1(\bv_2)} 0 1. \end{equation}  Represent $\bv\in V$ as $\bigl({\bv
\atop 1}\bigr)$: $V\xs
\GL_n(K)$ permutes $V$ by left multiplication. 

A subgroup $V\xs H$ with $H\le \GL_n(K)$ is an {\sl affine\/} group. If $K$ is a finite field, it is an easy exercise
to show the action of $V\xs H$ is primitive if and only if $H$ acts irreducibly (no proper subspaces) on $V$.   

We use a special notation for   
$\afA(K)$, affine transformations
$$x\mapsto ax+b, (a,b)\in K^*\times K.$$ M\"obius transformations are $\PGL_2(K)$. We use their 
generalization  to $\PGL_{u+1}(K)$ acting on $k$-planes, $k\le u-1$, of 
$\prP^u(K)$ ($K$ points of projective $u$-space).  Denote the set of
$r$ distinct unordered points of
$\prP^1_z$ by $U_r=((\prP_z^1)^r\setminus \Delta_r)/S_r$ ($\Delta=\Delta_r$ in \S\ref{fpcomm}). Quotient by
$\PGL_2(\bC)$  acting diagonally (commuting with $S_r$ on $(\prP_z^1)^r$). If $r=4$, these $\PGL_2(\bC)$
orbits form the classic
$j$-line
$\prP^1_j$ minus
$\infty$ \!\cite[\S2.2.2]{BFr}. 

We use groups and their representations, especially permutation
representations to translate the geometry of covers. In practice, as in  \S\ref{DickMon}, our usual setup
has a subgroup $G$ of $S_n$, the symmetric group of degree $n$ with multiplications from the {\sl right\/}. Example:
For $g_1=(2\,3),g_2=(1\,2)(3\,4)\in S_4$, $(2)g_1g_2=4$ gives the effect of the product of $g_1g_2$ on 2. (Action
on the left would give $g_1g_2(2)=1$.) Abstract notation of \S\ref{crypt}  
expresses the canonical permutation representation of a cover as  
$T:G\to S_V$: $G$ acts on a set
$V$. 

Recall:  A cover is {\sl tame\/} if over its ramification locus, its inertia
groups have orders prime to the characteristic. Since we restrict our maps to avoid singular sets, on the
varieties in the cover, there is no special subtlety to this definition. 

\subsubsection{Riemann Hurwitz} \label{RiemHur} 
An element
$g\in S_n$ has an index
$\ind(g)=n-u$ where $u$ is the number of disjoint cycles in $g$. Example: $(1\,2\,3)(4\,5\,6\,7)\in
S_8$ (fixing the integer 8) has index $8-3=5$. 
Suppose $\phi: X\to \prP^1_z$ is a degree
$n$ cover
(of compact Riemann surfaces). We assume the reader is familiar with computing the genus $g_X$ of 
$X$ given a branch cycle description $\bg=(\row g r)$ for $\phi$ (\S\ref{nc1}):
$2(n+g_X-1)=\sum_{i=1}^r
\ind(g_i)$ (\cite[\S2.2]{VB} or \cite[Chap.~4]{FB}). 

\subsubsection{Frobenius progressions and fiber products} \label{frobProg} We need a precise
notation for certain types of arithmetic progressions and their unions. Let $n$ be an
integer that refers to a modulus for an arithmetic progression $$A_a=A_{a,n}=\{a+kn\mid
0\le k\in \bZ\}
\text{ with }0\le a\in \bZ.$$ Call
$A_a$ a {\sl full\/} progression if $a< n$. Given $n$, any $u\in A_a$ 
defines $A_a$ uniquely. A full {\sl Frobenius\/} progression
$F_a=F_{a,n}$ is the union of the full arithmetic progressions $\mod n$ defined by the
collection of residue classes  $a\cdot (\bZ/n)^* \mod n$.  Example: The full
Frobenius progression
$F_{2,12}$ is $A_{2,12}\cup A_{10,12}$.

\section{Fiber products and extension of  constants} 
This short section has two topics even an experienced
reader has never seen before: pr-exceptionality (\S\ref{prexceptSt}) and the extension of constants series 
(\S\ref{galClosure}).  We use fiber products for the latter. Interpreting exceptionality is an example
(\S\ref{exCheck}). 

\subsection{Fiber products} \label{fibprod} There are diophantine subtleties in our use of fiber products  (see \S\ref{fibProdpts}), 
for we remain in the
category of normal varieties. 

\subsubsection{Categorical fiber product} \label{fpcomm}
Assume $\phi_i: X_i\to Y$, $i=1,2$, are two
covers (of normal varieties) over $K$. The set theoretic fiber product has geometric points
$$\{(x_1,x_2)\mid x_i\in X_i(\bar K), i=1,2,\ \phi_1(x_1)=\phi_2(x_2)\}.$$ Even if these are curves, this will not be
normal at $(x_1,x_2)$ if  $x_1$ and $x_2$  both ramify over $Y$. The {\sl categorical\/} fiber product
of two covers here means the normalization of the result. Its components will be disjoint, normal varieties. We
retain the notation
$X_1\times_YX_2$ often used for the purely geometric fiber product. An $\bF_q$ point $x$ of
$X$ ($x\in X(\bF_q)$)  means a geometric point in $X$ with coordinates in
$\bF_q$. 

When $\phi_1=\phi_2$ has degree at least 2 the fiber product, $X\times_YX$, has at least two components 
(if $\deg(\phi)=n>1$): one the diagonal. Denote $X\times_Y X$ minus the diagonal component by
$X^2_Y\setminus \Delta$. Then, for any integer $k$, denote the $k$th iterate of  the
fiber product minus the {\sl fat diagonal\/} (pairwise diagonal components) by
$X^k_Y\setminus \Delta$. This is empty if $k>n$. There is a slight abuse in
using the symbol $\Delta$ for all $k$. 

Any $K$ component of $X^n_Y\setminus
\Delta$ is a $K$ Galois closure $\hat \phi:\hat X\to Y$ of $\phi$, unique up to $K$
isomorphism of Galois covers of $Y$. The permutation action of
$S_n$ on $X^n_Y\setminus \Delta$ gives the Galois group  $G(\hat X/Y)$ as the subgroup
fixing $\hat X$. When considering a family of
covers $\{X_s\to Y_s\}_{s\in S}$ over (even) a smooth base space $S$, only in special 
situations do we expect the Galois closure construction to work over $S$.  In
characteristic 0 (where there is a locally smooth ramification section) there is an \'etale cover  
$\hat S\to S$ over which the Galois construction does occur (Rem.~\ref{consthS}).

\begin{rem} \label{consthS}  To effect construction of a Galois closure canonically for a family of curve covers in
characteristic 0,  use for $\hat S$ the pullback to the {\sl inner\/}  Hurwitz space $\sH(G,
\bfC)^\inn$ (notation from \S\ref{nc1}) as in \cite{FrVMS}.  Practical $A_n$ examples 
are in \cite[\S A.2.4, especially Prop.~A.5]{altGps}.   A theme of
\cite{FrMez}: Expect such a $\hat S$ in
positive characteristic only if a family of projective curve covers tamely ramifies. Further,  its computation is
explicitly understood only if $|G|$ is prime to the characteristic. \end{rem}

\subsubsection{pr-exceptional covers} \label{prexceptSt} 
 Let  $Y$ be an absolutely
irreducible normal variety over $\bF_q$. Our constructions are usually over an
absolutely irreducible base. As in Def.~\ref{excDef}, consider the restriction $\phi_{Y'}$ of a
cover
$\phi$ over some open $Y'$ where it becomes a morphism of nonsingular varieties.  

\begin{defn} \label{prexcDef} A {\sl pr-exceptional\/} (pr   for {\sl possibly
reducible}) cover
$\phi: X\to Y$  is one with
$\phi_{Y'}$  surjective on $\bF_{q^t}$ points for infinitely many $t$ for any allowable $Y'$. \end{defn}

We permit $X$
to have no absolutely irreducible $\bF_q$ component. (Since it is normal, such an $X$ 
has no $\bF_q$ points.)  It is essential for Davenport pairs (\S\ref{davPairs}) to
consider cases where
$X$ may have several absolutely irreducible
$\bF_q$ components. If
$X$ is absolutely irreducible, then a pr-exceptional cover 
$\phi$ is exceptional. 

Here is a special case of \cite{FrThmMac}. In \cite{FGS} it has a group
theory proof. In our generality (allowing $Y$ of arbitrary dimension) we need the special case of
Princ.~\ref{liftPrinc} applied to exceptional covers. 

\begin{prop}[Riemann Hypothesis Proposition] \label{RHLem} Suppose $\phi: X\to Y$ is a cover of
absolutely irreducible normal varieties (over $\bF_q$). Then $\phi$ exceptional is equivalent
to each of the following. 
\begin{edesc} \label{rh} \item  \label{rha} $X^2_Y\setminus \Delta$ has no absolutely
irreducible
$\bF_q$ component. \item  \label{rhb}  For any choice of $Y'$ in Def.~\ref{excDef}, there are 
$\infty$-ly many
$t$ with $\phi_{Y'}$ surjective (and one-one) on  
$\bF_{q^t}$ points. \end{edesc} 

Let $E_\phi(\bF_q)$ be those $t$ where \eql{rh}{rha} holds with $q^t$
replacing $q$: $X^2_Y\setminus \Delta$ has no absolutely irreducible
$\bF_{q^t}$ component.  A chain  $X\mapright {\scriptscriptstyle \phi'}
X'\mapright{\scriptscriptstyle\phi''} Y$ of covers is exceptional if and  only if each cover
in the chain is exceptional. Then: 
$$E_{\phi''\circ\phi'}(\bF_q)=  E_{\phi''}(\bF_q)\cap E_{\phi'}(\bF_q).$$
\end{prop} 

We call $E_\phi(\bF_q)$ the {\sl exceptionality set of $\phi$\/} (over $\bF_q$).
\S\ref{galClosure} restates exceptionality using the geometric-arithmetic 
monodromy groups
$(G_\phi,\hat G_\phi)$ of 
$\phi:X\to Y$. The quotient $\hat G_\phi/G_\phi$ is canonically isomorphic to the cyclic
group
$\bZ/d(\phi)$ where $d(\phi)$ defines the degree of the extension of constants field. A
 quotient $\bZ/d(X^2_\phi)$ of $\bZ/d(\phi)$ indicates precisely which values $t$ are
in $E_\phi(\bF_q)$ (Cor.~\ref{geomExcept}).  The exceptionality set
$E_\phi$ is a union of full Frobenius progressions. This extends to pr-exceptional
(Princ.~\ref{liftPrinc}): It has a Galois characterization and the
pr-exceptionality set $E_{\phi}(\bF_q)$ is a union of full Frobenius progressions. 

\subsubsection{Galois group of a fiber product} \label{galFiber} Recall the fiber
product of two surjective homomorphisms $\phi^*_i: G_i\to H$, $i=1,2$: $$G_1\times_H
G_2=\{(g_1,g_2)\in G_1\times G_2\mid \phi_1^*(g_1)=\phi_2^*(g_2)\}.$$ 

The following  hold from an equivalence of categories with fiber
product \cite[Chap.~3, Lem.~8.11]{FB}. Suppose
$\phi_i:X_i\to Y$ are two covers, with geometric (resp.~arithmetic) monodromy group
$G_{\phi_i}$ (resp.~$\hat G_{\phi_i}$), $i=1,2$. Let $\ab \hat X$ (resp.~$\hat X$) be
the maximal simultaneous quotient of $\ab \hat X_i\to Y$ (resp.~$ \hat X_i\to
Y$), $i=1,2$. Then the geometric (resp.~arithmetic) monodromy group of the fiber product
$$(\phi_1,\phi_2): X_1\times_YX_2\to Y$$ is $G_{\phi_1}\times_H G_{\phi_2}$
(resp.~$\hat G_{\phi_1}\times_H \hat G_{\phi_2}$) with $H=G(\ab \hat X/Y)$
(resp.~$G(\hat X/Y)$). Note: Determining $H$ is often the
hard part.

We now consider the natural permutation representation attached to a Galois
closure of a fiber product. Let $T_i: G_i\to S_{V_i}$, $i=1,2$, be permutation 
representations, $i=1,2$ (as in \S\ref{crypt}). These representations produce a tensor representation on the 
categorical fiber product as $T: G_1\times_HG_2\to
S_{V_1\times V_2}$ (as in \S\ref{davPairs}). 

\subsubsection{Introduction to branch cycles} \label{introBC} Now assume $Y=\prP^1_z$, the context for classical
exceptional covers. If we restrict to tame covers, then {\sl
branch cycle\/} descriptions often  figure out everything in one fell swoop.  Assume 
$\bz$ contains all branch points of both $\phi_1$ and $\phi_2$. As in
\S\ref{nc1}, branch cycles start from a fixed choice of classical generators on $U_\bz$ (we assume
this given; with
$r$ points in
$\bz$). \S\ref{algBrCyc} explains how this applies to tame  covers in positive characteristic. 

\begin{prop} \label{useBC} Assume $G_i$ is a geometric monodromy group  for
$\phi_i$,
$i=1,2$. Suppose
$\bg^i$ (resp.~$\bg$) is the branch cycle description for
$\phi_i$,
$i=1,2$ (resp.~$(\phi_1,\phi_2)$). Then,  $g_k=(g_k^1,g_k^2)$, $k=1,\dots,r$. The
orbits of $T$ on $\lrang{\bg}$ correspond to the absolutely irreducible components of the
fiber product $X_1\times_{\prP^1_z} X_2$. 
\end{prop}

\noindent Finding
$\bg$ is usually the hard part.  Prop.~\ref{branchCycleqn} has a practical
example.

\newcommand{\exc}{\text{{\rm exc}}}
\subsection{The extension of constants series} \label{galClosure} Many arithmetic
properties of covers appear from  an extension of constants in going to the Galois closure
of a cover.   
Let
$\phi:X\to Y$ be a $K$ cover, with $\deg(\phi)=n$, of absolutely irreducible (normal varieties). As
in \S\ref{fibprod}, let
$\hat 
\phi: \hat X\to Y$ be its arithmetic Galois closure with group $\hat G_\phi$. Denote the group of $\hat X\to X$ by
$\hat G_\phi(1)$, with similar notation for $\ab \hat X$. 

\subsubsection{Iterative constants}  \label{itConsts}  Let $\hat K_\phi(k)=\hat
K(k)$ be the minimal definition field of the  collection of (absolutely irreducible) $\bar
K$ components of
$X^k_Y\setminus
\Delta$, $1\le k\le n$. Then, the kernel of
$\hat G_\phi\to G(\hat K(n)/K)$ is $G_\phi$. Since $X^k_Y\setminus
\Delta$ has definition field $K$, each extension $\hat K(k)/K$
is Galois. Call it the {\sl
$k$th extension of constants field}. Further, the group $G(\hat K_\phi(k)/K)$ acts as
permutations of the absolutely irreducible components of $X^k_Y\setminus
\Delta$. Denote the corresponding permutation representation on these components by 
$T_{\phi,k}$.  

There is a natural sequence of quotients
$$G(\hat X/Y)\to G(\hat K_\phi(n)/K)\to \cdots \to G(\hat K_\phi(k)/K)
\to
\cdots
\to G(\hat K_\phi(1)/K).$$ Here 
$G(\hat K(1)/K)$ is trivial if and only if $X$ is absolutely irreducible.
As in Cor.~\ref{geomExcept} the exceptional cover topic primarily deals with the fields
$\hat K(2)$. 
We record here an immediate consequence of Prop.~\ref{RHLem}. 

\begin{cor} \label{T2excCh} For $K$ a finite field, $G(\hat K_\phi(2)/K)$ having no fixed
points under
$T_{\phi,2}$ characterizes $\phi$ being exceptional. \end{cor}

The only
general identity between these fields $\{\hat K_\phi(k)\}_{k=2}^n$ is in the next lemma. For
any ordered subset
$I=\{i_1<\cdots < i_k\}\subset
\{1,\dots,n\}$, denote projection of $X^n_Y\setminus \Delta$ on the coordinates
of $I$ by $\pr_I$. 

\begin{lem} The map $\pr_I: X^n_Y\setminus \Delta \to X^k_Y\setminus \Delta$ is a $K$
map. For $k=n-1$ it is an isomorphism. In particular, $\hat K_\phi(n)=\hat K_\phi(n-1)$.
\end{lem}

\begin{proof} The ordering on the coordinates of $X^n_Y\setminus \Delta$ is
defined over $K$. So, picking out any coordinates, as $\pr_I$ does, is also. Since
$X^k_Y\setminus \Delta$ is a normal variety, if $\pr_I$ is generically one-one then it is an
isomorphism. Off the discriminant locus 
points of $X^n_Y\setminus \Delta$ look like $(\row x n)$ where $\row x {n-1}$
determine $x_n$, the remaining point  over
$\phi(x_1)\in Y$. So, when $I=\{1<\cdots<n-1\}$, the map is one-one. 
\end{proof}

\begin{rem} \cite[App.~B]{exceptTow} shows how the arithmetic monodromy group of
$A_n$ covers is at the other extreme (depending solely on
${\hat K_\phi}(n\nm1)$. \end{rem}

\subsection{Explicit check for exceptionality} \label{exCheck} 
Apply the extension of constant series  
when $K=\bF_q$ and  $\hat \bF_q(k)$ is the $k$th extension of constants field. We write
$G(\hat \bF_q(k)/\bF_q)$ as $\bZ/d(\phi,k)$. 
The {\sl extension of constants\/} group is $$\hat G_\phi/G_\phi=G(\hat
\bF_{q,\phi}/\bF_q)\eqdef \bZ/d(\phi,n).$$ It
defines  $\hat \bF_q(n)=\hat \bF_{q,\phi}(n)$, the minimal field over which
$\hat X$ breaks into absolutely irreducible components. For 
$X$ absolutely irreducible, $\hat G_\phi/G_\phi=\hat G_\phi(1)/G_\phi(1)$. Any $t\in
\bZ/d(\phi,n)$ defines a $G_\phi$ coset $\hat G_{\phi,t}\eqdef G_\phi \bar t$, $\bar t\in
\hat G_\phi$ with $\bar t\mapsto t$.   

\subsubsection{Using equations} 
If $\phi$ is exceptional,  then \eql{rh}{rha}  visually  gives  
 $E_\phi(\bF_{q^t})$ for any integer $t$. List the
irreducible $\bF_q$ components  of $X^2_Y\setminus \Delta$ as $\row V
u$.   

\begin{cor} \label{geomExcept}  Exceptionality of $\phi$ holds if and only if each $V_i$
breaks into
$s_i$ components, conjugate over $\bF_q$, with $s_i>1$, $i=1,\dots, u$, over $\bar \bF_q$.
Denote 
  $\text{\rm lcm}(\row s u)$ by $d(\phi,2)$. Restrict elements of $\hat G_\phi$ to
$\bF_{q^{d(\phi,2)}}\subset \bF_{q^{d(\phi)}}$ to induce $\hat G_\phi(1)/G_\phi(1)\to
\bZ/d(\phi,2)$.  Then,
$E_\phi(\bF_q)$ is the union of  $t\in \bZ/d(\phi,2)$ 
not divisible by $s_i$ for any $1\le i\le u$. So, all 
$t\in (\bZ/d(\phi,2))^*$ (or in  $(\bZ/d(\phi,n))^*$) are in $E_\phi(\bF_q)$. 
\begin{triv} \label{galExc} A 
$t\in \bZ/d(\phi,n)$ is in $E_\phi(\bF_q)$ precisely when each $g\in \hat
G_{\phi,t}$ fixes (at least, or at most) one integer from $\{1,2,\dots,n\}$. \end{triv}
\end{cor}

\subsubsection{Rational points on fiber products} \label{fibProdpts} 
Let $\phi_i:X_i\to Y$, $i=1,2$, be two $K$ covers of (normal) curves. Consider the fiber product  
$X=X_1\times_YX_2$. Any $x\in X(\bF_{q^t})$
has image  
\begin{triv} \label{fibpts} $x_i\in X_i(\bF_{q^t})$, $i=1,2$, with  
$\phi_1(x_1)=\phi_2(x_2)$. \end{triv}

Conversely, if at least one $x_i$ doesn't ramify over $\phi_i(x_i)$, then  
 $x=(x_1,x_2)$ is the unique $\bF_{q^t}$ point over $x_i$, $i=1,2$. We now stress 
a point from Princ.~\ref{liftPrinc}.  

Assume $(\phi_1,\phi_2)$ is a Davenport
pair (DP) of curve covers and
$t\in E_{(\phi_1,\phi_2)}$. Then there is $x\in X(\bF_{q^t})$ lying over both
$x_i$ satisfying \eqref{fibpts},   even if both points ramify over the base. 
When $(\phi_1,\phi_2)$  is not a DP, the following is archetypal for counterexamples to
there being  
$x\in X(\bF_{q^t})$ when both the $x_i\,$s tamely ramify over
the base. Technically this example is a DP (two polynomial covers linearly related over $\bar
\bF_q$, but not over $\bF_q$), though not for the
$t$ we are considering. 

\begin{exmp} \label{DP-notDP} Assume $a\in \bF_q^*$ is not an $n$-power from $\bF_q$. Let
$f_1:\prP^1_{x_1}\to \prP^1_z$ map by $x_1\mapsto x_1^n$ and $f_2:\prP^1_{x_2}\to \prP^1_z$
map by $x_2\mapsto ax_2^n$. Then, the fiber product $\prP^1_{x_1}\times_{\prP^1_z}
\prP^1_{x_2}$ has no absolutely irreducible $\bF_q$ components, and so no
$\bF_q$ rational points. Still, $x_i=0$ maps to $z=0$, $i=1,2$.  It is {\sl much\/} 
more difficult to analyze this phenomenon if the ramification is wild. \end{exmp}  

\begin{rem} According to Cor.~\ref{geomExcept}, exceptionality depends only on 
group data. Let $\hat H\le \hat G_\phi$, $H=\hat H\cap G_\phi$ and $\hat H(1)=\hat
H\cap \hat G_\phi(1)$. Let $D_{\hat H}$ be the image of $\hat H(1)/H(1)$ in
$\bZ/d(\phi,2)$. Call the subgroup $\hat H$ {\sl exceptional\/} if
$H$ is transitive, and if no $s_i$ divides the order  of $D_{\hat H}$,
$i=1,\dots, u$. 
\end{rem}

\section{Pr-exceptional covers}  
\S\ref{exceptSetPR} interprets pr-exceptionality. Then, \S\ref{davPairs} relates it to DPs. 
Let $\phi:X\to Y$ be any $K=\bF_q$ cover. Though
$X$ may have several
$K$ components (some not absolutely irreducible), for each there is a Galois closure,
and a corresponding permutation representation. Together these components give a Galois closure group 
$\hat G_\phi=G(\hat X/Y)$ and a permutation representation: The direct sum of those coming from
each of the components. That is, the group acts on a set of cardinality $n=\deg(\phi)$, with orbits $\row O u$ of
respective cardinalities $(\row n u)$,  corresponding to the different
$\bF_q$ components $X_i$ of $X$. 

Denote restriction of $\phi$ to $X_i$ by $\phi_i$. The quotient $\hat G_\phi/G_\phi$ is isomorphic to $\bZ/d(\phi)$.
For each $i$ we have  $\hat G_\phi\to
\hat G_{\phi_i}$  defining a surjection $\bZ/d(\phi)\to \bZ/d(\phi_i,k_i)$, $1\le k_i\le n_i-1$, analogous to when
$X$ has one component.

\subsection{Exceptionality set for pr-exceptional covers} \label{exceptSetPR}  Use Def.~\ref{prexcDef} for
pr-exceptional covers. Comments on the proof of Princ.~\ref{liftPrinc} are handy for  
checking pr-exceptionality by going to a large $t$ and using properties on
fiber products off the discriminant locus. Call this the {\sl
a(void)-ram argument}. 
  
\subsubsection{Lifting rational points} 
The following variant on \eqref{galExc} defines $E_\phi(\bF_q)$ for $\phi$ pr-exceptional. The difference 
is removal of the phrase \lq\lq for at most one integer.\rq\rq\
\begin{triv} \label{galprExc} A 
$t\in \bZ/d(\phi,n)$ is in $E_\phi(\bF_q)$ precisely when each $g\in \hat
G_{\phi,t}$ fixes at least one integer from $\{1,2,\dots,n\}$. \end{triv}

\begin{princ}[Lift Principle] \label{liftPrinc} Suppose  $\phi:X\to Y$ is pr-exceptional and $Y'$ is chosen so
$\phi_{Y'}$ is a cover of nonsingular varieties. Then those 
$t$ with $\phi_{Y'}$ surjective on $\bF_{q^t}$ points is 
$E_\phi(\bF_q)$ union with a finite set. 
 \end{princ} 

\begin{proof}[Comments] \cite[Rem.~3.2 and 3.5]{AFH} discuss the literature and give a short formal proof for the
exceptional case. We extend that here to pr-exceptional. 

Assume $\phi: X\to Y$ is pr-exceptional over $\bF_{q^t}$. Let $Y^0$ be $Y$ minus the
discriminant locus of
$\phi$, and $X^0$ the pullback by $\phi$ of $Y^0$.  \cite[Rem.~3.9]{AFH} extends
in generality, with only notational change, the short proof of \cite[Lem.~19.27]{FrJ}  for DPs of
polynomials. This proof shows the equivalence of 
$\phi:X\to Y$ pr-exceptional over
$\bF_{q^t}$ (without assuming $X$ is absolutely irreducible)  with the following Galois
theoretic statement. 
\begin{triv} \label{GA-pr} Each  $g\in \hat
G_{\phi,t}$ fixes at least one element of $\{1,\dots,n\}$. \end{triv} Another way to say this: If each $g\in \hat
G_{\phi,t}$ fixes an integer in
$\{1,\dots,n\}$, not only is 
$\phi: X^0(\bF_{q^t})\to Y^0(\bF_{q^t})$ surjective,  so is $\phi: X(\bF_{q^t})\to
Y(\bF_{q^t})$. 

In the references cited above, everything was said for curves. \cite[Thm.~1]{FrThmMac} has the result for exceptional
covers $f:X\to Y$ where $X$ and $Y$ are copies of affine $n$-space (allowing ramification, of course), so $f$ is a
generalized polynomial map. The argument is much the same. It starts with $F_{y_0}$ in
the Galois group over $y_0\in Y(\bF_{q^t})$ that acts like the Frobenius on the residue
class field of a geometric point on the Galois closure over $y_0$. 
This argument only depends on the local analytic completion around $y_0$. So, it extends to 
any $f$  that (analytically) is a map of affine spaces. That is what we get for any 
$\phi_{Y'}$ with $Y'\subset Y_\phi^\ns$ (Def.~\ref{excDef}). 
\end{proof} 

\begin{rem} \label{subCov} The notation $E_\phi(\bF_q)$ for $\phi: X\to Y$ may be
insufficiently general for all pr-exceptional covers. Restricting $\phi$ to
a proper union $X'$ of 
$\bF_q$ components of $X$, to give $\phi':X'\to Y$, may also be pr-exceptional. Then,   
$E_{\phi'}(\bF_q)$ may be a proper subset of $E_\phi(\bF_q)$ and we call $\phi'$ a
pr-exceptional subcover of $\phi$. \end{rem}

\begin{prob}[A MacCluer-like Problem] Prop.~\ref{liftPrinc} goes through the domain of an extensive generalization
of MacCluer's Thm.~\cite{MacCluer}. When can we assert $\phi:X(\bF_{q^t})\to Y(\bF_{q^t})$ is one-one for $t\in
E_\phi(\bF_q)$, not just one-one over $Y_\phi^\ns$? 
\end{prob} 

\subsubsection{Pr-exceptionality vs exceptionality} \label{prvsex} If $\phi:X\to Y$
is pr-exceptional, then $E_\phi(\bF_q)$ in Princ.~\ref{liftPrinc} is the {\sl exceptional set\/} of $\phi$. From
comments of Princ.~\ref{liftPrinc}, when $\phi$ is exceptional we know 
each  $g\in \hat G_{\phi,t}$ fixes exactly one integer in
$\{1,\dots,n\}$. In fact, we have a characterization of the subset of those $t\in E_\phi(\bF_q)$ for which a
pr-exceptional cover acts like an exceptional cover:  $t$
with this property. 
\begin{triv} \label{practsexc} $X\ot
\bF_{q^t}$ has one absolutely irreducible
$\bF_{q^t}$ component $X'$, and restricting  $\phi$ to $X'$ gives an exceptional
cover over
$\bF_{q^t}$. \end{triv}

If $\phi$ is exceptional, then $1\in E_\phi(\bF_q)$. Ex.~\ref{DP-notDP} has a pair of covers
that is a Davenport pair, though its exceptionality set does not contain 1. Here the fiber product from this
Davenport pair produces a pr-exceptional cover
$\phi:X\to Y$ with $X$ containing no absolutely irreducible factor over $\bF_q$.  

\subsubsection{Pr-exceptional correspondences} \label{excCor} Suppose
$W$ is a subset of
$X_1\times X_2$  with the projections
$\pr_i: W\to X_i$ finite maps, $i=1,2$. Call $W$ a {\sl pr-exceptional\/} correspondence (over
$\bF_q$) if both $\pr_i\,$s are pr-exceptional.  We get nontrivial examples of pr-exceptional correspondences
that are not exceptional from
\eqref{davProp}: The fiber product from  a Davenport pair $(\phi_1,\phi_2)$ is a pr-exceptional correspondence.
Denote the   
 exceptionality set  defined by $X_1\times_Y X_2\mapright{\pr_i} X_i$, by
$E_{\phi_i}$, $i=1,2$   (\S\ref{morepr-exc}). In the Davenport pair case, $E_{\phi_1}\cap E_{\phi_2}$ is 
nonempty (as in Cor.~\ref{DavfromPr}).

If $W$ is absolutely irreducible  both 
$\pr_i\,$s are exceptional covers: $W$ is an {\sl exceptional correspondence}.
\S\ref{projForm} allows forming a common exceptional subtower
$\sT_{X_1,X_2,\bF_q}$ of both $\sT_{X_1,\bF_q}$ and of
$\sT_{X_1,\bF_q}$ consisting of the exceptional correspondences between $X_1$ and $X_2$. The exceptional
set for the correspondence is then
$E_{\pr_1}\cap E_{\pr_2}$. We
don't assume both $X_i\,$s have an exceptional cover to some particular $Y$.  

\begin{princ} \label{excCorRat} An exceptional correspondence between 
$X_1$ and $X_2$ implies $|X_1(\bF_{q^t})|= |X_2(\bF_{q^t})|$ for $\infty$-ly many $t$.  \end{princ}
Classical cryptology includes  $X_i=\prP^1_{z_i}$, $i=1,2$.  

Suppose $\phi_i: \prP^1_{z_i} \to \prP^1_z$, $i=1,2$, is exceptional. Then 
$\prP^1_{z_1} \times_{\prP^1_z} \prP^1_{z_2}$ has a unique
absolutely irreducible component,  which is an exceptional cover of  $\prP^1_{z_i}$,
$i=1,2$ (Prop.~\ref{excFibProp}).  So, \S\ref{SchurTow} produces a zoo of exceptional
correspondences between $\prP^1_{z_1}$ and $\prP^1_{z_2}$ (of arbitrary high genus). 

\subsection{Davenport pairs give pr-exceptional correspondences} \label{davPairs}
Suppose $\phi_i:X_i\to Y$, $i=1,2$, are (absolutely irreducible) covers.  The 
minimal ($\bF_q$) Galois closure $\hat X$  of both is any $\bF_q$ 
component of $\hat X_1\times_Y \hat X_2$ (\S\ref{galFiber}). The attached group $\hat G=\hat 
G_{(\phi_1,\phi_2)}=G(\hat X/Y)$ is the fiber product of $G(\hat X_1/Y)$ and $G(\hat 
X_2/Y)$ over the maximal $H$ through which they both factor. 
Its absolute version is $G=G_{(\phi_1,\phi_2)}$. 

\subsubsection{DPs and pr-exceptionality}  
Both $G$ and $\hat G$  have 
permutation representations, $T_1$ and $T_2$ coming from those of $G(\hat
X_i/Y)$, $i=1,2$. This induces the tensor product $T_1\otimes T_2$ of $T_1$ and $T_2$, a permutation representation
on $\hat G$.  The cyclic group 
$$\hat  G_{(\phi_1,\phi_2)}/ G_{(\phi_1,\phi_2)}=G(\hat 
\bF_{q,(\phi_1,\phi_2)}/\bF_q) $$ is $\bZ/d$:  
$d=d(\phi_1,\phi_2)$ is the {\sl extensions of 
constants\/} degree. For $t\in \bZ/d$, denote the $G_{(\phi_1,\phi_2)}$ coset mapping to 
$t$ by $\hat G_{(\phi_1,\phi_2),t}$. 

We modify Def.~\ref{excDef} to 
define  a {\sl Davenport pair\/} (DP). Assume $Y'$ is a Zariski open $K$ subset of $Y$
so $(\phi_1,\phi_2): X_1\times_YX_2\to Y$ restricts over $Y'$ to a cover of nonsingular algebraic sets ($Y'\subset
Y_{(\phi_1,\phi_2)}^\ns$; see Rem.~\ref{fiberNonSing}). 

\begin{defn} \label{davDef} Then, $(\phi_1,\phi_2)$ is a DP if we get equality of the ranges of $\phi_{i,Y'}$  on
$\bF_{q^t}$ points, $i=1,2$, for $\infty$-ly many $t$. \end{defn}  We show equivalence of these
conditions:  
\begin{edesc} \label{eqDP} \item \label{eqDPb}  $X_1\times_Y X_2
\mapright{\pr_{X_i}} X_i$, is pr-exceptional, and the exceptionality sets  $E_{\pr_i}(\bF_q)$, $i=1,2$, have
nonempty (so infinite) intersection 
$$E_{\pr_1}(\bF_q)\cap E_{\pr_2}(\bF_q)\eqdef E_{\phi_1,\phi_2}(\bF_q); \text{ and}$$  
\item \label{eqDPa} 
$(\phi_1,\phi_2)$ is a DP (independent of the choice of $Y'$).   
\end{edesc} 

The following is a corollary of Princ.~\ref{liftPrinc}. Again let $Y'$ as above be given, and denote its pullback
to $X_1\times_YX_2$ by $(\phi_1,\phi_2)^{-1}(Y')$, etc. 
\begin{cor} \label{DavfromPr} Either property of \eqref{eqDP} holds for  $(\phi_1,\phi_2)$ if and only if the other
holds. If
\eqref{eqDP}, then,
$t\in E_{(\phi_1,\phi_2)}(\bF_q)$ and  $x_i\in \phi_i^{-1}(Y')(\bF_{q^t})$, $i=1,2$,  with
$\phi_1(x_1)\!=\!\phi_2(x_2)$ implies there is $\!x\in (\phi_1,\phi_2)^{-1}(Y')(\bF_{q^t})$ with $\pr_i(x)=x_i$,
$i=1,2$: 
\begin{equation} \label{rangeEqual} \phi_1(\phi_1^{-1}(Y')(\bF_{q^t}))= \phi_2(\phi_2^{-1}(Y')(\bF_{q^t})). 
\end{equation} The set of $t$ for which \eqref{rangeEqual} holds is $E_{\phi_1,\phi_2}(\bF_q)$ union a finite
set. 

\noindent Further, both conditions of \eqref{eqDP} are equivalent to there being  $t_0\in \bZ/d(\phi_1,\phi_2)$ 
so
\begin{triv} \label{davProp}   $\tr(T_1(g))> 0$ if and only if $\tr(T_2(g))>0$ for all $g\in
\hat G_{(\phi_1,\phi_2),t_0}$. 
\end{triv} \end{cor}  

\begin{proof} Condition \eqref{davProp} says $T_1\ot T_2(g_1,g_2)=T_1(g_1)T_2(g_2)>0$
if and only if $T_i(g_i)>0$ (either $i$). This is exactly pr-exceptionality for $X_1\times_Y X_2\to
X_i$. It is also exactly the Davenport pair condition as in \cite[Thm.~3.8]{AFH}. So, this is equivalent to both
conditions of \eqref{eqDP}. 

For the range equality of \eqref{rangeEqual}, with $x_1\in \phi_1^{-1}(Y')(\bF_{q^t})$ apply pr-exceptionality
to get $x\in (\phi_1,\phi_2)^{-1}(Y')(\bF_{q^t})$ over it and let $\pr_2(x)=x_2$ to get $\phi_2(x_2)=\phi_1(x_1)$.
So, $\phi_1(x_1)$ is in the range of $\phi_2$ on $\bF_{q^t}$ points, etc. 
\end{proof} 

Each DP $(\phi_1,\phi_2)$ has an 
exceptional set: 
$$E_{(\phi_1,\phi_2)}(\bF_q)\eqdef \{t \!\mod  d(\phi_1,\phi_2) \text{ with }
\eqref{davProp}\}.$$
Multiplying by $(\bZ/d(\phi_1,\phi_2))^*$ preserves $E_{(\phi_1,\phi_2)}(\bF_q)$. Call 
$(\phi_1,\phi_2)$ a {\sl strong\/} Davenport pair (SDP)  if
\eqref{davProp} holds for all $t_0\in \bZ/d$.  

\begin{rem} Suppose $\phi: X\to Y$ is pr-exceptional. If we knew the
exceptionality set
$E_\phi(\bF_q)$ always contained 1, then the condition $E_{\pr_1}\cap E_{\pr_2}$ nonempty in
\eql{eqDP}{eqDPb} would be
unnecessary. \end{rem} 

\begin{rem}[Nonsingularity of a fiber product] \label{fiberNonSing} A Davenport pair,
given $\phi_i:X_i\to Y$, $i=1,2$, uses those $Y'$ with $(\phi_1,\phi_2)$ over $Y'$ a map of
nonsingular algebraic sets. The union of any two such $Y'\,$s is such a set. For such $Y'$, both
$\phi_i\,$s restrict over
$Y'$ to be maps of nonsingular algebraic sets. Sometimes, however, the converse may not hold. Let $S$ be the
intersection of the ramification loci of $\phi_1$ and $\phi_2$ minus common components. We can assume  $Y'$
contains the complement of $S$. 
\end{rem}
 
\newcommand{\iE}{\text{\rm i-}E} 
\subsubsection{Interpreting isovalent DPs using pr-exceptionality} \label{iDPint} 
Let $\phi_i:X_i\to Y$, $i=1,2$, be a pair of
$\bF_q$ covers. Call $(\phi_1,\phi_2)$ an {\sl isovalent DP\/} (iDP) if the equivalent properties of
\eqref{idavProp} hold. Then, $j=1$ in \eql{idavProp}{idavPropa} is just
the DP condition (in \eqref{davProp}). 

Denote the fiber product
$j$ times (minus the fat diagonal) of $X_i$ over $Y$ by $X_{i,Y}^j\setminus \Delta$. Use 
notation around \eqref{davProp}. We 
(necessarily) extend the meaning of pr-exceptional: Even the target may not
be absolutely irreducible. We also limit the $Y'\,$s used in Def.~\ref{davDef}. Use  only those for
which $\hat X\to Y$, the smallest Galois closure of both $X_i\to Y$, $i=1,2$, restricts to a cover of  nonsingular
varieties over  $Y'$. Notation compatible with Def.~\ref{excDef} would have $Y'\subset Y_{\hat \phi}^\ns$.  

\begin{prop} \label{iDPconsts} For any $t\in \bZ/d(\phi_1,\phi_2)$, the following are equivalent. 
\begin{edesc} \label{idavProp} \item \label{idavPropa}  For each $1\le j\le n-1$,
$X_{1,Y}^j\setminus 
\Delta \times_Y X_{2,Y}^j\setminus \Delta$ is a pr-exceptional cover of both  
$X_{1,Y}^j\setminus \Delta$ and $X_{2,Y}^j\setminus
\Delta$ and $t$ is in the intersection of the common exceptionality sets, over all
$j$ and projections to both factors.
\item \label{idavPropb} For an allowable choice of $Y'$, $t'$ representing $t$ and any $y\in Y'(\bF_{q^{t'}})$, there
is a range equality with multiplicity:   
$$|\phi_1^{-1}(y)\cap X_1(\bF_{q^{t'}})|=|\phi_2^{-1}(y)\cap X_2(\bF_{q^{t'}})|.$$  
\item \label{idavPropc}  
$\tr(T_1(g))=\tr(T_2(g))$ for all $g\in \hat G_{(\phi_1,\phi_2),t}$. 
\end{edesc}
\end{prop}  

\begin{proof} From the a-ram argument \eql{idavProp}{idavPropa}: $y\in
Y(\bF_{q^t})$ (not in the discriminant locus of $\phi_1$ or $\phi_2$) being the image of
$j$ distinct points of $X_i(\bF_{q^t})$ holds for  $i=1$ if
and only if it holds for $i=2$.  Running over all $j$, that says $y$ is achieved with the
same multiplicity in each fiber. The a-ram argument permits  $t$ large. So, 
the non-regular Chebotarev analog \cite[Cor.~5.11]{FrJ} has this equivalent to
\eql{idavProp}{idavPropc}. 
 \end{proof}

\begin{defn} Denote those $t$ giving the iDP 
property \eqref{idavProp} by $\iE_{(\phi_1,\phi_2)}$. \end{defn}

Prop.~\ref{DP-prExp} generalizes  
\cite[Thm.~4.8]{AFH}.

\begin{lem} \label{LemGind} Suppose $G$ and $\hat G$ are groups with  $G\norm \hat G$. Let 
$T_i$ be a faithful permutation  represention of $\hat G$ induced from the identity
representation on
$H_i\le G$, $i=1,2$. Suppose $\chi_{T_1}=\chi_{T_2}$ upon restriction to $G$. Then,
$\chi_{T_1}=\chi_{T_2}$ on
$\hat G$. \end{lem} 

\begin{proof} Since $T_i=\ind_{G}^{\hat G}(\ind_{H_i}^G({\pmb 1}))$, equality of the
inner  term representations for $i=1$ and 2 implies equality of the representations $T_1$ and $T_2$.
\end{proof}

\begin{prop} \label{DP-prExp} If $(\phi_1,\phi_2)$ is an iDP, then   $0\in
E_{(\phi_1,\phi_2)}(\bF_q)$ if and only if $(\phi_1,\phi_2)$ is an isovalent SDP:
i-$E_{(\phi_1,\phi_2)}(\bF_q)=\bZ/d(\phi_1,\phi_2)$. 

Assume now $(\phi_1,\phi_2)$ is a DP and for some $t\in
E_{(\phi_1,\phi_2)}(\bF_q)$, $X_1\times_Y X_2$ has a unique absolutely irreducible 
$\bF_{q^t}$ component $Z$.  Then, both
$X_i\to Y$,
$i=1,2$, are $\bF_{q^t}$ exceptional. If this holds for some $t\in E_{(\phi_1,\phi_2)}$,
then
$1\in E_{(\phi_1,\phi_2)}(\bF_q)$. 
\end{prop}

\begin{proof} The first statement is from Lem.~\ref{LemGind} using 
characterization \eql{idavProp}{idavPropc}.  Now consider the second paragraph statement and for simplicity assume
we have already restricted to where $(\phi_1,\phi_2)$ is a map of nonsingular spaces. 

For such a
$t$, restricting  to $Z\to X_i$ is a pr-exceptional cover (Cor.~\ref{DavfromPr}) since the only
$\bF_{q^t}$ points on $X_1\times_Y X_2$ must be on $Z$. As $Z$ is
absolutely irreducible, Prop.~\ref{RHLem} says
$Z\to X_i$,
$i=1,2$, is exceptional. To see that $\phi_i$ is exceptional, again from Prop.~\ref{RHLem} we have only to show it is
one-one. Using the a-ram argument, it suffices to do this over the nonbranch locus of both maps. Suppose
$x_1,x_1'\in X_1(\bF_{q^t})$ and $\phi_1(x_1)=\phi_1(x_1')=z$. Since this a DP, there is $x_2\in X_2(\bF_{q^t})$
lying over $z$. In, however, the fiber product, the points $(x_1,x_2),(x_1',x_2)\in Z$ both lie over $x_2$. This
contradicts that $Z\to X_2$ is exceptional.     

Any absolutely irreducible component of $X_1\times_Y X_2$
over $\bF_q$ is an absolutely irreducible component over
$\bF_{q^t}$ for every $t$. Suppose, however, $X_1\times_Y X_2$
has no absolutely irreducible component over $\bF_q$. Then, over the
algebraic closure, components fall into conjugate orbits (of length at least 2). 
Any definition field for one component in this orbit is a definition field for all 
components in this orbit. 

So, if for some $t$ there is a unique  absolutely irreducible 
component, then this holds for $t=1$. Conclude: Exceptionality for $t$ implies 
$X_1\times_Y X_2$ has a unique absolutely irreducible $\bF_q$ component. The exceptionality set of an exceptional
cover always contains 1 (for example, Prop.~\ref{excFibProp}), giving $1\in
E_{(\phi_1,\phi_2)}$.    
\end{proof}

\subsection{DPS and the genus 0 problem} \label{davPairProb} It is easy to
form new DPs (resp.~iDPs if $(\phi_1,\phi_2)$ is an iDP). Compose
$\phi_i$ with
$\psi_i: X_i'\to X_i$, with $\psi_i$ exceptional, $i=1,2$, with 
$E_{\psi_1}\cap E_{\psi_2}\cap E_{(\phi_1,\phi_2)}\not = \emptyset$. Then, $(\phi_1\circ\psi_1,\phi_2\circ\psi_2)$
is a DP (resp.~iDP). 

This subsection shows how we got explicit production of  iSDPs (that are not exceptional) from  our
knowledge of  iSDPS that exist over number fields. I mean this as a practicum on the value of the genus
0 problem. 

\subsubsection{Exceptional correspondences and DPs} \label{morepr-exc}

Prop.~\ref{DP-prExp}  characterizes 
DPs in which both maps are exceptional: Those with $X_1\times_Y X_2$
having  precisely one $\bF_q$ absolutely irreducible component $Z$. Then,  $Z\to X_i$, is exceptional,
$i=1,2$. 

Assume  
$\phi_i: X_i\to
Y$,
$i=1,2$, over $\bF_q$ is any pair of covers and $Z$ any 
correspondence between $X_1$ and $X_2$ (with the natural projections both covers).
We say $Z$ {\sl respects\/} $(\phi_1,\phi_2)$ if
$\phi_1\circ \pr_1=\phi_2\circ \pr_2$. Lem.~\ref{lookAtFP} says components of
$X_1\times_Y X_2$ suffice when seeking pr-exceptional
correspondences that respect $(\phi_1,\phi_2)$. 

\begin{lem} \label{lookAtFP} Let $Z$ be a pr-exceptional correspondence between
$X_1$ and
$X_2$ with $E_{\pr_1}\cap E_{\pr_2}=E$ nonempty. If $Z$ respects $(\phi_1,\phi_2)$, then 
$(\phi_1,\phi_2)$ is a DP (resp.~pair of exceptional covers) with $E=E_{\phi_1,\phi_2}$. Also, the image
$Z'$ of
$Z$ in
$X_1\times_Y X_2$ is a pr-exceptional (resp.~exceptional) correspondence between $X_1$
and
$X_2$.
\end{lem}

\begin{proof} Assume $Z$ with the properties in the lemma statement and $t\in E$. Apply the a-ram
argument
\eql{idavProp}{idavPropa} and consider $x_1\in X_1(\bF_{q^t})$ off the discriminant locus. 
Pr-exceptionality gives  
$z\in Z(\bF_{q^t})$ over $X_1$, and $\pr_2(z)=x_2\in X_1(\bF_{q^t})$. Since $Z$ respects $(\phi_1,\phi_2)$,
$\phi_1(x_1)=\phi_2(x_2)$. This argument is symmetric in $\phi_1$ and $\phi_2$ and shows  $(\phi_1,\phi_2)$ is a
DP.   

Any correspondence 
respecting $(\phi_1,\phi_2)$ maps naturally to $X_1\times_Y X_2$. The above shows the image is
pr-exceptional. If $Z$ is exceptional, then its image is an absolutely irreducible variety $Z'$. Since
$Z\to X_i$ is exceptional, both the natural maps $Z\to Z'$ and $Z'\to X_i$, $i=1,2$, are
exceptional, with the same exceptionality set (Prop.~\ref{RHLem}). Now use that having one absolutely
irreducible component on $X_1\times_Y X_2$ characterizes $(\phi_1,\phi_2)$ being a pair of exceptional
covers (Prop.~\ref{DP-prExp}). \end{proof}

\subsubsection{Some history of DPs}  \label{DPHist} Polynomial pairs $(f,g)$, over a number field $K$,
with the same ranges on almost all residue class fields, were what we once called {\sl
Davenport pairs\/}. \S~\ref{BackDPs} and App.~\ref{MullerDPs} has background  on these and on characteristic $p$
DPs. Investigating Davenport pairs started with proving Schur's conjecture. \cite{AFH} used {\sl DP\/} to mean a
pair of polynomials  over
$\bF_q$ as we do in Def.~\ref{davDef}: Equal ranges on $\bF_{q^t}$ for
{\sl $\infty$-ly many $t$}. We usually include the not-linearly-related assumption \S\ref{DavProb} to exclude such
exceptionality situations as a degree one cover together with 
any   exceptional cover.   We don't expect covers in an
isovalent DP to have the same degree. Still,  we learned much from the case Davenport
started: Polynomial pairs gave the covers (totally ramified over $\infty$ and genus 0). 

When exceptional covers, possibly with $g>0$,  took on a life over a given
finite field in
\cite{FGS}, it made sense to do the same for Davenport pairs (DPs). \cite[\S5.3]{Fr-Schconf} showed that over every
finite field $\bF_q$ ($q=p^s$) there are indecomposable i-SDPs $(f,g)$ of all degrees
$n=\frac{p^{t\cdot(u+1)}-1}{p^t-1}$ running over all $u\ge 2$ and $t\ge 1$. The geometric  monodromy group in this
case is
$\PGL_{u+1}(\bF_{p^s})$. I used \cite{AbProjPol} for the construction of the polynomial $f$ (over $\bF_p$) with its
monodromy representation on points of projective space. Then, I showed existence of the polynomial $g$ from the
action on hyperplanes of the same space. Since $f$ and $g$ both wildly ramified, it was tricky to 
compute the genus of the cover of $g$ (yes, it came out 0).  \cite{ABl}  constructed $g$ more explicitly.

By contrast, \cite[Thm.~2]{FrRedPol} showed this positive
conclusion toward Davenport's problem. No indecomposable polynomial DPs could occur over
$\bQ$. This was because the occurring conjugacy classes 
$\bfC$ include a single  {\sl Singer cycle\/} preventing $\bfC$ from being a rational union (see also
\cite[\S2.3]{thompson}). Yet, reducing these
pairs modulo primes produces tame polynomial i-SDPs over many prime finite fields. Further, over number
fields there was a finite set of possible degrees (\S\ref{cryptexc}). What has this to do with the genus-0
problem?  It was the precise group theory description, using branch cycles,  that allowed us to grab appropriate
wildly ramified covers from Abhyankar's genus 0 bag in positive characteristic.   

\begin{prob}  Show these 
examples nearly give a complete classification of DPs over $\bF_q$ given by polynomials $(f,g)$ with
$f$ indecomposable and $(\deg(f),p)=1$.  \end{prob}  

\section{Exceptional towers and cryptology} \label{excTowers} Let $Y$ be a normal, absolutely irreducible variety 
over 
$\bF_q$. It need not be projective (affine $n$-space is of interest). We consider
the category 
$\sT_{Y,\bF_q}$  of exceptional covers of $Y$ over $\bF_q$. 
It has  this interpretation (Prop.~\ref{excFibProp}): 
\begin{edesc} \label{excProps} \item \label{excPropsa} there is at most one morphism between two objects; and 
\item \label{excPropsb} $\sT_{Y,\bF_q}$ has fiber
products.  \end{edesc} 

With fiber products we can consider {\sl generators\/} of subtowers (\S\ref{classExc}). \S\ref{classSubTows}
lists classical subtowers on which many are expert, because  their generators are
well-studied exceptional covers. Our formulation, however, is different than from typical
expertise. That comes clear from questions arising in going to the less known subtowers of
\S\ref{serreOpenIm}. These questions directly relate to famous problems in arithmetic geometry. 
\S\ref{excPerms} documents mathematical projects in which exceptional covers had a significant role. Finally,
\S\ref{jar-defn} reminds that even for a polynomial the  word {\sl exceptional\/} historically meant something
included but not quite the same as in our context. 

\subsection{Canonical exceptional towers} \label{canExcTow} This subsection shows $\sT_{Y,\bF_q}$ is a
projective system canonically defining a  profinite arithmetic Galois group $\hat G_{Y,\bF_q}$
with a self-normalizing permutation representation 
$T_{Y,\bF_q}$. Further, with some extra conditions, pullback allows us to use classical exceptional covers to
produce new exceptional covers on an arbitrary variety $Y$ (Prop.~\ref{pullback}). 

\subsubsection{Projective systems of marked permutation representations}
\label{crypt}  For $V$ a set,  denote the permutations of $V$ 
by $S_V$. For a permutation representation
$T:G\to S_V$ and $v\in V$, denote the subgroup of
$\{g\in G\mid  (v)T(g)=v\}$ by $G(T,v)$. Suppose
$\{(G_i,T_i)\}_{i\in I}$ is a system of groups with faithful transitive permutation
representations, $T_i: G_i\to S_{V_i}$, $i\in I$, a  partially ordered index set $I$. Assume
also: 
\begin{edesc}  \item for $i> i'$, there is a homomorphism
$\phi_{i,i'}: G_i\to G_{i'}$, with $$\phi_{i,i''}=\phi_{i',i''}\circ \phi_{i,i'},\text{ if }
i>i'>i''; \text{ and}$$  
\item there is a distinguished sequence $\{v_i\in V_i\}_{ i\in I}$ (markings). 
\end{edesc} 

\begin{defn} \label{compReps} We say $\{(G_i,T_i,v_i), 
\phi_{i,i'})\}_{i\in I}$  is a compatible system of permutation
representations if for $i> i'$, $\phi_{i,i'}$ maps $G_i(T_i,v_i)$ into $G(T_{i'},v_{i'})$.
\end{defn} 

The following is an easy addition of a permutation representation to a standard lemma on projective limits on
groups. 
\begin{lem} Suppose in Def.~\ref{compReps} the partial ordering on $I$ is a projective system. Then, there is a
limit group $G_I$ whose elements naturally act as permutation representations on projective
systems of cosets of $G(T_I,v_I)=\lim_{\infty \leftarrow i} G(T_i,v_i)$. \end{lem}  

\subsubsection{The projective system on $\sT_{Y,\bF_q}$} \label{projForm} 
We use the usual
category structure for spaces over a base. Morphisms 
$(X,\phi)\in
\sT_{Y,\bF_q}$ to
$(X',\phi')\in \sT_{Y,\bF_q}$ are morphisms 
$\psi:X\to X'$ with $\phi=\phi'\circ \psi$. 
Partially order $\sT_{Y,\bF_q}$ by $(X,\phi)> (X',\phi')$ if there
is an ($\bF_q$) morphism $\psi$ from $(X,\phi)$ to $(X',\phi')$. 

Then $\psi$ induces a
homomorphism 
$G(\hat X_{\phi}/X_{\phi})$ to $G(\hat X_{\phi'}/X_{\phi'})$, and so a canonical map from
the cosets of $G(\hat
X_{\phi}/X_{\phi})$ in $G(\hat
X_{\phi}/Y)$ to the corresponding cosets for $X'$. Note: $(X,\psi)$ is
automatically in $\sT_{X',\bF_q}$. 
Prop.~\ref{excFibProp}, a converse to the second paragraph of Prop.~\ref{DP-prExp}, shows the partial ordering
on
$\sT_{Y,\bF_q}$ is a projective system. 

The nub of forming an exceptional tower of $(Y,\bF_q)$  is that there is  a unique
minimal exceptional cover dominating any two exceptional  covers $\phi_i: X_i\to Y$,
$i=1,2$ (supporting \eql{excProps}{excPropsb}). This gives fiber products in the category $\sT_{Y,\bF_q}$. Note the
extreme dependence on
$\bF_q$. We augment the Prop.~\ref{excFibProp} proof with a pure group theory argument (Rem.~\ref{noMaps}) of
the unique map property \eql{excProps}{excPropsb}. 

Let
$I\le \bN^+$. Examples we use: $I=\{t\}$, a single integer, or $I$ a union of Frobenius
progressions (Def.~\ref{frobProg}). Denote those exceptional covers  
with $I$ in their exceptionality sets 
(\S\ref{classExc}) by $\sT_{Y,\bF_q}(I)$.  For  $y_0\in Y(\bF_{q^t})$,  let $\sT_{Y,\bF_q,y_0}(I)$ be
those exceptional covers of $\sT_{Y,\bF_q}(I)$ where $y_0$ doesn't ramify in $\phi$. 

\begin{prop} \label{excFibProp} With $\phi_i: X_i\to Y$,
$i=1,2$, exceptional over $\bF_q$, $X_1\times_Y X_2$ has a unique 
absolutely irreducible
$\bF_q$  component $X$. Call its  natural
projection 
$\phi: X\to Y$: Assigning $(X,\phi)$ to
$(\phi_1,\phi_2)$ gives a categorical fiber product in $\sT_{Y,\bF_q}$. 

In this category
there is at most one ($\bF_q$) morphism between objects 
$(X,\phi)$ and $(X^*,\phi^*)$. So, $\phi: X\to Y$ has no $\bF_q$
automorphisms, which has this interpretation: For any exceptional cover $\phi: X\to Y$, the
centralizer of
$\hat G_\phi$ in
$S_{V_\phi}$ is trivial.  

For $(X,\phi)\in \sT_{Y,\bF_q}$ denote the cosets of
$G(\hat X_\phi/X_\phi)$ in $G(\hat X_\phi/Y)=\hat G_\phi$ by $V_\phi$, the 
coset of the identity by $v_\phi$ and the representation of $\hat
G_\phi$ on these cosets  by $T_\phi: \hat G_\phi \to S_{V_\phi}$. Then, $\{(\hat
G_\phi,T_\phi,v_\phi)\}_{(X,\phi)\in \sT_{Y,\bF_q}}$ canonically defines a compatible system of permutation
representations. Denote its  limit  $(\hat G_{Y,\bF_q},T_{Y,\bF_q})$. 

For $I\le \bN^+$, $t\in I$ and $y_0\in Y(\bF_{q^t})$, there is a canonical projective sequence  $x_\phi\in
X(\bF_{q^t})$ of base points for
all  $(X,\phi)\in \sT_{Y,\bF_q,y_0}(I)$ satisfying $\phi(x_\phi)=y_0$.

Consider $E=  E_{\phi_1}(\bF_q)\cap
E_{\phi_2}(\bF_q)$. Then,
$E=E_\phi(\bF_q)$ contains a full Frobenius progression $F_{1,d}$ (\S\ref{frobProg})
for some integer $d$. 
\end{prop}

\begin{proof} Suppose $\phi':X'\to Y$ is an exceptional cover and $\hat
G_{\phi'}/G_{\phi'}=\bZ/d'$. Then, for each field disjoint from $\hat
\bF_{\phi'}$, $X'\times_Y X'$ has only the diagonal as an absolutely irreducible
component. This holds for each  $t\in (\bZ/d')^*$, $t\in E_{\phi'}$. Continuing the
notation prior to the statement of the proposition, we show 
$X_1\times_Y X_2$ has a unique absolutely irreducible $\bF_q$ component. Note: No component on it can appear
with multiplicity for that would mean the cover ramified over every point of $X_i$, rather than over a finite set.
Let $Y'$ be any open subset of $Y_{\phi_1}^\ns\cap Y_{\phi_2}^\ns$ (Def.~\ref{excDef}). 

First, consider why $X_1\times_Y X_2$ has at least one absolutely irreducible $\bF_q$
component. Suppose not.  Let
$\bF_{q^{t_0}}$ be a field containing the coefficients of equations of all absolutely irreducible components of
$X_1\times_Y X_2$. Then, over any field disjoint from $\bF_{q^{t_0}}$ (over $\bF_q$), 
$X_1\times_Y X_2$ has no absolutely irreducible components. So, over such a
field the subset $X'_{1,2}$ of it over $Y'$, being nonsingular,  has no rational points. We show this leads
to a contradiction. Let $X_i'$ be the pullback in $X_i$ of $Y'$. 

From the first paragraph above, for any integer $t$ in both  $(\bZ/d(\phi_i))^*$,
$i=1,2$, Prop.~\ref{RHLem} says $\phi_i$ is one-one and onto on 
$X_i'(\bF_{q^t})$, $i=1,2$. Since it is onto, for $t$ large, this implies
$X'_{1,2}(\bF_{q^t})$ has rational points. To get a
contradiction, take $t$ large and in $(\bZ/d')^*$. This
gives us the absolutely irreducible component $X$. Denote the pullback in it of $Y'$ by $X'$. 

Consider
$t\in E$. Use the a-ram argument of
Princ.~\ref{liftPrinc}.  Suppose two points
$x,x'\in X'(\bF_{q^t})$ go to  the same non-branch point point of $Y'$. Then they map to
distinct points, in one of
$X_1'(\bF_{q^t})$  or $X_2'(\bF_{q^t})$ (say the former), that in turn map to the same point
in
$Y'$. This is contrary to $t$ being in the  exceptional support of $\phi_1$.  This shows
$X\to Y$ is exceptional, and
$t\in E_\phi(\bF_q)=E$. 

Assume \!$X$ and \!$X^*$ are distinct absolutely irreducible $\bF_q$ components of
$X_1\!\times_Y\! X_2$. Then, for $t\in E$ (large) and $x\in X_1(\bF_{q^t})$ (off the
discriminant locus of $\phi_1$), there is
$(x,z)\in X(\bF_{q^t})$ and $(x,z^*)\in
X^*(\bF_{q^t})$. Then, $z$ and $z^*$ are two distinct points of $X_2'(\bF_{q^t})$
lying over $\phi_2(z)=\phi_1(x)$. This contradicts $X_2\to Y$ being exceptional. 

What if two different $\bF_q$ morphisms $\psi_1,\psi_2:
X\to X^*$ commute with $\phi^*$? Again $X'$ is the pullback of $Y'\subset Y_\phi^\ns$. Assume $t$ is large and in
both the $(X,\phi)$ and
$(X^*,\phi^*)$ exceptionality sets. Then there is $x\in
X'(\bF_{q^t})$ with
$\psi_1(x)\ne
\psi_2(x)$. Yet, 
$\phi(x)=\phi^*\circ\psi_1(x)=\phi^*\circ\psi_2(x)$: $\phi^*$ maps $\psi_1(x)$ and 
$\psi_2(x)$ to the same place. This  
contradicts exceptionality of
$\phi^*$ for $t$.

Rem.~\ref{self-norm} gives the equivalence of $\phi: X\to Y$ having no $\bF_q$
automorphisms and the
centralizer of
$\hat G_\phi$ statement.  

To see $E_\phi(\bF_q)$ is nonempty, consider that
$E_{\phi_i}(\bF_q)$ contains all $t\in (\bZ/d(\phi_i))^*$ for both 
$i=1,2$ (from above). Since $\bZ/d(\phi)$ maps surjectively to $\bZ/d(\phi_i)$,
$i=1,2$, any integer $t$ in $(\bZ/d(\phi))^*$ is also in $(\bZ/d(\phi_i))^*$,
$i=1,2$. So, $1\in E_\phi(\bF_q)$. 
The remainder, including existence of $(\hat G_{Y,\bF_q},T_{Y,\bF_q})$, is from previous
comments. 
\end{proof}

\begin{rem}[Self-normalizing condition] \label{self-norm} Denote the normalizer of a
subgroup
$H$ of a group
$G$ by $N_G(H)$. We say $H\le G$ is {\sl self-normalizing\/} if 
$N_G(H)=H$. We can interpret this from  
$G$ acting on cosets $V$ of $H$: $T_H: G\to S_V$. The following  equivalences are  in
\cite[Lem.~2.1]{FrHFGG} (or \cite[Chap.~3, Lem.~8.8]{FB}, for example). 

Self-normalizing is the same
as  the centralizer of $G$ in $S_V$ being trivial. Finally, suppose everything
comes from field extensions (or covers): $L/K$ is a finite separable extension, and
$\hat L$ its Galois closure, with  $G=G(\hat L/K)$ and $H=G(\hat L/L)$. Then, 
self-normalizing means $L/K$ has no automorphisms. If $T_H$ is a
primitive representation (and $G$ is not cyclic), self-normalizing is automatic. 
\end{rem}

\begin{rem}[An exceptional cover
$\phi: X\to Y$ has no
$\bF_q$ automorphisms] \label{noAutos} We can see this special case of
Prop.~\ref{excFibProp} from group  theory. An  automorphism
$\alpha$ identifies with an element in $G(\hat X/Y)\setminus G(\hat X/X)$ normalizing 
$G(\hat X/X)=\hat G(T_\phi,1)$ (Rem.~\ref{self-norm}). Consider any
$g\in
\hat G_{\phi,t}\cap G(\hat X/X)$. Then, $\alpha g\alpha^{-1}\in \hat G_{\phi,t}\cap G(\hat
X/X)$ according to this data. This, however, is a contradiction, for $(1)T_\phi(\alpha)\ne
1$. So, contrary to  
Cor.~\ref{geomExcept}, $\alpha g\alpha^{-1}$ fixes two integers in the representation.
\end{rem} 

\begin{rem}[Group theory of unique morphisms in Prop.~\ref{excFibProp}]  \label{noMaps} More general than
Rem.~\ref{noAutos}, we interpret with groups that there is at most one $\bF_q$
morphism  between $(X,\phi)$ and $(X^*,\phi^*)$. Say it this way: If $(X,\phi)> (X^*,\phi^*)$,
then
$gG(\hat X^*/X^*)g^{-1}$ contains the image of $G(\hat X/X)$ only for $g\in G(\hat X^*/X^*)$. 
 
Suppose $x\in X$ is generic, and there are two maps $\psi_i$, giving 
$\psi_i(x)=x_i^*\in X^*$, $i=1,2$. Since $\phi^*\circ\psi_i=\phi$, $K(x_i^*)$, $i=1,2$, are
conjugates. This interprets as $\hat G(T_\phi,1)$ has image in $\hat G_{\phi^*}$ contained
in both $\hat G(T_{\phi^*},1))$ and $\hat G(T_{\phi^*},2)$. For exceptional covers  the
contradiction is that $K(y,x_1^*,x_2^*)$ is not a regular extension of
$K(y)$ while $K(x)$ (supposedly containing this) is. \end{rem}

\subsubsection{Pullback} Fiber products give pullback of pr-exceptional covers, and with an
extra condition, of exceptional covers.  

\begin{prop} \label{pullback} Suppose $\psi: Y'\to Y$ is any cover of absolutely irreducible
$\bF_q$ varieties. If $\phi: X\to Y$ is pr-exceptional (over $\bF_q$), then  $\pr_{\phi,Y'}:
X\times_Y Y' \to Y'$ is pr-exceptional and $E_\phi(\bF_q)$ injects into
$E_{\pr_{\phi,Y'}}(\bF_q)$. 

Let $\sT_{Y,\bF_q,Y'}$ be those exceptional covers
$\phi: X\to Y$ in $\sT_{Y,\bF_q}$ with
$X\times_Y Y'$ absolutely irreducible. This gives a map  
$\pr_{Y'}\circ (\cdot,\psi): \sT_{Y,\bF_q,Y'}\to \sT_{Y',\bF_q}$,  $$\phi\mapsto 
\pr_{\phi,Y'}: X\times_Y Y'
\to Y', \text{ by projection on } Y'.$$   

In particular, 
$\sT_{Y,\bF_q}$ is nonempty for any variety 
$Y$.  
\end{prop} 

\begin{proof} Use the a-ram argument
of Princ.~\ref{liftPrinc} with these hypotheses. Assume $t\in
E_\phi(\bF_q)$, and yet  $\pr_{\phi,Y'}: X\times_Y Y' \to
Y'$ maps $(x_1,y'),(x_2,y')\in  
X\times_Y(\bF_{q^t})$ to $y'$. Then,
$\phi(x_i)=\psi(y')$, and since $\phi$ is
exceptional, this implies $x_1=x_2$. So, $t$ is in the exceptionality set of the
pr-exceptional cover
$\pr_{\phi,Y'}$. 

If a pr-exceptional cover is of absolutely irreducible varieties, then it is exceptional (from \eqref{practsexc}).
This gives the second paragraph statement. Now consider the problem of showing $\sT_{Y,\bF_q}$ is nonempty for any
variety 
$Y$. 

Complete $Y$ in its ambient projective space, and then normalize the result. Normalization of a
projective variety is still projective \cite[p.~400]{Mum}.  So, if we construct an exceptional cover of the result,
then restriction gives an exceptional cover of $Y$. This reduces all to the case $Y$ is projective. 
N\"other's normalization lemma now says there is a cover
$\psi: Y\to \prP^t$ with $t$ the dimension of $Y$ \cite[p.~4]{Mum}. Suppose we produce an exceptional cover
$\phi: X\to \prP^t$ whose  Galois closure has order prime to the degree of $\psi$. Then, pullback of $X$ to $Y$ will
still be irreducible.  

If $Y$ is a curve, so $t=1$, we can use one of the many exceptional $\bF_q$ covers of
$\prP^1_z$  with absolutely irreducible fiber products with
$\psi$ (the easy ones in \S\ref{introGp}, for example). For  $t>1$, \cite[\S2]{FrL} constructs many 
exceptional covers of $\prP^t$ for every $t$ by generalizing the Redyi functions and Dickson polynomials (and their
relation) to higher dimensions. The construction, based on Weil's restriction of scalars, applies 
to any exceptional cover of 
$\prP^1$ to give exceptional covers of $\prP^t$.
\end{proof}

\begin{rem} The map $E_\phi(\bF_q) \to 
E_{\pr_{\phi,Y'}}(\bF_q)$ in Prop.~\ref{pullback} may not be onto. \end{rem} 

\begin{rem}[Generalization of Prop.~\ref{pullback}] Suppose $\psi: Y'\to Y$ is any morphism of absolutely
irreducible normal varieties, not necessarily a cover or a surjection. Then, Prop.~\ref{pullback} still holds: This
is a very general situation including restriction to any normal subvariety $Y'$ of $Y$. The hard part, of
course, is figuring out when irreducibility of the pullback will hold. \end{rem}

\subsection{Subtowers and equivalences among exceptional covers}
\label{classExc} 
Suppose a collection $\sC$ of 
covers from an exceptional tower $\sT_{Y,\bF_q}$ is closed under the categorical fiber
product. We say  $\sC$ is a {\sl subtower}. We may also speak of the minimal
subtower any collection generates. The following comes from the Prop.~\ref{excFibProp}
formula  and that the fiber product of unramified covers is unramified. Again, $I\le
\bZ^+$. 

\begin{lem} The collections $\sT_{Y,\bF_q}(I)$ and $\sT_{Y,\bF_q,t_0}(I)$ ($t_0\in
Y(\bF_{q^t})$; \S\ref{excPerms}) are both subtowers of $\sT_{Y,\bF_q}$. \end{lem}

It is often useful to say  
$h,h'\in \bF_q(x)$ are {\sl $\PGL_2(\bF_q)$ (resp.~$\afA(\bF_q)$) equivalent\/} if
$h=\alpha\circ h'\circ
\alpha'$ for some $\alpha,\alpha'\in \PGL_2(\bF_q)$ (resp.~$\afA(\bF_q)$). 

Practical cryptology focuses on genus 0 exceptional curve covers: 
$\phi: X\to \prP^1_y$ is exceptional, and $X$ has genus 0. Over a finite field, $X$ is
isomorphic to $\prP^1_x$ for some variable $x$. 
Since cryptology starts with an explicit place to put data, we expect to identify such an
$x$. Yet, to give an expedient list of  all exceptional covers we often  drop that
identification, and  extend
$\PGL_2(\bF_q)$ equivalence.  

If $\row h v$ and
$\row {h'} v$ are two sequences of rational functions over
a field $K$, then $h_1\circ h_2\circ \cdots \circ h_v$ is {\sl $\PGL_2(K)$
equivalent\/} to 
$h_1'\circ h_2'\circ \cdots \circ h_v'$ if each $h_i'$ is $\PGL_2(K)$ equivalent to
$h_i$, $i=1,\dots,v$. Let $\sR_{\row n v}$ be the collection of 
compositions of $v$ exceptional rational functions of respective degrees $\row n v$.
Denote by $\sR_{\row n v}/\PGL_2(K)$ its $\PGL_2(K)$ equivalence classes.
Similarly, for affine equivalence, and spaces of polynomials using the notation $\sP_{\row n
v}/\afA(K)$. 

Any explicit composition $f$ of $v$ rational functions (with degrees $\row n
v$),  over $K$, defines its $\PGL_2(K)$ equivalence class. Still, there may be other
$\PGL_2(K)$ inequivalent compositions of $f$ into rational functions over
$K$. (If $K=\bF_q$ and $f$ is exceptional, then each composition factor will
automatically be exceptional.) 

So, rather than invariants for the rational
functions, these equivalence classes are invariants for rational functions with explicit
decompositions. Still, for any interesting composition of exceptional rational functions,
we immediately recognize the whole $\PGL_2(\bF_q)$ equivalence class.

We extend this definition further. Suppose $\phi: Y\to \prP^1_z$, with $Y$ of genus 0, has
an explicit decomposition and  $\psi: X\to Y$ is a $K$ cover. 
\begin{defn} \label{PGL2genEq} Refer to
$\phi\circ \psi: X\to\prP^1_z$ as having an explicit decomposition. Then, the
$\PGL_2(K)$ action on $\psi$ induces a $\PGL_2(K)$ action on $\phi\circ \psi$ by
composition with $\psi$ after the action. This gives the $\PGL_2(K)$ equivalence class
of $(\psi,\phi)$. \end{defn} 

\newcommand{\unr}{{\text{\rm unr}}} \newcommand{\tm}{{\text{\rm tm}}} 
Let $Y$ be an open subset of $\bar Y$, a projective curve. Consider the subtower   
$\sT_{Y,\bF_q}^{\unr}$ (resp.~ $\sT_{Y,\bF_q}^{\unr,\tm}$) of $\sT_{Y,\bF_q}$ consisting of exceptional covers
unramified over $Y$ (resp.~in $\sT_{Y,\bF_q}^{\unr}$ and whose extension to a cover of $\bar Y$ is tamely ramified). 
Prop.~\ref{pullback} shows how pullback from one curve to another allows passing exceptional covers around. Still,
it is significant to know when exceptional covers are {\sl new\/} to a particular curve. We even guess the
following. 

\begin{guess} Suppose two curves $Y$ and $Y'$ over $\bF_q$ are not isomorphic over $\bF_q$. Then, the limit groups
of  $\sT_{Y,\bF_q}^{\unr}$ and $\sT_{Y',\bF_q}^{\unr}$ (and even of $\sT_{Y,\bF_q}^{\unr,\tm}$ and
$\sT_{Y',\bF_q}^{\unr,\tm}$) are not isomorphic.
\end{guess} 

Even if we restrict to exceptional covers with affine monodromy groups, this may be true. It is
compatible with 
\cite{Tamagawa}, a topic continued in \cite{exceptTow}.     

\subsection{History behind passing messages through the $I$ subtower} \label{excPerms} 
\S\ref{enthCrypt} compares enthusiasm for cryptology with topics fitting 
the phrase {\sl scrambling data\/}.  Then, \S\ref{excScram} relates cryptology and 
exceptional correspondences. 

\subsubsection{Derangements and enthusiasm for cryptology} \label{enthCrypt} 
Many  applications model 
statistical events with card shuffling. Depending on what is a shuffle and the size of a
deck, we might expect a random scrambling (shuffling) to have a good probability to move
every card.  Combinatorics rephrases
this to another question: In a given subgroup $G\le S_n$, what is the proportion of
elements that will be a {\sl derangement\/}
(\S\ref{excScram}; \cite{DMP}). We assume elements equally likely selected (uniform
distribution). Restricting to a particular subgroup $G$ then
stipulates what is a shuffle. The hypothesis of a group just says you can invert and
compose shuffles.  

Consider this setup:  
$G\norm \hat G\le S_n$, with $\hat G$ primitive, and $\hat G/G=\lrang{\alpha}$ cyclic and
nontrivial.  Combinatorialists might ask if a good fraction of the coset $\hat
G_\alpha$ (notation of \S\ref{exCheck}) is derangements. Example: 
\cite{FulmanG} outlines progress in this guiding case (conjectured earlier by N.~Boston and 
A.~Shalev \cite{Sh}) where $\lrang{\alpha}$ is trivial, contrary to our assumption. 

\begin{prob} \label{gurFulman} Restrict to  $\hat G=G$ and $G$ is simple. Show the fraction
of derangements exceeds some nonzero constant, independent of $G$. \end{prob} 

Group theory calls $\hat G$  {\sl almost simple\/}
when $G\norm \hat G\!\le \!\Aut(G)$ with $G$  
(non-abelian)  simple. To generalize  
Prob.~\ref{gurFulman} to $\hat G_\alpha$ you must exclude 
possible exceptional covers. Alternatively, use the near derangement property
of this coset (\S\ref{excScram}).

Many agencies today use cryptology to justify applying algebra outside  
pure mathematics. To include many approaches,  
cryptologists  advertise alternative expertises, including encoding in different rings 
or higher dimensional spaces.  
Modern cryptology (or as formerly,  cryptography) connects with historical mathematics
literature. Consider this enthusiastic citation \cite[p.~279]{Lidl},  quoting from
\cite{Kahn}:
 \begin{quote} The importance of mathematics in cryptography was already recognized by the
famous algebraist A.~Adrian Albert, who said in 1941: \lq\lq It would not be an
exaggeration to state that abstract cryptography is identical with abstract
mathematics.\rq\rq\end{quote}
\cite[p.~279-282 and Chap.~6]{Lidl} emphasize that many inverse problems appear when we
consider data extraction.  

Hiding data is only one part of cryptography. The nature of the hiding techniques and
finding out what it means that they are secure is the other half. Also, there is no
escaping contingencies and serendipities from patient use of
tricks. You get more of a feeling about these when you hear the outcomes of successful code cracking. The story
of  \cite[Chap.~1]{Tuchman} shows the tremendous resources that are required for a significant payoff for code
cracking. 

Public key cryptography has been around a long time. Yet, there is a
sexy new tactic -- {\sl quantum\/} cryptography. While the inspection of data encoded in different finite fields is
at the heart of modern diophantine equations, they who know this also know about modern diophantine
equations. That doesn't include those bankers who know about cryptography. See \cite{Stix} for the quickest
and simplest look at the likelihood that RSA may soon be replaced. 

\subsubsection{Periods of exceptional scrambling} \label{excScram} As above,  $g\in S_n$ is 
a {\sl derangement\/} if it fixes no integer. We see this definition appear for   
$T:
\hat G\to S_n$, the arithmetic  ($G$ the geometric)  monodromy group of an
exceptional cover. A  whole $G$ coset of $\hat G$ consists of {\sl near\/} 
derangements. Its elements each fix precisely one of $\{1,\dots,n\}$
(Prop.~\ref{RHLem}).  This nonabelian aspect of exceptional covers raises questions on
shuffling of data embedded in finite fields. 

General cryptology starts by encoding information into a set. Our sets are finite fields. 
So, let $t$ be large enough so that the bits needed to describe elements in $\bF_{q^t}$ allow
encoding our message as one of them. Put
$I=\{t\}$. Then, we  select 
$(X,\phi) \in \sT_{Y,\bF_q}(I)$. Embed our message as $x_0\in X(\bF_{q^t})$.  We use
$\phi$ as an efficient one-one function to pass $x_0$ to  $\phi(x_0)=y_0\in Y(\bF_{q^t})$
for publication. You and everyone else who can understand \lq\lq message\rq\rq\  $x_0$ can
see
$y_0$ below it. To find out what is $x_0$, requires 
an inverting function $\phi^{-1}_t: Y(\bF_{q^t})\to X(\bF_{q^t})$.  

\!\begin{quest}[Periods] \label{invFunct} Suppose  $X$ and $Y$ are explicit copies of
$\prP^1$.   Identify them to regard $\phi$ as $\phi_t$, permuting 
$\bF_{q^t}\cup\{\infty\}$. Label the order of $\phi_t$ as $m_{\phi,t}=m_t$. Then, 
$\phi_t^{m_t-1}$ inverts  $\phi_t$. How does  
$m_{\phi,t}$ vary, for genus 0 exceptional $\phi$, as $t$ varies? \end{quest} 

Quest.~\ref{invFunct} generalizes to exceptional correspondences as in Princ. ~\ref{excCorRat}. 
We can refine Quest.~\ref{invFunct} to ask about the distribution of 
lengths of $\phi_t$ orbits  on $\bF_{q^t}\cup\{\infty\}$. In 
standard RSA they are the lengths of orbits on $\bZ/(q^t-1)$ from 
multiplication by an invertible integer. This works for all covers
in the Schur Tower (\S\ref{ordSchurT}). We don't know what to expect of genus 0 covers in
the subtowers of \S\ref{serreOpenIm}.  Similar questions make sense fixing  
$t$ fixed and varying $\phi$. See the better framed Quest.~\ref{invFunct2}.    

\subsection{$k$-exceptionality} \label{jar-defn}  
We list alternative meanings for  {\sl exceptional\/} over a
number field $K$. \S\ref{1-K} gives the most obvious from reduction
modulo primes. \S\ref{m-K} has a sequence of {\sl $k$-exceptional\/} conditions; 1-exceptional 
is that of \S\ref{1-K}. 

\subsubsection{Exceptionality defined by reduction} \label{1-K} Assume 
$\phi:X\to Y$ is a cover over a number field
$K$, with ring of integers $\sO_K$. A number theorist might define an
exceptional set
$E_\phi(K)$ to be  those primes $\eu{p}$ of  $\sO_K$ for
which
$\phi$ is exceptional $\mod {\eu p}$. That matches an unsaid use in, say, 
Schur's Conjecture (Prop.~\ref{basicSchur}) describing polynomial maps
with $E_\phi(K)$ infinite.   Regard
$E_\phi(K)$ as defined up to finite set. Then, we say $\phi$ is exceptional if $E_\phi(K)$
is infinite. 

There is a complication. Even if  
$\phi:X\to Y$ and
$\psi:Y\to Z$ are exceptional (over $K$), it may be that $\psi\circ\phi$ is not.
Similarly,  you might have two exceptional covers of  $Y$ and yet their fiber product has
no component exceptional over $Y$. Ex.~\ref{compdc} 
 and Ex.~\ref{deg4exc} produce both types of situations.   

\begin{exmp}[Compositions of Dickson and cyclic polynomials] \label{compdc} \S\ref{SchurTow}
and
\S\ref{dickson} describe all indecomposable tamely ramified exceptional polynomials.
These descriptions work over any number field. Suppose
$K=\bQ$ and $f\in \bQ[x]$ is a composition of such polynomials. (From
\cite[Thm.~1]{FrSchur}, the composition is of prime degree polynomials over $\bQ$.) We can
decide when
$f$ has an infinite exceptional set by 
knowing how primes  decompose in  a cyclotomic extension
$L/\bQ$ formed from the degrees of the composition factors. List
these as
$\row s {v_1}$ (cyclic factors) and $\rowb s {v_1\np1} {v_2}$ (Dickson factors). The
exceptional set  
$E_f(\bQ)$ is  those $p$ having 
residue degree exceeding one in each of $$L_j=\bQ(e^{2\pi i/s_j}),\ j=1,\dots, v_1
\text{ and in } L_j=\bQ(e^{2\pi i/s_j}+e^{-2\pi i/s_j}),\ j=v_1\np1,\dots, v_2.$$  
\end{exmp} 

\begin{quest} Given  $f\in \bQ[x]$, can we decide when $E_f(\bQ)$ is infinite?  
\end{quest}  

The author (as referee of \cite{Matthews}) 
showed this example to Rex Matthews, who wrote out the numerics of when 
$E_f(\bQ)$ is infinite.   Still, Matthews assumed such an
$f$ is a composition of known degree cyclic and Dickson polynomials. An effective answer
for deciding for any $f\in \bQ[x]$ if it has such a form might be
harder, requiring the technique of \cite{AGR} (see \S\ref{indecRat1}).  

A related example comes from
\cite[\S7.1]{GMS} (aided by M.~Zieve). It stands out from any of the other examples they
constructed. 
 
\begin{exmp}[Degree 4 exceptional rational functions] \label{deg4exc}  Let $K$ be a number
field, and let $E/K$ have group $A_3=\bZ/3$ (resp.~$S_3$). Then, there is a rational
function
$f_E$ over $K$ with geometric monodromy  $\bZ/2\times \bZ/2$ and arithmetic monodromy
$A_4$ (resp.~$S_4$), with extension of constants $E$. This
gives  4 genus 0 exceptional covers with neither their
compositions nor fiber products exceptional. \cite{GMS} used any 
$U/\bQ$ with group
$\bZ/3\times \bZ/3$. Each of the (4) cyclic subgroups  is the kernel of a map
$\bZ/3\times
\bZ/3\to \bZ/3$. So, each map defines a degree 3 cyclic extension $E/\bQ$.
The functions $f_E$ from these cyclic extensions of $\bQ$ have the
desired property.
\end{exmp} 

\subsubsection{Exceptionality defined by rank of subgroups} \label{m-K} Recall:  
A group's {\sl rank\/} is the minimal number of elements required to generate it. Example: Simple non-cyclic
finite groups have rank 2 (this requires the classification of finite simple groups for its proof \cite[Thm.~B]{AG}).
Denote the absolute Galois group of
$K$ by
$G_K$. Suppose
$\psigma\in (G_K)^k$. Denote the fixed field in $\bar K$ of $\lrang{\psigma}$ by $\bar K^{(\psigma)}$.

Suppose
$\phi:X\to Y$ produces the extension of constants homomorphisms 
$G\to
\hat G
\mapright{\psi} G(\hat K(2)/K)$ as in Cor.~\ref{T2excCh}.  
 Consider a
conjugacy class of  subgroups represented by
$H\le G(\hat K/K)$. 

\begin{defn} If 
restricting $T_{\phi,2}$ to $H$ has no fixed points, then we say  $\phi$ is
$H$-exceptional. Also, $\phi$ is {\sl $k$-exceptional\/} if the smallest rank of a
subgroup  $H\le \hat G_\phi/G_\phi$ with $H$-exceptionality  is $k$.
\end{defn}  

For $H=\lrang{\tau}$ having rank 1,  the Chebotarev density theorem gives a positive
density of  primes $\eu{p}$ where $\tau$ is the Frobenius in $\hat K$ for
$\eu{p}$. So, $1$-exceptional is equivalent to the definition in 
\S\ref{1-K}.  We can also apply 
\cite[Thm.~18.27]{FrJ}. This shows 1-exceptional is equivalent to $X^2_Y\setminus
\Delta$ having no rational points over $\bar K^{\sigma}$ for a positive density of
$\sigma\in G_K$. 

The analog for $k$-exceptionality  is that $k$ is the minimal integer with a
positive density of elements
${\pmb
\sigma}\in (G_K)^k$ so that $X^2_Y\setminus \Delta$ has no $\bar K^{\pmb \sigma}$ points. 

\begin{rem} All these definitions extend to replace $T_{\phi,2}$ by $T_{\phi,j}$ for
$j\ge 2$. \end{rem}

\section{The most classical subtowers of $\sT_{Y,\bF_q}$}  \label{classSubTows}   We put
some structure into particular exceptional towers. Especially, we use now classical
contributions to form interesting subtowers. The tool that allows explicitly computing the limit group for these
subtowers is branch cycles as in \S\ref{introBC} (and Nielsen classes App.~\ref{nc1}). These are  the
easiest significant cases. We are illustrating to a newcomer how to use branch cycles. 

We here describe  
 subtowers that tame polynomials \wsp essentially all exceptional
polynomials with degrees prime to the characteristic (\S\ref{wildRamExc}) \wsp generate. \S\ref{restGen0tame}
considers the majority of tame exceptional covers from rational functions not in this section. Then, there is a
finite list of sporadic genus 0 exceptional cover monodromy groups. Solving the genus 0 problem
simplified their precise description in
\cite{GMS}. That produced their possible branch cycle descriptions, placing them as Riemann surface covers. The
inverse Galois techniques of
\cite{FrHFGG} (the Branch Cycle Lemma (\S\ref{BCL}) and the Hurwitz monodromy criterion) then finished the
arithmetic job of showing they did give exceptional covers. No new technical problems happened in these
cases.   

In turn, refinements (as in 
\S\ref{nonSporadics}) of the original genus 0 problem came from exceptional polynomial and Davenport pair studies: 
\S\ref{DPHist}, 
\S\ref{wildRamExc}  and App.~\ref{MullerDPs}.  Using these preliminaries simplifies how \cite{exceptTow} continues  
this topic.  For all genus 0 covers in any exceptional tower, we may  consider
Quest.~\ref{invFunct}.

\subsection{The Schur subtower of $\sT_{\prP^1_y,\bF_q}$} \label{SchurTow}
Degrees of polynomials in this section
will always be prime to
$p=\text{char}(\bF_q)$. A reminder of $\afA(\bF_q)$ equivalent polynomials prime to $p$ is
in Lem.~\ref{basicSchur}.   For
$p\ne 2$, and $n$ odd, there is a the unique polynomial $T_n$ with the property $T_n(\frac 1
2(x+1/x))=\frac 1 2 (x^n+1/x^n)$. Note: $T_n$ maps $1,-1,\infty$ respectively to
$1,-1,\infty$.  For $u\in \bF_{q^2}^*$ and $a=u^2$, define 
  $T_{n,a}=l_u\circ T_n\circ l_{u^{-1}}$, $l_u: z\mapsto uz$. Then, $T_{n,a}$  maps  
$u,-u,\infty$ respectively to
$u,-u,\infty$.  

\begin{prop} \label{ordSchurT} Assume $n$, $n'$ and $p$ are odd. By its defining property,
$T_n$ is an odd function.  So
$T_{n,a}$ depends only on $a$ (rather than $u$) and $T_{n,a}\circ T_{n',a}=T_{n\cdot
n',a}$.   

Suppose
$h$ is a  polynomial with $\deg(h)>1$,  
$(\deg(h),p)=1$ and $h\in \sT_{\prP^1_y,\bF_q}$. Then,  $h$ is a composition of odd prime
degree polynomials $\bF_q[x]$ of one of two types:
\begin{edesc} \label{Schdesc} \item \label{Schdesca} $\afA(\bF_q)$ equivalent to $x^n$  with
$(n,q-1)=1$; or 
\item \label{Schdescb} $\afA(\bF_q)$ \!equivalent to \!$T_{n,a}$, \ \!$(n,q^2-1)=1\!$, 
$\!a\!$ representing $[a]\in \bF^*_q/(\bF^*_q)^2$. 
\end{edesc} 
Conversely,  a composition of polynomials satisfying these conditions for a given $q$ is
exceptional.  In case \eql{Schdesc}{Schdesca} (resp.~\eql{Schdesc}{Schdescb}) a functional
inverse for 
$x^n$ (resp.~$T_{n,a}$) on $\bF_q$ is $x^m$ (resp.~$T_{m,a}$) where $n\cdot m\equiv 1 \mod
q-1$ (resp.~$n\cdot m\equiv 1 \mod q^2-1$).  
\end{prop} 

\begin{proof}[Comments on the proof] Map $x$ to $-x$ in the functional equation
$$T_n({\scriptstyle\frac 1 2}(x+1/x))={\scriptstyle\frac 1 2}(x^n+1/x^n)$$ to see $T_n$ is
odd. So, $l_u\circ T_n\circ l_{u^{-1}}$ is invariant for the change $u\mapsto -u$. 
Apply both of $T_{n,a}\circ T_{n',a}$ and $T_{n\cdot
n',a}$ to the composition of $x\mapsto \frac 1
2(x+1/x)$ and $l_u$. They both give the composition of $x\mapsto \frac 1
2(x^{n\cdot n'}+1/x^{n\cdot n'})$ and $l_u$ and are therefore equal. 
  
Let 
$g_\infty=(1\,2\,\dots\,n)$, \begin{equation} \label{tcycles} \begin{array}{rl}
g_1&=(1\,n)(2\,n-1)\cdots((n-1)/2\,(n+3)/2)
\text{ and } \\ g_2&=(n\,2)(n-1\,3) \cdots ((n+3)/2\, (n+1)/2).\end{array}\end{equation} 
\cite{FrSchur} shows  an indecomposable polynomial $h\in \sT_{\prP^1_y,\bF_q}$ of degree prime to $p$ is in one of
two absolute Nielsen classes:
$\ni(\bZ/n,(1,-1))$ (1 and -1 representing conjugacy classes in $\bZ/n$) or
$\ni(D_n,\bfC_{2^2\cdot\infty})$ (with conjugacy classes represented by
$(g_1,g_2,g_\infty)$ resp.). Further, suppose we give the branch points in order. Then 
only one absolute  branch cycle class gives a cover with those branch
points: 
$(g_\infty,g_\infty^{-1})$ or $(g_1,g_2,g_\infty^{-1})$. The translation starts with group
theory using the small, significant, arithmetic observation that $h$
indecomposable over $\bF_q$ implies $h$ indecomposable over $\bar \bF_q$. This holds
because $h$ is a polynomial of degree prime to $p$. 

For doubly transitive geometric monodromy $G$ acting on $\{1,\dots,n\}$, it is
immediate that any coset $Gt$ as in Cor.~\ref{geomExcept}
has an element  fixing at least two integers. Reason:  We can assure a representative $t$ 
fixes 1. If it sends 2 to $j$, multiply $t$ by  
$g\in G(T,1)$ with  $(j)g=2$ (use double transitivity). So,  $gt$ fixes
1 and $j$. Serious group theory uses that $G$ is primitive, but not doubly transitive. 

Consider the second case. This indicates a cover $\phi:X\to
\prP^1_y$ with two finite branch points $y_1,y_2$ (and corresponding branch cycles $g_1$
and $g_2$). Further, as a set, the
collection
$\{y_1,y_2\}$ has field of definition $F$. Each $y_i$ has a unique unramified
$F$ point $x_i\in X$ over it corresponding to the length 1 disjoint cycle of $g_i$. With no loss, up
to
$\afA(\bF_q)$ equivalence, take
$y_1+y_2=0$,
$y_1=u$, $y_2=-u$, and $-y_1^2=-u\in F$. So, we produce such a cover by the
polynomial map $T_{n,a}(x)$. This has $\pm u$ as the unramified points over 
$\pm u$. Up to $\afA(F)$ equivalence, that determines $u$ as a representative of
$F^*/(F^*)^2$. 

Similarly, the first case has one finite branch point $y'$, over which is exactly 
one place. Thus, up to $\afA(F)$ equivalence $\phi: \prP^1_x\to \prP^1_y$ by $x\mapsto
ax^n$. If, however, $\phi$ is exceptional over $\bF_q$, then there exists $x_0\in \bF_q$
for which $a(x_0)^n=1$, and $a$ is an $n$th root in $\bF_q$. Again, since $\phi$ is
exceptional, there is only one $n$th root in $F$, showing the $\afA(F)$ equivalence of
$\phi$ to $x\mapsto x^n$. 

See Prop.~\ref{dickProp} for why compositions from \eqref{Schdesc} are
exceptional. 
\end{proof}

\begin{rem}[Decomposability over $\bar K$ and not over $K$] \label{firstDecompStatement} 
\cite[\S4]{FGS} analyzes  the
decomposability situation for polynomials $h$ when $(\text{char}(K),\deg(h))>1$.  A
particular example where an indecomposable $h$ over $\bF_p$ becomes decomposable over
$\bar \bF_q$ occurs (\cite[Ex.~11.5]{FGS}, due to
\cite{Mu21}) with 
$\deg{h}=21$ and $p=7$.

For rational functions, \S\ref{indecRat} gives many examples of this, in all
characteristics,  from Serre's Open Image Theorem. The geometric monodromy groups of
these rational functions has the form $(\bZ/n)^2\xs \{\pm 1\}$. \end{rem}

\subsection{The Dickson subtower}  \label{dickson} 
Here we study the subtower of exceptional covers generated by
Dickson polynomials. 

\subsubsection{Dickson polynomials} \cite[p.~8]{LMT} defines  Dickson polynomials as
$$D_{n,a}(x)=\sum_{i=0}^{[n/2]}
\frac n {n-i} {n-i \choose i } (-a)^ix^{n-2i}.$$ Most relevant is its functional property
$D_{n,a}(x+a/x)=x^n+(a/x)^n$. While
$T_{n,a}(x)$ does not equal $D_{n,a}(x)$ it is a readily related to 
it.  

\begin{prop} \label{dickProp} Assume $n$ is odd. Then, $D_{n,a}(2x)/2\eqdef D_{n,a}^*=u^{n-1}T_{n,a}(x)$. In
particular, the two polynomials are $\afA(\bF_q)$ equivalent. Both polynomials, independent
of $a\in
\bF_q^*$, give exceptional covers over
 $\bF_{q^t}$ precisely when $(n,q^{2t}-1)=1$. 
\end{prop} 

\begin{proof} Let $m_b(x)=\frac 1 2(x+b/x)$. Consider $x\mapsto \frac 1 2({x^n+(a/x)^n})$ as a
composition of two maps in two ways:
\begin{equation} \label{ramDick} x\mapsto x^n \mapsto m_{a^n}(x) =x\mapsto m_a(x) \mapsto
D_{n,a}^*(m_a(x)).\end{equation} Note: 
$x\mapsto m_b(x)$ maps the ramified points $\pm \sqrt{b}$ to $\pm \sqrt{b}$. So, the left
side of \eqref{ramDick} shows this for the composite: $\pm u \mapsto \pm
{u^n}$; over each of $\pm u^n$ there are precisely $n$ points ramified of order 2; and
there are two points with ramification orders
$n$ that map to
$\infty$. As  $x\mapsto m_a(x)$ maps $\pm u \mapsto \pm u$, ramified of order 2, and it 
maps 0 and $\infty$ to $\infty$,  $D_{n,a}(x)$ has these properties. 
There are $(n-1)/2$ points ramified of order 2 over $\pm u^n$, and $\pm u$ also lie over
these points, but as the only (respectively) unramified points. So, these determining
properties (with $\infty\mapsto \infty$) show  $u^{n-1}T_{n,a}=D_n^*(x,a)$. 

Exceptionality under the condition $(n,q^{2t}-1)=1$ is in
\cite[Thm.~3.2]{LMT}. It is exactly the proof in \cite{FrSchur}, using the equation 
$D_{n,a}(x+a/x)=x^n+(a/x)^n$ (the latter said only the case $a=1$). 
\end{proof}  

\subsubsection{Exceptional sets} We list  
exceptional sets for certain Dickson
subtowers. These easy specific subtowers are a model for harder cases like 
\S\ref{restGen0tame} and in \cite{exceptTow}. 

\begin{defn} Let $v$ be an integer and $n=p_1\cdots p_v$,  a product of (possibly not distinct) primes with
$(n,2\cdot 3\cdot p)=1$.   
Compose all degree
$\row p v$ Dickson polynomials up to    
$\afA(\bF_q)$ equivalence.  (Order and repetitions of the
primes don't matter, nor what are the $a$-values
attached to them.) We denote the subtower these generate by
$\sD_{n,q}$,  the
$n$-Dickson Tower (over $\bF_q$).   
\end{defn}
 
\begin{prop} \label{decompLaw} With $n$ as above, $\phi\in \sD_{n,q}$ has exceptional set
equal to
$E'_{n,q}\eqdef \{t\mid (n,q^{2t}-1)=1\}$. This is nonempty if and only if the order of $q
\mod p_i$ exceeds 2, $i=1,\dots,v$. 
\end{prop}

\begin{proof} Consider a composition of 
$v$ degree  $\row p v$ Dickson polynomials under
$\afA(\bF_q)$ equivalence. Use the notation of \S\ref{classExc}. \eql{Schdesc}{Schdescb}
gives a natural map
$$\psi_{\row p  v; q}: (\bF_q^*/(\bF_q^*)^2)^v\to \sP_{\row p
v}/\afA(\bF_q)$$  representing all such equivalence classes. Any point $[f]$ in the image
has the exceptionality set given in the statement of the proposition. Apply
Prop.~\ref{excFibProp} to see any element in this tower has the same exceptionality set. 

Now consider when $E'_{n,q}$ is nonempty. If $p_i$ divides $n$ and
 $q^2-1\equiv 0 \mod p_i$, then $p_i$ divides $(n,q^{2t}-1)$ for any $t$. So, assume this
does not hold for any such $p_i$. That implies $(n,q^2-1)=1$. Whatever is the order $d$ of
$q^2 \mod n$, then for $t$ prime to $d$, $t\in E'_{n,q}$. 
\end{proof} 

We leave as an exercise to 
describe the exceptional set for any composition of $v$ Dickson and Redei
functions over $\bF_q$.  

\begin{rem}[Varying $a$ in $D_{n,a}(x)$ and Redei functions] \cite[Chap.~6]{LMT}, in 
their version of the proof of the Schur conjecture, make one distinction from that of
\cite{FrSchur}. By considering the possibility $a$ is 0, they include $x^n$ as a
specialized Dickson polynomial, rather than treating them as two separate cases. 

The function $x^n$ ($n$ odd) maps
$0,\infty$ to $0,\infty$. Consider $l_u': x\mapsto \frac{x-u}{x+u}$, mapping $\pm u$ to
$0,\infty$.  A similar, but easier, game comes from 
$$\text{ twist }x^n \text{ to } R_a=(l'_u)^{-1}\circ (\frac{x-u}{x+u})^n$$ for which
$\pm u$ are the only ramified points,
$u^2=a$ and $R_a(\pm u)=\pm u$. We've pinned down $R_a$ precisely by adding the condition
$\infty \mapsto 1$. This, modeled on that for the Dickson polynomials,  
matches \cite[\S5]{LMT}.  
\end{rem} 

\subsubsection{Dickson subtower monodromy} \label{DickMon} 
Order exceptional covers in a tower as in \S\ref{crypt}. One
exceptional cover sits above them all in any finitely generated subtower 
(Prop.~\ref{excFibProp}). We call that the limit (cover). When all generating covers tamely
ramify, the limit has a branch cycle description, represented by an 
absolute Nielsen class. Using this succinctly describes the geometric monodromy
of the  limit cover. 

We use some subtowers of
$\sT_{\prP^1_z,\bF_q}$ to show how this works. Consider the subtower generated by
$\PGL_2(\bF_q)$ equivalence classes of $v$ compositions of cyclic and Dickson polynomials
over
$\bF_q$ running over all $v$. Denote this by $\text{Sc}_{\bF_q}$. We now use
Prop.~\ref{useBC} to consider branch cycles for some subtower limit covers. 

For $a\in \bF_q^*$, \eqref{tcycles} gives a branch cycle description for $T_{n,a}$.
Label letters on which these act as $\{1_a,\dots,n_a\}$, and elements
corresponding to $\eqref{tcycles}$ acting on these by $(g_{a,1},g_{a,2})$. To label the  limit cover branch cycles,
use an ordering $a_1,\dots,a_{q-1}$  of  $\bF_q^*$. For each $a_j$, let $\pm u_j$ be its square roots, these
being branch points for $T_{a_j,n}$. 

We induct  on  $1\le k\le q-1$. Assume we
have listed branch cycles 
\begin{equation} \label{nqk-1}
(g_{a_1,1},g_{a_1,2},\dots,g_{a_{k-1},1},g_{a_{k-1},2},g_{\row a
{k\nm1},\infty})\end{equation} for the limit cover generated by
$T_{n,a_1},\dots,T_{n,a_{k\nm1}}$. In the inductive fiber product construction, 
permutations act on  $V_{\row a {k\nm1}}=\{(j_{a_1},\dots, j_{a_{k\nm1}})\mid 1\le
j_{a_u}\le n\}_{u=1}^{k-1}$. Also, the following hold:
\begin{edesc} \label{bcConds} \item $g_{a_j,1},g_{a_j,2}$ are respective branch cycles
corresponding to
$\pm u_j$; \label{brnqk-1} \item entries in \eqref{nqk-1} generate a transitive
group and their  product is 1; and 
\item $g_{\row a {k\nm1},\infty}$ is a product of disjoint $n$-cycles. 
\end{edesc} 

\begin{prop} \label{branchCycleqn} For a given $n$,with $q$ odd and $(n,q^2-1)=1$,
denote the subtower of 
$\sD_{n,q}$ generated by $\{T_{n,a}\mid a\in \bF_q^*\}$ (resp.~$\{T_{n,a}+b\mid a\in
\bF_q^*, b\in \bF_q\}$) by $\sD_{n,q}'$ (resp.~$\sD_{n,q}''$). Then, the limit cover for
$\sD_{n,q}'$ has degree $q^n$ over $\prP^1_z$, and it has
unique branch cycles in the absolute Nielsen class formed inductively from the conditions
\eqref{bcConds}. Also,
$\sD_{n,q}'=\sD_{n,q}''$. 
\end{prop} 

\begin{proof} 

Denote branch cycles for $T_{n,a_k}$ by $g_{a_k,1},g_{a_k,2}$, acting on $\{1_{a_k},\dots,
n_{a_k}\}$ as in \eqref{tcycles}. Our goal is to form   
$$(g^*_{a_1,1},g^*_{a_1,2},\dots,g^*_{a_{k},1},g^*_{a_{k},2},g^*_{\row a
{k},\infty})$$ with * indicating the actions extend corresponding elements to the set
$V_{\row a {k}}$, yet satisfying the corresponding conditions to \eqref{bcConds}.
We show now how this forces a unique element up to absolute equivalence in the resulting
Nielsen class. We'll use $n=3$ (even though this never gives an exceptional cover)
and $k=2$ to help sort the notation as a subexample. First we construct one element as follows. 

In the induction, $g^*\,$s act on pairs  $(u,v)$: $u$ (resp.~$v$) from  the
permuted set of $$\lrang{g_{a_i,j}, 1\le i\le k\nm1, 1\le j\le2} \text{ (resp.~}\lrang{g_{a_k,j}, 1\le j\le 2}).$$
This is the tensor notation in
\S\ref{galFiber}.  Form the elements
$g^*_{a_j,t}$, $t\in \{1,2\},j\le k-1$, by replacing any cycle
$(u\,u')$ in $g_{a_j,t}$ by $\prod_{i_{a_k}=1}^n((u,i_{a_k})\,(u',(i_{a_k})\pi))$ with
$\pi\in S_n$. 

With $\pi=1$, list as rows orbits of the product 
$g^*_{a_1,1}\cdot g^*_{a_1,2}\cdot \dots \cdot g^*_{a_{k\nm1},1}\cdot g^*_{a_{k\nm1},2}$. Call
this row display $R_{n,k\nm1}$. Here is $R_{3,1}$,  
$n=3$, $k-1=1$: 
$$ \begin{array} {rl} &(1,1)\to (2,1)\to (3,1) \\
&(1,2)\to (2,2)\to (3,2)\\
&(1,3)\to (2,3)\to (3,3).\end{array}$$ 

Now consider the corresponding extension
$g^*_{a_k,1}, g^*_{a_k,2}$ of $g_{a_k,1}, g_{a_k,2}$  by replacing any disjoint cycle $(i\,i')$ for one of
$g_{a_k,1}, g_{a_k,2}$ with  
$\prod_{u\in V_{\row a {k\nm1}}} ((u,i)\,((u)\tau,i'))$ with $\tau$ a permutation on
$V_{\row a {k\nm1}}$.

Whatever is our choice in this last case we can read off the effect of the product of 
the $g^*$ entries by considering the orbits of this in the table $R_{n,k\nm1}$. We 
know the group generated by the $g^*\,$s is to be transitive, and all these orbits will
proceed from left to right and be of length $n$. Conclude, that up to a reordering of
the rows and a cycling of each row (it was up to us where we started the row), the orbit
path in
$R_{n,k\nm1}$ takes the shape of a stair case to the right. Example, $n=3$, $k-1=1$, the
product of the $g^*$ entries starting at $(1,1)$ would give $(1,1)\to (2,2)\to (3,3)$ as
an orbit. So, the conditions of \eqref{bcConds} determine $g^*_{a_k,1}, g^*_{a_k,2}$.  

To conclude the proof we have only to show the covers $T_{n,a}+b$ are
quotients of the limit cover for $\sD'_{n,q}$. The branch points of $\sT_{n,a}+b$ are at
$\pm u+b$ in the previous notation. We show  the cover $T_{n,a}+b$ is a
quotient of the exceptional cover fiber product of $T_{\pm(u+b)}$ and $T_{\pm(b-u)}$, the
degree $n$ Dickson polynomials with branch points at
$\pm(u+b)$ and $\pm(b-u)$ respectively. 

This fiber product has branch points at $\pm(u+b)$, $\pm(b-u)$, and $\infty$, and
branch cycles $(g_{1,1},g_{1,2},g_{2,1},g_{2,2},g_\infty)=\bg$ with branch points
$u+b,b-u$ corresponding to $g_{1,2},g_{2,1}$ at the 2nd and 3rd positions. Let $G$ be the
geometric monodromy of this fiber product, with
$T'$ and $T''$ the permutation representations from
$T_{\pm(u+b)}$ and $T_{\pm(b-u)}$. All we need is some representative in the absolute
equivalence class of this branch cycle with the shape
$(g_1',g_{1,1},g_{1,2},g_4',g_\infty)$ for some $g_1',g_4'$. Then, $T'$ applied to this gives branch
cycles for
$\sT_{n,a}+b$ (the same for $T_{\pm(u+b)}$ but with branch points at the appropriate
places). Apply the braid $q_2q_1\in H_5$ (as in \eqref{HurMon}) to  $\bg$:
$$(\bg)q_2q_1=(g_{1,1}, g_2',g_{1,2}, g_{2,2},g_\infty)q_1=(g_1',g_{1,1},g_{1,2}, g_{2,2},g_\infty)$$ with 
$g_2'=g_{1,2}g_{2,1}g_{1,2}^{-1}$ and $g_1'=g_{1,1}g_2'g_{1,1}^{-1} $. We already
know this represents the same element in the Nielsen class as $\bg$.   
\end{proof} 

\begin{prob} Use Prop.~\ref{branchCycleqn}  to describe the
limit branch cycles for $\text{Sc}_{\bF_q}$. \end{prob}

\section{Introduction to the subtowers in \cite{exceptTow}} \label{serreOpenIm} Serre's open image theorem (OIT) 
\cite{SeAbell-adic} forces a divide between two types, $\GL_2$ and
$\CM$, of  contributions to the genus 0 covers in the 
$\sT_{\prP^1_z,\bF_q}$ tower. We concentrate on the  mysterious $\GL_2$
part, limiting to topics around one serious question: Decomposition of rational functions and their
relation to exceptional covers in \S\ref{indecRat}. 

Any one elliptic curve $E$ without complex multiplication
produces a collection of $\{f_{p,E}\}_{p> c_E}$ for some constant $c_E$ with these properties. Each 
$$f_{p,E}
\mod
\ell:
\prP^1_x\to
\prP^1_y\text{ is indecomposable and  exceptional,}$$  but it decomposes over $\bar \bF_\ell$.
\S\ref{expPrExc} then considers using automorphic functions to give a useful description of primes
$\ell$ for which a given
$f_{p,E}$ has these properties. Finally,
\S\ref{wildRamExc}  sets straight a precise development about wildly ramified exceptional covers that several
sources have garbled. Using this to describe the wildly ramified part of exceptional subtowers generated by genus 0
covers continues in
\cite{exceptTow}.  

\subsection{Tame exceptional covers from modular curves} \label{restGen0tame} 
\cite[\S6.2]{luminy} will continue in \cite{exceptTow}. The former is the Modular Tower setup of Serre's OIT. 
This framework shows there are other Modular Towers whose levels  are
$j$-line covers (though not  modular curves) having cases akin to
$\GL_2$ and 
$\CM$.  

\subsubsection{Setup for indecomposability applications} 
The affine line  $\prP^1_j\setminus \{\infty\}=U_\infty$ identifies with the quotient 
$S_4\backslash(\prP^1_z)^4\setminus
\Delta/\PGL_2(\bC)$ (\S\ref{notation}).  
For $p>1$ an odd prime, and  
$K$ a number field, infinitely many $K$ points on
$U_\infty$ produce rational functions of degree $p^2$ with 
these  properties. 
\begin{edesc} \label{apOIT} \item \label{apOITa} They are indecomposable over $K$,  yet decompose over $\bar K$
(\S\ref{indecRat}).
\item \label{apOITb} Modulo almost all primes they give tamely ramified rational functions with property
\eql{apOIT}{apOITa} over finite fields. 
\item \label{apOITc} They give
exceptional covers (as in \S\ref{1-K}) with nonsolvable extension of constants group.
\end{edesc}

 Most
from the remaining genus 0, tame exceptional covers are 
related to   \eqref{apOIT} \cite[\S2]{FrGGCM}.  \cite{GMS} concentrated
more on the CM type, because there are hard problems with being explicit in the $\GL_2$ case. \S\ref{expPrExc}
gives specific examples of those problems.  Ribet's words  
\cite{ribet} from 14 years ago on \cite{SeAbell-adic}  still apply:   
\begin{quote} Since the publication of Serre's book in 1968, there have been numerous
advances in the theory of $\ell$-adic representations [of absolute Galois groups] attached
to abelian varieties [He lists Faltings' proof of the semisimplicity of the representations;
and ideas suggested by Zarhin]. \dots Despite these recent developments, the 1968 book of
Serre is hardly outmoded. \dots it's the only book on the subject [\dots and] it can be
viewed as a toolbox [of] clear and concise explanations of fundamental topics [he lists
some].
\end{quote}

\newcommand{\Q}[1]{\text{Q}^{{#1}}}

\subsubsection{Sequences of nonempty Nielsen classes} We briefly remind how \cite[\S6]{luminy} 
formulates additional examples that have OIT properties using a comparison with OIT. You can skip this 
without harm for the indecomposability applications of \S\ref{indecRat}.  Consider the following objects:
$F_2=\lrang{x_1,x_2}$, the free group on two generators;  $J_2=\bZ/2=\{\pm 1\}$ 
acting as
$x_i\mapsto x_i^{-1}$,
$i=1,2$, on $F_2$; and $P_2$, all primes different from 2. Denote the 
{\sl nontrivial\/} finite
$p$ group quotients of
$F_2$ on which $J_2$ acts, with
$p\not\in P_2$, by 
$\Q {F_2}(P_2)\eqdef Q^{F_2}(P_2,J_2)$. 

Use the notation $\bfC_{2^4}=\bfC$ for  four repetitions of the nontrivial conjugacy class
of
$J_2$.  For any $U\in \Q {F_2}(P_2,J_2)$, $\bfC$ lifts uniquely to conjugacy classes of
order 2 in
$U\xs J_2$. This defines a collection of Nielsen classes: 
$$\sN= \{\ni(G,\bfC_{2^4})^\inn\}_{\{G=U\xs J_2\mid  U\in
\Q {F_2}(P_2,J_2)\}}.$$

Suppose for some $p$, $\sG_{p,I}=\{U_i\}_{i\in I}$ is a projective subsequence of 
(distinct) $p$ groups from $\Q {F_2}(P_2)$. Form a limit group  $G_{p,I}=\lim_{\infty
\leftarrow i} U_i\xs J_2$. Assume further, all Nielsen classes
$\ni(U_i\xs J_2,\bfC)$ are nonempty. Then,  $\{\ni(U_i\xs J_2,\bfC)^\inn\}_{i\in I}$
is  a project system with a nonempty limit $\ni(G_{p,I},\bfC)$. 

\subsubsection{Achievable Nielsen classes from modular curves} 
Let $\bz=\{\row z
4\}$ be any four distinct points of $\prP^1_z$, without concern to order. As in
\S\ref{nc1},  
 choose a set of (four) classical generators for the fundamental group of
$\prP^1_z\setminus
\bz=U_\bz$. 

This group identifies with the free group on four generators $\psigma=(\row
\sigma 4)$, modulo the product-one relation $\sigma_1\sigma_2\sigma_3\sigma_4=1$. Denote
its completion with respect to all normal subgroups for which the kernel to $J_2$ is $2'$ (has order prime to 2) by
$\hat  F_\psigma$. Let $\bZ_p$ (resp.~$\hat F_{2,p}$) be the similar completion of $\bZ$
(resp.~$F_2$) by all normal subgroups with $p$ ($\not=2$) group quotient. 
The following is \cite[Prop.~6.3]{thompson}. 

\begin{prop} \label{H2NC}  Let $\hat D_\psigma$ be
the quotient of $\hat  F_\psigma$ by the relations $$\sigma_i^2=1,\   i=1,2,3,4\ (\text{so\
}\sigma_1\sigma_2=\sigma_4\sigma_3).$$ Then, $\prod_{p\ne 2}\bZ^2_p\xs J_2\equiv 
\hat D_\psigma$. Also, $\bZ^2_p\xs J_2$ is the unique $\bfC_{2^4}$ $p$-Nielsen class
limit. 
\end{prop} 

As an if and only if statement,  it 
has two parts (\S\ref{nonempNielsen2}):  A Nielsen
class from an abelian $U\in \Q {F_2}(P_2)$ (resp.~nonabelian $U$) is nonempty
(resp.~empty). 

\begin{rem}[For those more into
Nielsen classes] The major point of
\cite{thompson} starts by contrasting this $J_2$ case with with an
action of
$J_3=\bZ/3$ on $F_2$ (illustrating a general situation). The exact analog
there has all Nielsen classes nonempty \cite[Prop.~6.5]{thompson}. It also conjectures
\wsp special case of a general conjecture \wsp that each $H_4$ (\eqref{HurMon}, the group $H_r$ with $r=4$) orbit on
those limit Nielsen classes contains a Harbater-Mumford representative: Element of the form $(g_1,g_1^{-1},
g_2,g_2^{-1})$.  We know the
$H_4$ orbits precisely for the
$J_2$ case (\S\ref{nonempNielsen2}). \end{rem}

\subsubsection{Nature of the nonempty Nielsen classes in Prop.~\ref{H2NC}}
\label{nonempNielsen2} Denote an order 2 element in $G_{p^{k+1}}=(\bZ/p^{k+1})^2\xs \{\pm
1\}$ by 
$(-1;\bv)$ with $\bv\in (\bZ/p^{k+1})^2$. An explicit  $\bv$ has the form $(a,b)$,
$a,b\in \bZ/p^{k+1}$. The multiplication $(-1;\bv_1)(-1;\bv_2)$ yields $\bv_1-\bv_2$ as one would
expect from formally taking the matrix product $$\smatrix {-1} {\bv_1} 0 1  \smatrix {-1} {\bv_2} 0
1 \text{ as in \eqref{rightAct}.}$$ 

We have an explicit description of the Nielsen classes
$\ni(G_{p^{k+1}},\bfC_{2^4})$. Elements are 4-tuples $((-1;\bv_1),\dots,(-1;\bv_4))$
satisfying two conditions from \S\ref{nc1}:
\begin{edesc} \label{nconds} \item \label{ncondsa}  Product-one: $\bv_1-\bv_2+\bv_3-\bv_4$; and
\item \label{ncondsb}  Generation:  $\lrang{\bv_i-\bv_j, 1\le i< j\le 4}=(\bZ/p^{k+1})^2$.
\end{edesc} By  conjugation in $G_{p^{k+1}}$ we may assume $\bv_1=0$. Now take $\bv_2=(1,0)$, 
$\bv_3=(0,1)$ and solve for $\bv_4$ from \eql{nconds}{ncondsa}.  

Prop.~\ref{SerreRes1} explains subtleties on the inner and absolute Nielsen classes in this case.
For  
$V=V_{p^{k+1}}=(\bZ/p^{k+1})^2$, $V\xs \GL_2(\bZ/p^{k+1})$ is the normalizer of $G_{p^{k+1}}$ in $S_V$ 
 (notation of \S\ref{crypt}). Let $\ni(G,\bfC)$ be a Nielsen class (with $\bfC$
a rational union of conjugacy classes) and assume there is a permutation representation
$T: G\to S_n$. There is always a natural map
$\Psi:
\sH(G,\bfC)^{\inn}\to 
\sH(G,\bfC)^{\abs}$ (or $\Psi^\rd$) on the reduced spaces (\S\ref{nc2}). Restricted to any $\bQ$ component
of
$\sH(G,\bfC)^{\inn}$, $\Psi$ is Galois with group a subgroup of $N_{S_n}(G)/G$  
\cite[Thm.~1]{FrVMS}. For the Nielsen class from $(G_{p^{k+1}},\bfC_{2^4})$, etc.~denote this map $\Psi_{p^{k+1}}$. 

\begin{prop} \label{SerreRes1} The following properties hold for these absolute classes.
\begin{edesc} \label{absProp1} \item  \label{absProp1a} 
$|\ni(G_{p^{k+1}},\bfC_{2^4})^\abs|=1$, so 
$\sH(G_{p^{k+1}},\bfC_{2^4})^{\abs,\rd}$ identifies with $U_\infty$.
\item  \label{absProp1b} Rational functions of degree
$(p^{k+1})^2$ represent $\ni(G_{p^{k+1}},\bfC_{2^4})^\abs$ covers.  
\end{edesc} 
The following properties hold for these inner classes. 
\begin{edesc} \label{innProp1} \item  \label{innProp1a} $H_4$ has $p^{k+1}-p^k$ orbits  on
$\ni(G_{p^{k+1}},\bfC_{2^4})^\inn$. 
\item  \label{innProp1b}  $\Psi_{p^{k+1}}$ (or $\Psi_{p^{k+1}}^\rd$) is Galois with group
$\GL_2(\bZ/p^{k+1})/\{\pm1\}$.  
\item  \label{innProp1c} Fix $j'\in U_\infty(\bQ)$ without complex multiplication. Then,  excluding
a finite set
$P_{j'}$ of primes $p$, the fiber of $\Psi_{p^{k+1}}^\rd$ over $j'$ is irreducible. 
\end{edesc} 
\end{prop}

\begin{proof}[Comments on using the proposition] Use the symbol $(\row {\bv} 4)$
to denote the Nielsen element $((-1;\bv_1),\dots,(-1;\bv_4))$. Conjugating by
$\beta\in \GL_2(\bZ/p^{k+1})$ on this Nielsen element maps it to $(\beta(\bv_1),\dots,
\beta(\bv_4))$.  Conjugating by $(1,\bv)$ translates by $(\bv,\bv,\bv,\bv)$. So, now we
may take
$\bv_1={\pmb 0}$. That there is one absolute class follows from transitive action of
$\GL_2(\bZ/p^{k+1})$ on pairs
$(\bv_2,\bv_3)$, whose entries are now forced to be independent if they are to represent an element of the Nielsen
class. 

On the other hand, consider the action of the $q\,$s in $H_4$. Example: $q_2$ applied to
the symbol $(\row {\bv} 4)$ gives  $(\bv_1,2\bv_2-\bv_3,\bv_2,\bv_4)$.  So these actions
are in $\SL_2(\bZ/p^{k+1})$. Any cover in the Nielsen class has odd degree $(p^{k+1})^2$
and genus 0 as computed by Riemann-Hurwitz. Take $j'\in \bQ$ to be the $j$-invariant of the branch point
set corresponding to the cover. Conclude, there is a
rational function
$f_{j'}:\prP^1_w\to
\prP^1_z$ representing this odd degree genus 0 cover.

According to 
\cite[IV-20]{SeAbell-adic} we can say explicit things about the fibers of
$\sH(G_{p^{k+1}},\bfC_{2^4})^\inn\to
\sH(G_{p^{k+1}},\bfC_{2^4})^\abs$ over $\bp\in \sH(G_{p^{k+1}},\bfC_{2^4})^\abs$ depending on the $j$-value of
the 4 branch points for the cover $\phi_\bp: X_\bp\to \prP^1_z$ corresponding to $\bp$. 
\S\ref{OggsEx} and
\S\ref{indecRat} show our special interest in such covers over $\bQ$ with the full arithmetic
monodromy group 
$V_{p^{k+1}}\xs
\GL_2(\bZ/p^{k+1})$. 

We now note what is the cover $\phi_\bp$. Let $E$ be any elliptic curve in Weierstrass
normal form, and
$[p^{k+1}]: E \to E$ multiplication by $p^{k+1}$. Mod out by the action of $\{\pm 1\}$ on both
sides of this isogeny to get $$E/\{\pm 1\}=\prP^1_w\mapright{\phi_{p^{k+1}}} E/\{\pm
1\}=\prP^1_z,$$ a degree
$p^{2(k+1)}$ rational function. Composing  $E\to E/\{\pm 1\}$ and multiplication by
$p^{2(k+1)}$ gives the Galois closure of $\phi_{p^{k+1}}$. This is a geometric proof why
$\ni((\bZ/p^{k+1})^2\xs J_2,\bfC_{2^4})$ is nonempty. If $E$ has definition field $K$, so does
$\phi_{p^{k+1}}$. We may, however, expect the Galois closure field of $\phi_{p^{k+1}}$ to have an
interesting set of constants coming from the fields of definition of $p^{k+1}$ division points on
$E$.  
 
The geometric group is
$(\bZ/p^{k+1})^2\xs
\{\pm 1\}$ acting as permutations on $(\bZ/p^{k+1})^2$.  
This group isn't primitive because $\{\pm 1\}$ does not act irreducibly. On each side of the degree $p^2$ isogeny 
$E\mapright{[p]} E$, mod out by
$\{\pm1\}$. If $E$ has no complex multiplication but a number field as definition field, then 
for almost all primes $p$, 
\begin{triv} \label{arithMonsur} the arithmetic monodromy group is $(\bZ/p)^2\xs
\GL_2(\bZ/p)$: and for $p^{k+1}$ it is $(\bZ/p^{k+1})^2\xs
\GL_2(\bZ/p^{k+1})$. \end{triv} 
\end{proof} 

\begin{rem}[More on explicitness] The proof of \cite[IV-20]{SeAbell-adic} concludes the proof
of\eqref{arithMonsur} for non-integral (so not complex multiplication) $j$-invariant.  Serre's initial
proof of
\eql{innProp1}{innProp1c} for almost all primes for integral (not complex multiplication)
$j$-invariant relied on unpublished results of Tate. Though Falting's theorem now replaces that, it is still not
explicit. So even today, being explicit on the exceptional primes in Prop.~\ref{SerreRes1} still requires
non-integral $j$-invariant. (Note, however, comments of \S\ref{automorConn} from Serre's using modularity
of an elliptic curve.) 
\end{rem}

\subsection{Indecomposability changes from $K$ to $\bar K$} \label{indecRat} \S\ref{indecRat1} notes that 
finding the minimal field over which one may decompose 
rational functions, or any cover $\phi: X\to Y$, is a problem in identifying a specific
subfield $K_\phi(\ind)$ of $\hat K_\phi$ (\S\ref{galClosure}). For tamely ramified covers, Prop.~\ref{GL2pprop} shows
the OIT  is the main producer of rational functions $\phi=f: \prP^1_x\to \prP^1_z$ over a number field (or over a
finite field) where $K_\phi(\ind)$ will nontrivially extend the constant field. 

\subsubsection{The indecomposability field} \label{indecRat1} Two ingredients go into a test for 
indecomposability of any cover $\phi:
X\to Y$. These are a use of fiber products and a test for reducibility in
the following way. Check $X\times_Y X$ minus the diagonal for irreducible components
$Z$ which have the form $X'\times_Y X'$. If there are none, then
$\phi$ is indecomposable. Otherwise, $\phi$ factors through $X'\to Y$.

\cite[Thm.~2.3, Thm.~4.2]{FrMacRae} used the polynomial cover case of this 
when the degree was prime to $p$. As a result for that case, there is a
maximal proper variables separated factor. \cite{AGR} exploited \cite{FrMacRae} similarly for rational
functions. Denote the minimal Galois extension of $K$ over which $\phi$ decomposes into
absolutely indecomposable covers by $K_\phi(\text{ind})$: The indecomposability field of
$\phi$. Conclude the following. 

\begin{prop} \label{indecfield} For any cover $\phi: X\to Y$ over a field $K$,
$K_\phi(\ind)\subset \hat K_\phi(2)$. \end{prop} 

\subsubsection{Ogg's example} \label{OggsEx} 
\cite[IV-21-22]{SeAbell-adic} outlines computing  $\rho_{3^+,p}(G_\bQ)$, the
image of
$G_\bQ$ on the $p$-division points $E[p]$ of an elliptic curve $E$, for a case of $E$ where we can list
$p\,$s that are exceptions to \eqref{arithMonsur}. 

The curve $3^+$ of \cite{Ogg} has affine model $\{(x,y)\mid y^2+x^3+x^2+x=0\}$
with
$j$ invariant 
$2^{11}\cdot 3^{-1}$, discriminant $-2^4\cdot 3$
and conductor 24. It also has an isogeny of degree 2 to the modular curve $X_0(24)$. The nontrivial 
degree 2 isogeny shows the image $\rho_{3^+,2}(G_\bQ)$ of $G_\bQ$ is not $\GL_2(\bZ/2)$, and
the image has order 2, corresponding to the field extension $\bQ(\sqrt{-3})$. For, however,
$p\ne 2$, he shows the following. 
\begin{itemize} \item Determinant on $\rho_{3^+,p}(G_\bQ)$ has image $\bF_p^*$
(because the base is
$\bQ$). \item $\rho_{3^+,p}(G_\bQ)$ has a transvection (use Tate's form of $3^+$ for $p=3$:
$3^{1/p}\in E$ revealing the tame inertia group generator acts as a
transvection). \end{itemize} If we know $G_\bQ$ acts irreducibly for $p$, then \cite[IV-20,
Lem.~2]{SeAbell-adic} says the complete action is through $\GL_2(\bZ/p)$. All we need 
is to assure, from the irreducible action, the transvection $\smatrix 1 1 0 1$
conjugates to $\smatrix 1 0 1 1$, and these two generate
$\SL_2(\bZ/p)$. 

Serre uses Ogg's list to see that for $p\ne 2$ the action is irreducible, for otherwise 
there would be a degree $p$ isogeny $3^+\to E'$ over $\bQ$, and $E'$ would also have conductor
$24$. Ogg listed all the curves with conductor 24, and they are all isogenous to $3^+$ by an
isogeny of degree $2^u$, with $u=0,\dots,3$. Therefore $3^+$ would have an isogeny not in
$\bZ$, contrary to nonintegral $j$-invariant. 

\subsubsection{Exceptional covers giving $K_\phi(\text{\ind}) \ne K$} Prop.~\ref{GL2pprop} gives
exceptional covers of $p^2$ degree over any number field from any elliptic curve $E$ without complex multiplication,
excluding a finite set of primes $p$ (dependent on $E$). Still, using Ogg's example shows the best meaning of being
explicit for we may include any prime $p>3$. Here we use $\ell$ for a prime of reduction to get indecomposable
rational functions, and exceptional covers,   
$\mod \ell$ that decompose in $\bar \bF_\ell$.

Consider 
$E=3^+$ as in
\S\ref{OggsEx}. 

\begin{prop} \label{GL2pprop} For  this $E$,  $f_p: \prP^1_x\to \prP^1_y$ ($p> 3$) decomposes into two 
degree $p$ rational functions over some extension $K_p$ of $\bQ$ with group
$\GL_2(\bZ/p)/\{\pm 1\}$. It is, however, indecomposable over $\bQ$. 

Suppose
$\ell\ne 2, 3, p$, and $A_\ell \in \GL_2(\bZ/p)$ represents the conjugacy class of the
Frobenius in $K_p$. Then, reduction of $f_p$ mod $\ell$, gives an exceptional
indecomposable rational function precisely when the  group $\lrang{A_\ell}$ acts
irreducibly on $(\bZ/p)^2=V_p$. This holds for infinitely many primes $\ell$. \end{prop} 

\begin{proof} \S\ref{OggsEx} showed for $E=3^+$ the arithmetic (resp.~geometric) monodromy group of the cover $f_p$
is $(\bZ/p)^2\xs \GL_2(\bZ/p)$ (resp.~$(\bZ/p)^2\xs\{\pm 1\}$).  Now apply the nonregular analog of the Chebotarev
density theorem \cite[Cor.~5.11]{FrJ}. Modulo a prime $\ell$ of good reduction, the geometric
monodromy of $f_p \mod
\ell$ doesn't change, and it and some $g=(A_\ell,\bv)\in (\bZ/p)^2\xs \GL_2(\bZ/p)$ (notation of
\S\ref{gpNot}) generate the arithmetic monodromy
$H_p$ where
$A_\ell$ generates a decomposition group for
$\ell$ in the field $K_p/\bQ$. That is, the image of $A_\ell$ in $\GL_2(\bZ/p)/\{\pm 1\}$ is in the conjugacy class
of the Frobenius for the prime $\ell$. Also, $f_p \mod \ell$ is indecomposable if and only if $H_p$ is
primitive. From
\S\ref{gpNot}, this holds if and only if $A_\ell$ acts irreducibly on $(\bZ/p)^2$. 

The same Chebotarev analog also
says any element of $\GL_2/\{\pm 1\}$ is achieved as (the image of) $A_\ell$ for infinitely many $\ell$.
Acting irreducibly is the same as the (degree 2) characteristic polynomial of $A_\ell$ being irreducible over
$\bF_\ell$. The elementary divisor theorem says every irreducible degree 2 polynomial is represented by a matrix
acting irreducibly. From this there are infinitely many $\ell$ with $f_p \mod \ell$  indecomposable
over $\bF_\ell$ but not over its algebraic closure. We have only to relate 
exceptionality and indecomposability $\mod \ell$.

Suppose $A_\ell  \in \GL_2(\bZ/p)$ acts irreducibly. Let $X=\prP^1_x$ and $Y=\prP^1_y$. Then, $f_p \mod \ell$ 
decomposes into two  degree $p$ rational functions over $\bF_{\ell^2}$. Any
component
$U$ of  $X_Y^2\setminus \Delta$ is birational to the algebraic set defined by a relation
between $x_1$ and $x_j$ with $x_1$ and $x_j$ two distinct points of $X$ over a generic
point $y \in Y$. With no loss assume $A_\ell$ fixes $x_1$. So it moves $x_j$ to another
point,  a point different from the conjugate of $x_j$ from applying the
nontrivial element of the geometric monodromy group corresponding to $-1$ (or else $A$
leaves a subspace invariant). Conclude: The Frobenius moves the absolutely irreducible component from the
relation between $x_1$ and $x_j$. So, that component is not defined over $\bF_
\ell$. That means indecomposability is equivalent
to exceptionality. 
\end{proof}  

In Prop.~\ref{GL2pprop}, $K_p$  contains all $p$th roots of 1, but it is far 
from abelian.  So those $\ell$ above, running over all $p$, produce tremendous
numbers of exceptional rational functions. Asking Ques.~\ref{invFunct} on the order of the inverse
of  $\phi_t$ for each is  valid. 

\subsection{Explicit primes of exceptionality} \label{expPrExc} We give a model for \cite{exceptTow} for our best
understanding of how we could explicitly describe the primes $\ell$ that give exceptionality for $f_p\mod \ell$ in
Prop.~\ref{GL2pprop}. Our two primes $p$ and $\ell$ defies classical notation. So, in figuring where
\S\ref{xn-x-1} is going, substitute $(p^2-1)/2$ for $n$  and $\ell$ for $p$.

\subsubsection{A tough question for the easy polynomials $x^n-x-1$} \label{xn-x-1} For an irreducible quadratic
polynomial
$f(x)\in
\bZ[x]$, {\sl quadratic reciprocity\/} allows explicitly writing down the collection of primes for which $f$ has
no zeros as a union of arithmetic progressions (and a finite set of explicit primes). 

\cite{Se-La} considers this
set of polynomials
$\{x^n-x-1\}_{n=1}^\infty$, well-known to be irreducible, with group
$S_n=G(L_n/\bQ)$. The task he sets is to write, for each $n$, an automorphic form (on the upper half plane) whose $q$
expansion is $\sum_{m=0}^\infty a_mq^m$ and from which we can decide the number of
zeros $N_{p,n}$ of $x^n-x-1$ mod $p$ from $a_p$. 

The last case he gives is when $n=4$. He says \cite{crespo}
gives a newform $F(q)$ of weight 1 from which he extracts the formula 
\begin{equation} \label{serreForm} (a_p)^2=\bigl(\frac{p}{283}\bigr)+N_{p,4}-1, \text{ for
$p\not = 283$.}\end{equation} It so happens there is a cover  $\phi: \GL_2(\bF_3) \to S_4$ with kernel $\bZ/2$ 
and a natural embedding $\rho: \GL_2(\bF_3)\to \GL_2(\bC)$. 

A  
theorem of Langlands and Tunnell says, if a Galois extension of $\bQ$ has group $\GL_2(\bF_3)$, then you can 
identify the Mellin transform of the L-series for $\rho$ with a weight 1 automorphic function.  Tate constructed a
Galois extension $\tilde L_4$ of $\bQ$ unramified over $L_4$ realizing $\phi$. Since Serre already had experience
with this L-series from Tate's extension, he knew how to express it using standard automorphic functions. The
character formula $\rho\otimes \rho=\epsilon\oplus (\theta -1)$ is done in standard books on representation theory
to  write all characters of a small general linear group. Here $\theta$ is the degree 4
permutation representation character for $S_4$. So, $\theta(\tilde g)$ is
$N_{p,4}$ if the image of $\tilde g$ is the Frobenius for $p$ in $L_4$, and $G$ is the character from quadratic
reciprocity on the degree 2 extension of $\bQ$ in $L_4$ (sign character of $S_4$). Even with this, however, Serre
has no closed formula for $N_{p,4}$;  in his expression in standard automorphic forms, they appear to powers. 

\subsubsection{Automorphic connections to exceptionality primes}  \label{automorConn}
To me the statement \cite[p.~435]{Se-La} is still cryptic (though I'm aware there are few nonsolvable
extensions of $\bQ$ expressed through the Langlands program by cusp forms):
\lq\lq No explicit connection with modular forms
\dots is known [for $n\ge 5$], although some must exist because of the Langland's program.\rq\rq\  Still, compatible
with another Serre use of automorphic forms in this paper, I accept it as a worthy goal and formulate an analog of
finding such a form related to Ogg's example. Let $K_p/\bQ$ be the constant extension of the Galois closure of the
cover $f_p$. 

\begin{prob} For each prime $p\ge 5$, express the primes $\ell$ where the Frobenius in
$G(K_p/\bQ)=\GL_2(\bZ/p)/\{\pm1\}$ acts transitively on $(\bZ/p)^2\setminus \{\pmb 0\} \mod \pm I$ as a function of
the
$\ell$-th coefficient 
$a_\ell$ of the $q$-expansion of an automorphic function
$F_p(q)=\sum_{n=0}^\infty a_mq^m$. This is equivalent to expressing the primes $\ell$ in
Prop.~\ref{GL2pprop} with
$f_p\mod
\ell$ exceptional.  \end{prob} 

\cite{exceptTow} uses results from the Langlands Program for $\SL_2(\bZ/5)/\{\pm 1\}=A_5$ to look at the case $p=5$.
Of course, one may consider this problem for any elliptic curve over $\bQ$ without complex
multiplication.

Now Ogg's curve has been long known to be modular. So there is an explicit expression for its Hasse-Weil zeta
function as a weight two cusp form. For any elliptic curve $E$ over $\bQ$, consequence of Wiles' proof of the
Shimura-Taniyama-Weil conjecture, the same holds. \cite[Thm.~22]{Se-Cheb} uses that cusp form  to show,  under the
generalized Riemann hypothesis, that if $E$ has no complex multiplication then there is a constant $c$ independent
of $E$ for which the Galois group generated by the $p$-division points on $E$ is isomorphic to $\GL_2(\bZ/p)$ for
all $p> cD_E$ where $D_E$ is an expression just of the product of the primes at which $E$ has bad reduction. 

If $F_E(q)=\sum_{m=0}^\infty b_m q^m$ is this automorphic function, then for the primes of good reduction of $E$,
$b_p=1+p-N_p(E)$ where $N_p(E)$ is the number of $\bF_p$ points on $E\mod p$.  Use similar notation for
another elliptic curve $E'$. Here are results of
\cite{Se-Cheb} that give the result above. 

\begin{edesc} \label{asymp} \item \label{asympa} For any specific integer $h$ There is an asymptotic bound on  the
number of primes
$p< x$ for which 
$b_p=h$. \item \label{asympb} For some $p$ less than a specific constant of the type above,  $a_p\not
=a_p'$. 
\end{edesc} 

It is with \eql{asymp}{asympa} when $h=0$ (supersingular primes for $E$) that we conclude, though it is in the wrong
direction,  for our next question. So, we note 
\cite{LT} conjectures the number of supersingular primes for $E$ without complex multiplication is asymptotic to
$c_Ex^{1/2}/\log(x)$, $c_E>0$. Our final question is on the median value curve topic of \S\ref{nameExc}. 

\begin{prob} \label{superSingExc} Let $E$ be Ogg's elliptic curve $3^+$. Is there a presentation of  $E \mod p$ as an
exceptional cover for all primes $p$ for which $E$ is supersingular. \end{prob}
While we can ask this kind of question for all elliptic curves, this explicit curve and its isogenies to other
elliptic curves have been well-studied.  The result we are after is to give one elliptic curve whose reductions
have presentations as exceptional covers of $\prP^1_y$ for infinitely many $p$. 

\subsection{Wildly ramified subtowers} \label{wildRamExc} 
This subsection is on wildly ramified exceptional covers. We  assume understood  that all (indecomposable )
polynomial exceptional covers
$P:\prP^1_x\to \prP^1_z$ over $\bF_q$ of degree prime to $p$ come from the proof
of Schur's conjecture. This is Prop.~\ref{basicSchur}, slightly augmented by \cite[\S5]{FGS} to handle the
characteristic 2 case, where there is some wild ramification. 

Our comments 
aim at describing the limit group of the subtower $\WP_{\prP^1_y,\bF_q}$ (of 
$\sT_{\prP^1_y,\bF_q}$, $q=p^u$) that indecomposable polynomials, wildly ramified over $\infty$, generate. Call
the subtower generated by those of $p$-power degree  the {\sl pure wildly ramified
subtower}.  Denote it by $\WP^\pu_{\prP^1_y,\bF_q}$. The Main Theorem of \cite{FGS} says this. 
\begin{edesc} \label{wildPoly} \item \label{wildPolya} If $p\not=2$ or 3, then
$\WP^\pu_{\prP^1_y,\bF_q}=\WP_{\prP^1_y,\bF_q}$, and generating polynomials have affine geometric monodromy 
$(\bF_p)^t\xs H$ with $H\le
\GL_t(\bZ/p)$ (\S\ref{notation}). 
\item \label{wildPolyb} If $p=2$ and 3, add to $\WP^\pu_{\prP^1_y,\bF_q}$ polynomial generators with
almost simple monodromy of core $\PSL_2((\bZ/p)^a)$ ($a\ge 3$ odd) to get $\WP^\pu_{\prP^1_y,\bF_q}$. 
\end{edesc} 

\subsubsection{What can replace Riemann's Existence Theorem} \label{repRET} A general use of RET related ideas 
appears in  \cite{FrGCov} and \cite{Gur-Stev} under the following rubric.  Given a pair of groups $(G,\hat
G)$ that could possibly be the geometric-arithmetic monodromy group pair for an exceptional cover, each shows that
covers do occur with that pair. \cite[\S3.2.2]{prelude} explains the different territories covered by these results.
We briefly remind of these. The former gives tame covers of $\prP^1_y$ over $\bF_q$ where $p$ is sufficiently
(though computably) large. The latter gives wildly covers of curves of unknown genus 
over
$\bF_q$ with
$p$ fixed, but $q$ unknowably large. What \cite{exceptTow} continues is the use \cite{FrMez} to get a result like
\cite{Gur-Stev}, but with the virtues of  \cite{FrGCov}. That means, effective, even for covers of $\prP^1_y$ over
$\bF_q$ with $p$ fixed, and $q$ bounded usefully.   
 
The Guralnick-Stevens paper uses \cite[Main Theorem]{Kz88}. We comment on that and  
a stronger result from
\cite[p.~231--234]{FrSchurII}, which was used almost exactly for their purpose. (There are more details and
embellishments in
\cite{FrMez}.) \cite[Main Theorem]{Kz88} says separable extensions of
$\bar
\bF_p((\frac 1 z))$ correspond one-one with geometric Galois covers
$\phi: X\to \prP^1_z$ with these properties.

\begin{itemize} \item They totally ramify over $\infty$ with group $P\xs
H$.\item The group $H$ is cyclic and $p'$, and $P$ is a $p$-group. 
\item  $\phi$ tamely ramifies over 0 and does not ramify outside 
$\{0,\infty\}$.\end{itemize} 

RET works by considering the deformation of the branch points of a tame cover of a curve $C$. In the explicit case
when $C=\prP^1_y$, RET gives great command of how these covers vary as you deform their ($r$) branch points keeping
them distinct. That control comes from representations of the Hurwitz monodromy group (as in \eqref{HurMon}),
identified with the fundamental group of the space $U_r$ of $r$ unordered branch points. 

The space $U_r$ is a target for any family of $r$ branch
point covers. By recognizing the hidden assumptions in this \wsp under the label {\sl configuration space} \wsp 
\cite{FrMez} forms a configuration space that replaces $r$ by a collection of data called {\sl ramification data}.
Note that exceptional covers are {\sl far\/} from Galois. 

This ramification data, and the Newton Polygon attached to
it, are invariants defined for any cover, not necessarily Galois. The significance of this Galois closure
observation is serious when considering wildly ramified covers. That is because the Galois closure process used
for families of covers in
\cite{FrVMS}, by which we compare arithmetic and geometric monodromy, is much subtler for wildly ramified covers
\cite[\S6.6]{FrMez}. The use of Harbater patching in \cite{Gur-Stev} sets them up for dealing with, one wildly
ramified branch point, with the rest tamely ramified.   It allows nice comparison with general
use of  \cite{FrMez} applied to exceptional polynomial covers, with the only case left, where they have affine
monodromy groups (see below). 

\subsubsection{A surprising source of dissension!} If you were a co-author of a book, you likely would expect
your co-author to ask your opinion on matters in which you are demonstrably expert. You wouldn't expect him to
publish,  in a new edition, versions of {\sl your\/} results as if they belonged to others,  versions many years
later than yours.  You wouldn't expect to have no say about all this, would you? 

Related to the topics of this paper, \cite[Lem.~21.8.11]{FrJ4} quotes
\cite[Prop.~2.2]{turnwald} for the proof of the statement Lem.~\ref{decompKtoKbar},  quoted from two of my
first four papers, essentially from the same time as \cite{FrSchur}. The proof of 
\cite[Prop.~2.2]{turnwald} is identical to mine in \cite[Prop.~3, p. 101]{FrPap0}. The whole context of using the
lemma for primitivity is mine, used whenever related topics come up. Further, my proof of Schur's conjecture was
in about four pages. 

If contention caused this, then its
bone is RET. Having developed 
tools enhancing RET that work in generality, I went home one night  as a recent PhD (at the Institute
for Advanced Study in 1968) and thought I would apply it to a list of problems that included Davenport's. First,
however, there was Schur's Conjecture. I saw the tools were in place so it all came down to group theoretic
statements. I found in the library Burnside's and Schur's group theorems soon after.  With
Schur's conjecture out of the way, it was possible to attack the serious business in Davenport's problem,
and the study of the exceptional examples there. 

Twenty-five years later there is in print another proof of Schur's conjecture,
differing at one point. From Riemann-Hurwitz alone, exactly as done in all these sources, you get down to wanting to
know this. Is a genus 0 dihedral cover totally ramified over $\infty$, and ramified over two finite branch points, 
represented up to linear equivalence by a Chebychev polynomial? (As comments on Prop.~\ref{ordSchurT} explains,
sensitivity to Dickson polynomials is illusory generality.) 

The uniqueness up to affine equivalence of a polynomial cover with
$D_p$ as monodromy group comes immediately from RET and the uniqueness of the branch
cycle description. Instead of that, \cite{turnwald} give a \lq\lq direct proof.\rq\rq\ Of course that is easy!
The Galois closure of the cover is a sequence of two genus 0 cyclic covers. RET in that
case follows from using the first semester of graduate complex variables {\sl branch of log\/} \cite[Chap.~1]{FB}.
Still, essentially my first paper proved a (then) 50 year old unsolved problem overnight because I
powerfully used RET to turn the whole thing into combinatorics and deft use of Lem.~\ref{decompKtoKbar}. Then, I
went on to Davenport's much  tougher problem \cite{FrRedPol}.  

Here is
\cite[p.~493]{FrJ4} dismissing RET: \lq\lq \cite{FrSchur} uses the
theory of Riemann surfaces to prove Schur's conjecture.\rq\rq\  Consider this in the light of what happens with 
nonsolvable monodromy groups: The only real tool is insights from RET. 

\begin{prob} Explain why a co-author who often asks for your mathematical help would do this. Then, try, why he would
want to dismiss one of the greatest geniuses of mathematical history (Riemann)? Then, for fun, take up my challenge
in \S\ref{nonSporadics} of doing Davenport's problem as in
\S\ref{MullerDPs} without RET.  
\end{prob}

Yet, there is more. \cite{FGS} take on wildly ramified exceptional covers, the first to do so coherently. Step back!
If exceptional covers have any significance, then you want their nature. That means their arithmetic
monodromy groups, period! 

Again primitivity is the key,  so you need only look at the primitive groups. The
result is this. \cite{FGS} listed all arithmetic monodromy groups of primitive polynomials over a finite field
with one caveat. A mystery was this affine monodromy possibility. There might be
unknown exceptional polynomials over
$\bF_q$ ($q=p^u$) with geometric monodromy group $(\bZ/p)^n\xs H$, $H$ acting irreducibly on $(\bZ/p)^n$
(as in \eqref{wildPoly}). The polynomial would then have degree $p^n$. There are so many primitive affine groups,
so that is what we considered the major unsolved remainder about exceptional polynomials. Yet,
\cite[Thm.~8.1]{FGS} almost trivialized the nearly 100 year old Dickson conjecture (\eql{trueExc}{trueExcc}; no
serious group theory needed), including it in the precise description  of the rank $n=1$ case of 
exceptional polynomials.   

Jarden sent our paper \wsp as an editor of
the Israel Journal \wsp to D.~Wan who, apparently in this refereeing period, formulated the {\sl Carlitz-Wan
conjecture\/}. That says the exceptional polynomial degrees are prime to $q-1$.  So, the affine case already passes
this conjecture. Instead of the above, \cite[p.~487]{FrJ4} says only that a proof is contained in
\cite{FGS}. It says nothing of what \cite{FGS} proves, as given in the previous two sentences. I quote: 
\begin{quote} A proof of the Carlitz-Wan Conjecture for $p>3$ that uses the classification of finite simple groups
appears in \cite{FGS}. It gives information about the possible decomposition factors [of the monodromy
groups].\end{quote}

Both the $p>3$ and the lazy reference to decomposition groups is ridiculous. We knew exactly what the
monodromy groups (of the non-$p$-power degrees) were for
$p=2$ and 3, and for all others they were affine groups as listed above. More so, 
\cite{FGS} has nothing to say on the Carlitz-Wan conjecture because the paper was already in print before we heard
of it. 

Most importantly,
\cite{FrJ4} takes three pages on the Carlitz-Wan conjecture proof 
\wsp exposition from \cite{Cohen-Fried} 
\wsp and what does that give?  That conjecture is on the nature of tamely ramified extensions over the completion at
infinity. The Carlitz-Wan conjecture is a contrivance  to steal attention from a real theorem. That contrivance
worked and is supported by
\cite{FrJ4}. Compare it with  \cite{FGS} about the topic of interest, exceptional
polynomials as explained in
\S\ref{wildRamExc}. 

\begin{rem} I never saw a copy of \cite{FrJ4} until it was in print. While there seem to be laws 
preventing that, you have go to court: international in this case! \end{rem}

\begin{rem}[Producing the monodromy groups] \label{toLenstra} Note how careful attention to monodromy groups led
others to projects (listed in \eql{trueExc}{trueExcb} and \eql{trueExc}{trueExcc}) investigating 
actual exceptional polynomials. This exemplifies  being able to {\sl grab a group\/}: having a 
workable  use of the classification (as in 
\S\ref{primGrab}). Yet, Lenstra never once mentioned \cite{FGS} in his talk at MSRI in Fall of 1999 (see p.~1
acknowledgements). 

Using \cite{FGS}, the papers  \cite{GZ} and \cite{GRZ} classify all indecomposable exceptional
polynomials with $\PSL_2$ monodromy (as in \eql{wildPoly}{wildPolyb} and \eql{trueExc}{trueExcc}).   Also, 
\cite{GZ} has all the indecomposable polynomials, excluding those in \eqref{wildPoly} with affine monodromy group of 
prime-power degree,  that become decomposable over some extension.  These are the only examples: In
characteristic 7, that of M\"uller in
Rem.~\ref{firstDecompStatement}  of degree 21; and in characteristic 11, of degree 55. 
\end{rem}

\subsubsection{Problems on periods of exceptional correspondences} 
Suppose we have an
exceptional correspondence between copies of $\prP^1_z$ (\S\ref{excCor}). Is there some 
structure on the permutations these produce on
$\prP^1_z (\bF_{q^t})$ running over $t$ in the exceptional set? Example: If  
$(n,q^t-1)=1$, then Euler's Theorem ($\bF_{q^t}^*$ is cyclic) gives the inverting map for
$z\mapsto z^n $ on
$\prP^1_z(\bF_{q^t})$.
We pose finding analogs for more general exceptional covers  
such as those in  these exceptional
towers.  
\begin{edesc} \label{trueExc} \item \label{trueExca}  The
$\GL_2$ exceptional tower (\S\ref{restGen0tame}); or \item  \label{trueExcb} 
1-point and 2-point wildly ramified exceptional towers which will contain all subtowers generated by exceptional
polynomials  
\item \label{trueExcc}  \label{trueExcc} Especially from the Dickson conjecture proof  \cite[Thm.~8.1]{FGS} of 1896 
and the Cohen-Lenstra-Matthews-M\"uller-Zieve $\PSL_2$ monodromy 
examples (as in \eql{wildPoly}{wildPolyb}; \cite{CM}, \cite{LeZ96}, \cite{Mu}).\end{edesc} 

Suppose $\phi: \prP^1_x\to \prP^1_y$ is one of the exceptional genus 0 covers listed in
\eqref{trueExc}. Use the notation of Ques.~\ref{invFunct} for the period $m_{\phi,t}$ of $\phi$ over
$\bF_{q^t}$ after identifying $\prP^1_x$ and $\prP^1_y$. Consider the Poincar\'e series
$P_\phi=\sum_{t\in E_\phi(\bF_q)} m_{\phi,t}w^t$. 

\begin{quest} \label{invFunct2} Is $P_\phi$ a rational function? 
\end{quest}

Suppose
$\phi_i: X_i\to Y$,
$i=1,2$,  is any pair of $\bF_q$ covers (of absolutely irreducible curves). From
\eqref{davProp}, these are a DP if and only if $X_1\times_{Y} X_2$ is a pr-exceptional
correspondence between $X_1$ and $X_2$ with $E_{\pr_1}\cap E_{pr_2}$
infinite. Then, it is automatic from the Galois characterization of DPs (in 
\eqref{davProp}) that this intersection is a union of full Frobenius progressions. 

Suppose $W$ is a pr-exceptional correspondence between any two varieties $X_i$,
$i=1,2$. Then, the exceptional sets for $\pr_i: W\to X_i$, $i=1,2$ are also unions of full
Frobenius progressions. 

\begin{quest} Could it happen that $E_{\pr_1}\cap E_{\pr_2}$ is empty (even if these varieties come with 
covers $\phi_i: X_i\to Y$, $i=1,2$, and $W=X_1\times_Y X_2$)?  \end{quest} 

\section{Monodromy connection to exceptional covers} \label{monConn}

This section extends the historical discussion from \S\ref{prehist}. The name
exceptional arose from Weil's Theorem on Frobenius eigenvalues applied to a
family of curves. Davenport and Lewis considered special  situations for the following
question. Suppose 
$P_{f,g}=\{f(x,y)+\lambda g(x,y)\}$ is the pencil over $\bF_p$, and $p+E_\lambda$ is the
number of solutions in $(x,y)\in \bF_p\times \bF_p$ of the equation given by the parameter
$\lambda$. 

\begin{quest} Can you give a lower bound on an accumulated estimate for the error term
from  Weil's result running over rational values of $\lambda$? \end{quest}  Their aim was
find out for which $(f,g)$  a nonzero constant times 
$p^2$  would be a lower bound for $\sum_{\lambda} E_\lambda^2\eqdef W_{f,g}$. That is, when
would the Weil error of
$c_\lambda \sqrt{p}$ accumulate significantly in the pencil? 
  
\subsection{The name exceptional appears in \cite{dlcong1}} \label{monConna}
\cite{dlcong1} considered this hyperelliptic pencil: $y^2-f(x) +\lambda$, $f\in
\bF_p[x]$. They concluded
$W_{y^2-f(x),1}
\ge c_fp^2$, with $c_f>0$, if $f:X=\prP^1_x\to \prP^1_z=Y$ is not {\sl exceptional\/}. 

Use notation from
\S\ref{fibprod}. Soon after publication of \cite{dlcong1}, being exceptional meant  
\eql{rh}{rha} in Prop.~\ref{RHLem}: $X_Y^2\setminus \Delta$ has no absolutely
irreducible $\bF_p$ components. For their case, let 
$k_f$ be the number of its absolutely
irreducible $\bF_p$ components. Though confident of expressing
$c_f$ in the degree of
$f$, they are not precise about it.

Denote the Jacobi symbol of $u \mod p$ by $(\frac u p)$.  
Notice:  $|\{(x,y)\mid y^2=f(x)\}|$ is $\sum_{x\in \bF_p} 1+(\frac{f(x)+\lambda}
p)=p+E_\lambda$. Therefore: 
$$(E_\lambda)^2 =
\sum_{x,y}\Bigl(\frac{f(x)+\lambda} p\Bigr) \Bigl(\frac{f(y)+\lambda} p\Bigr)=
\sum_{x,y}\Bigl(\frac{(f(x)+\lambda)(f(y)+\lambda)} p\Bigr). $$ Now sum a
particular summand in $(x,y)$ over $\lambda$. If $f(x)\equiv f(y) \mod p$, then all
arguments are squares, adding up to $p-1$ for the nonzero arguments. Otherwise, 
complete the square in  $\lambda$. The sum becomes $U_d\eqdef
\sum_{u}\bigl(\frac{u^2+d} p\bigr)$ for some nonzero $d\mod p$. Note:  $U_d$ depends
only on whether $d$ is square $\mod p$. From that,  summing $U_d$
over $d$ shows $U_d$ is independent of $d$: it is  $-1$. 

Let $V=\prP^1_x\setminus\{\infty\}$, $U=\prP^1_z\setminus \{\infty\}$.  
We conclude:
$W_{y^2-f,1}=pN_f$ with $N_f=|(V^2_U\setminus \Delta)(\bF_p)|$. Weil's
estimate shows 
$N_f=k_fp+O(p^{1/2})$. So, $k_f$ is the main determiner of the constant in the
Davenport-Lewis result.  This is the source of the name {\sl exceptional\/} for polynomials
$f$. 

\cite[p.~59]{dlcong1} notes cyclic and Chebychev polynomials are exceptional
for those primes $p$ where they are permutation. 
\begin{quote} Both substitution polynomials and exceptional polynomials admit functional
composition:  If $f$ and $g$ belong to these classes, then so does $f(g(x))$. This is
obvious in the case of substitution polynomials and \dots \end{quote} They partially  
factor $f(g(x))$ over $\bF_p$ to see it is exceptional if $f$ and $g$ are. They weren't
sure  their meaning of exceptional also meant \eql{rh}{rhb} in Prop.~\ref{RHLem}. 
Was $f$ automatically substitution? 
C.~MacCluer's 1966 thesis
\cite{MacCluer} took on that question, answering it affirmatively for tame polynomials
satisfying \eql{rh}{rha}. The proof of Princ.~\ref{liftPrinc} seems easy now, applying
generally to pr-exceptional. Yet, the literature shows that belies much 
mathematical drama.

\subsection{The monodromy problem of \cite{Kz81}} \label{monConnb} 
Let $\Phi: \sX\to S$ be a smooth family of (projective) curves over a dimension $N$ base
$S$. Assume the  family has definition field $K$, which we take to be a number field.
This setup has an action of the fundamental group $\pi_1(S,s_0)=G$ on the 1st
cohomology $V=H^1(\sX_{s_0},\bC)$ of the fiber of
$\Phi$ over $s_0\in S$. Let $V_s=H^1(\sX_s,\bC)$ for $s\in
S$. 
\begin{triv} \label{locConst} Equivalently,  $\dot{\cup}_{s\in
S}V_s$ is a
locally constant bundle over $S$. \end{triv}

\subsubsection{Using complete reducibility} A theorem of Deligne says 
$G$ has {\sl completely reducible\/} action \cite[Thm.~3.3]{griffith}. So, 
$V$ breaks into a direct sum $\oplus_{i=1}^m V_i$ with $G$ acting on each $V_i$
irreducibly (with no proper invariant subspace). Two irreducible representations $\rho':
G\to GL(V')$ and
$\rho'': G\to GL(V'')$ of
$G$ are equivalent if $\dim_{\bC}(V')=\dim_{\bC}(V'')=n$, and for some
identification of these with $\bC^n$, there is an element $M\in \GL_n(\bC)$ with 
$\rho'=M\circ\rho''\circ M^{-1}$. Rewrite the sum $\oplus_{i=1}^m V_i$ as $\oplus_{i=1}^{m'}
m_iV_i'$ with the $V_i'\,$s pairwise inequivalent.    Denote $\sum_{i=1}^{m'}m_i^2$ by
$W_\Phi$.  Then, with $V^*$ the complex dual of $V$ (with $G$ action): 
\begin{equation} W_\Phi=\sum_{i=1}^{m'}m_i^2=\dim_\bC
\text{End}_G(V,V)=\dim_\bC (V\ot V^*)^G.\end{equation}  

\subsubsection{The strategy for going to a finite field}  The $\ell$-adic analog of
\eqref{locConst} gives  varying $\ell$-adic 1st cohomology groups over the base $S$. These  
form a locally constant sheaf $\sT=\sT_\ell$ with
$G$ action. Elements of the absolute Galois group $G_K$ also act on this.  There is a comparison theorem in $\bQ_\ell$
developed by Artin, Deligne, Grothendieck and  Verdier that Deligne used extensively
\cite{deligne}.  

The idea from here is to regard $S$ as an algebra over some ring of
integers $R$ of $K$ and to use primes $\bp$ of $R$ for 
reducing the whole family. 
Suppose the residue 
class field $R/\bp$ has order $q$. We would then have
a sheaf on which the Frobenius $\Fr_{q}$ ($q$-power map) acts. To relate this to a 
Davenport-Lewis type sum for the accumulated Weil error, we need a two-chain comparison.

\begin{edesc} \label{comp} \item  \label{compa} Extract the Davenport-Lewis estimate
for the family over
$R/\bp$ from $\Fr_{q^t}$ action (some $t$) on the cohomology of the $\ell$-adic
sheaf
$\sT\ot\sT$. 
\item  \label{compb} Compare $\Fr_{q^t}$ on the cohomology with the quantity
$W_\Phi$. 
\end{edesc} 

The comparison \eql{comp}{compa} is crucial. The rational prime $p$ that appeared in the
Davenport-Lewis estimate is long gone. So, we will be considering the analog of their
computation with $\bF_{q^t} (\supset R/\bp\eqdef \bF_\bp)$ for $t$ large replacing $\bF_p$, and
subject \wsp as we will see \wsp to another constraint.  The convention for writing the
Davenport-Lewis estimate for the family over $\bF_{q^t}$ is in the following
notation: 
\begin{equation} \label{lefschetz1} \sum_{s\in S\ot_R\bF_{\bp}(\bF_{q^t})}
E_{\bp,t,s}^2=\sum_{s\in S\ot_R\bF_{\bp}}\tr(\Fr_{q^t}|\sT_s\ot\sT_s).\end{equation}
The Lefschetz fixed point formula computes the right side as \begin{equation}
\label{lefschetz2} \sum_{i=0}^N (-1)^i\tr(\Fr_{q^t}| H^i_\ell(S\ot\bar
\bF_{\bp},\sT\ot
\sT)).\end{equation} 

\subsubsection{Using the full Weil Conjectures} Deligne's version of the Riemann
hypothesis isolates one term ($i=N$) from this. With that we conclude by
fulfilling
\eql{comp}{compb}. To do so requires assuring the trace term on $H_\ell^{2N}$ has
a bound away from zero in the limsup over $t$: So $\Fr_{q^t}$ eigenvalues on it don't nearly cancel for all $t$.

Then, that term will have absolute value  roughly 
$q^{(N+1)t}$ times $\dim_{\bQ_\ell}(H_\ell^{2N})$. (Don't forget to add the affect of
$\Fr_{q^t}$ on the stalk $\sT_s\ot
\sT_s$ in which the cohomology elements take values.) This will dominate all other terms in
\eqref{lefschetz2}. Still, to isolate out that term, we must choose $t$ large, and yet 
mysteriously. Reason? We don't actually know what are the eigenvalues of the
Frobenius on
$H^{2N}(S\ot\bar
\bF_{\bp},\sT\ot
\sT)$, though we  soon interpret how many there are.  

To fix notation, suppose $\row
\gamma u$ are the eigenvalues of the Frobenius for $\bF_{\bp}$ on $H^{2N}(S\ot\bar
\bF_{\bp},\sT\ot
\sT)$, with $\bF_{\bp}=\bF_{q^{t_0}}$. Consider the
corresponding eigenvalues of the Frobenius for $\bF_{q^{t}}$ with $t_0$ dividing $t$, which
is the $t/t_0=v$ power of the first Frobenius. So its eigenvalues are the
$v$th powers of $\row \gamma u$. These all have absolute
value $q^{v(N+1)}$. A simple diophantine argument shows there is a subsequence $L$ of such
$t$ so the absolute value of $(\sum_{i=1}^u \gamma_i^{u})/q^{v(N+1)}$ has limit $u$. 
This is the limsup of the right side of \eqref{lefschetz2} divided by $q^{t(N+1)}$ as a
function of
$t$ (divisible by $t_0$). Therefore $u$ is Davenport-Lewis limit of the left side of
\eqref{lefschetz1} divided by $q^{t(N+1)}$. For the hyperelliptic family,
this was the number of absolutely irreducible factors of $X_Y^2\setminus \Delta$ over the
fields $\bF_{q^t}$, $t\in L$. 

The number $W_\phi$ is the same as $\dim(H^0_\ell(S\ot\bar
\bF_{\bp},\sT\ot
\sT))$. By Poincar\'e duality, this is the same as $\dim(H^{2N}_\ell(S\ot\bar
\bF_{\bp},\sT\ot
\sT))=u$. It is the left side of \eqref{lefschetz1} divided by $q^{t(N+1)}$.  
So, the
Davenport-Lewis
 estimate only works on the quantity Katz is after if we run over the $\limsup_t$ 
$\ell$-adic cohomology estimate.  

Generalizing this
situation has straightforward aspects. We comment on that, then conclude in 
\S\ref{excZeta} with a different tack on the Davenport-Lewis setup. This motivates how
\cite{exceptTow} uses zeta relations to detect the effects of exceptionality. 

Since the fibers are curves, you can easily 
adjust to consider collections of  affine curves with points deleted from the
fibers. This doesn't affect the final computation: Using error estimates
from the affine (instead of from the projective) fibers gives the same result.
\cite[\S IV]{Kz81} writes this in detail. Also, in estimating counting errors
in rational points, it may be  useful to have $S$ an open set in $\afA^N$ over $R$, with 
the family the restriction of $\sW\to \afA^N$ (still with 1-dimensional fibers).
If we use the latter family to make the count, likely some fibers will be singular, even 
geometrically reducible. What happens if we include them in the computation for our
estimate for the calculation over $S$? Answer: This makes the error for 
a family over
$\afA^N$ an upper bound to counting the sums of squares of the irreducible components
for the monodromy action
\cite[\S V]{Kz81}. 

Katz uses the {\sl wrong\/} direction from \cite{dlcong1}; as an upper, 
rather than lower, bound. It is a shame to lose the precision. So, observe when
$\dim(S)=1$, the correct estimate for $W_\Phi$ is the $\limsup$ of the
Davenport-Lewis error  estimate divided by
$q^{2t}$. That is the expected $k_f$
(computed over the algebraic closure of $K$). 

\subsubsection{Detecting exceptionality through zeta properties} \label{excZeta}

Now we list lessons from the combination of Davenport-Lewis and Katz. 
Consider the projective curve $U_{\lambda}$ defined by  
$y^2 +\lambda u^2-xu=0$ in projective 2-space $\prP^2$ with variables
$(x,y,u)$, for a fixed value of a parameter $\lambda$. Denote the space in
$\prP^2\times \afA^1_\lambda$ defined by same equation as $U^*$. There is a well-defined 
map
$\phi:(x,y,u,\lambda)\in U^*\mapsto x/u=z\in
\prP^1_z$.  

View any (nonconstant) $f(w)\in \bF_q(w)$, $f: \prP^1_{w} \to \prP^1_z$, as a
substitution. \cite{dlcong1} asked how substituting $f(w)$ for $z$ affects  the sum
over $\lambda\in \afA^1(\bF_{q^t})$  of the squared difference between  
$|U_\lambda(\bF_{q^t})|$ and 
$q^t+1$.   
This 
Weil error  vanishes over $\bF_{q^t}$ where $f$ is
exceptional, and we know exactly when that happens (as in Lem.~\ref{basicSchur} and \S\ref{restGen0tame} ). Excluding such $f$ and a possible finite set of $t$ values, it is far from
vanishing. The investigation starting from MacCluer's thesis \cite{MacCluer} found this precise vanishing
for infinitely many $t$ to characterize exceptionality. Note:  In this  formulation, you can
replace
$w\mapsto  f(w)$ by any cover $\psi: X\to \prP^1_z$.  

Katz interpreted this error variation as a zeta function
statement.  Specific conclusions related
to $\pi_1(S,s_0)$ action  involved an
$f$ exceptional over a number field (as in
\S\ref{1-K}). This is just one phenomenon. 
Relations between general zeta functions  
defined by exceptional covers and iDPs (\S\ref{nameExc}) generalize the Davenport-Lewis
situation around exceptional polynomials.

\section{The effect of pr-exceptionality on group theory and zeta functions}  
\label{prodExcCov} The Davenport-Lewis collaboration
\cite{dlcong1}  motivated MacCluer's Theorem \cite{MacCluer}. This first 
connecting of two meanings of exceptionality (\S\ref{monConna}) applied 
just to tame polynomials. Our final form as in Princ.~\ref{liftPrinc}:  pr-exceptionality translates to a pure 
monodromy statement, a (now) transparent proof. This section lists examples of how  
pr-exceptionality  relates to many other topics. 

\S\ref{gpsExcept} enhances the {\sl crossword\/} analogy of
\S\ref{primGrab} for an historical explanation of how exceptionality and
Davenport's problem affected {\sl group theory}.  The examples of \S\ref{cryptexc}  show these special arithmetic
covers raise tough questions on the nature of zeta functions and how much they capture of cover arithmetic.
Finally, we discuss the history of Davenport pairs.  These topic introductions continue in \cite{exceptTow}.

\subsection{Group theory versus exceptionality} \label{gpsExcept} Many supposed by 1969 that we knew
everything about rational functions in one variable that one could possibly care about. \S\ref{ratFunct} and 
\S\ref{expexcTowers} (with technical fill from the appendices) take us through the mathematical history that exposed
that supposition as premature. 

\subsubsection{Rational functions set the scene} \label{ratFunct} Consider a rational function $f$, indecomposable
over  $\bar \bF_q$, that might have appeared in 
\S\ref{excZeta}. When $f$ is a polynomial and
has degree prime to $p$, we know either that $f$ is Dickson or cyclic, or $k_f$ is exactly
1. With any
$f\in \bF_q(w)$, the  limsup of the Davenport-Lewis variation divided by $q^{2t}$ is 
still   $k_f$ computed over $\bar K$. Even,
however, under our extra hypotheses, we don't expect this to be 1. For example, having just
one absolutely irreducible component translates as doubly transitive geometric monodromy.

Our indecomposability criterion is that the geometric monodromy is primitive. The geometric monodromy group of a
rational function is called a {\sl genus 0\/} group. I suspect even those who knew what primitive 
meant in 1969 would have thought the geometric monodromy group of an indecomposable rational function could be any
primitive group whatsoever. That is what the genus 0 problem tackled. The serious unsolved aspects in 1987 
translated to considering genus 0 covers whose geometric monodromy is  primitive, but not doubly transitive. The
main tool, besides group theory, was Riemann's existence theorem (existence of branch cycles as in
\S\ref{introBC}).

\subsubsection{Guralnick's optimistic conjecture} \label{nonSporadics}  I've used the same title for this section
as does \cite[\S7.3]{thompson}. For the convenience of the reader I repeat a bit of that to express what 
is expected (and has been partly proved) on the geometric monodromy of genus 0 covers. (For genus $g=1$ and $g>1$,
there is a similar conjecture about $g$-sporadic groups.)  The easiest result from the elementary part of Riemann's
existence theorem (RET) \wsp use of branch cycles in \S\ref{introBC} 
\wsp is that every finite group is the geometric monodromy group of a cover of $\prP^1_z$. If the following were
truths for {\sl you}, then you might not suspect the need for RET. 
\begin{itemize} \item  It is easy to construct genus 0 covers of $\prP^1_z$ with desired properties. 
\item All groups appear as monodromy groups of genus 0 covers of $\prP^1_z$. 
\end{itemize} Both, however, are false, whatever you mean by {\sl easy\/}, even if you restrict to genus 0 covers
with a totally ramified place (represented by polynomials; see \S\ref{MullerDPs}). 

The original
Guralnick-Thompson conjecture was that for each $g$, excluding finitely many simple groups, the only composition
factors of monodromy groups of
$\prP^1_z$ covers are alternating groups and cyclic groups. Still, composition factors are one thing, actual  genus
0 primitive monodromy groups another. Also, the attached permutation representations do matter. What arose in
the middle 1800s from elementary production of covers were cyclic, dihedral, alternating and symmetric groups
using genus zero covers. Such examples appear in 1st year graduate algebra books. The list of
\eqref{0-sporadicb} shows these and a small set of
tricky  alternatives to these. 

\begin{defn} We say
$T:G\to S_n$, a faithful permutation representation,  with properties \eqref{0-sporadica} and
\eqref{0-sporadicb} is {\sl
$0$-sporadic\/}. \end{defn} 
Denote  $S_n$ acting  on 
unordered $k$ sets of $\{1,\dots,n\}$ by $T_{n,k}:S_n\to S_{\scriptscriptstyle{{n\choose k}}}$:  
standard action is $T_{n,1}$. Alluding to $S_n$ (or $A_n$) with $T_{n,k}$ nearby
refers to this presentation. In \eqref{0-sporadicb}, 
$V_a=(\bZ/p)^a$ ($p$ a prime). Use \S\ref{nonempNielsen2} for semidirect product in the 
$T_{V_a}$ case on points of $V_a$; $C$ can be $S_3$. In the 2nd $(A_n,T_{n,1})$ case, $T:G\to
S_{n^2}$. 
\begin{triv} \label{0-sporadica} $(G,T)$ is the monodromy group of a primitive (\S\ref{nc1})
compact Riemann surface cover
$\phi:X\to\prP^1_z$ with
$X$ of genus $0$.  \end{triv}  
\begin{triv} \label{0-sporadicb} $(G,T)$ is not in this list of group-permutation types. \end{triv}
\!\begin{itemize} \item  $(A_n,T_{n,1})$: $A_n\le G\le S_n,\text{ or } A_n\times A_n \xs \bZ/2 \le G
\le S_n\times S_n\xs \bZ/2$. 
\item $(A_n,T_{n,2})$: $A_n\le G\le S_n$. 
\item $T_{V_a}$: $G= V\xs C$,\  $a\in \{1,2\}$, $|C|=d\in \{1,2,3,4,6\}$ and $a = 2$
only if
$d$ does not divide $p-1$. 
\end{itemize}
 
Indecomposable rational functions $f\in \bC(x)$ represent 0-sporadic groups by 
$f:\prP^1_x\to \prP^1_z$ if their monodromy is not in the list of \eqref{0-sporadicb}. We say 
$(G,T)$ is {\sl polynomial 0-sporadic}, if some $f\in\bC[x]$ has monodromy outside this list. We know of covers
satisfying
\eqref{0-sporadica} and falling in the series of groups in the list of \eqref{0-sporadicb}.  There are, however,
other 0-sporadics with an $A_n$ component \cite{GSh04}. For example, if there were a genus 0 cover with
monodromy
$A_6$ acting on unordered triples from $\{1,2,3,4,5,6\}$, we would call it 0-sporadic. The point, however, of 
0-sporadics is that you only have a small list of $n\,$s for which the geometric monodromy of the genus 0 cover will
be $A_n$ acting on unordered triples. 
 
Emphasis: Don't toss the 0-sporadics away, because it is they that give a clue for quite different set of
primitive genus 0 covers in positive characteristic. The finite set of (genus 0)-sporadic groups (over
$\bC$; App.~\ref{MullerDPs}) adumbrates a bigger set of genus 0 groups  over finite fields. While we don't have so
precise a Riemann's existence theorem in characteristic $p$, there are tools.  By focusing on the group requirements
for exceptional covers and Davenport pairs we have  applied characteristic 0 thinking
to characteristic $p$ problems. 
An understanding why this works starts from  
\cite{FrSchurII}, and a preliminary version of \cite{FrMez} in 1972. More solid
applications in print encourage extending \cite{FrGCov} and
\cite{Gur-Stev}. The precise structure of exceptional towers 
makes describing their limit groups an apt sub-problem from the unknowns left 
by  Harbater-Raynaud (\cite{Ha}, \cite{Ra}) in their solution of Abhyankar's problem.   

Davenport asked me 
several times to explain why transitivity of a permutation representation (from a
polynomial cover $p:\prP^1_x\to \prP^1_z$) is equivalent to irreducibility of $p(x)-z$
over the field $K(z)$. He didn't like Galois theory, and his reaction to  group theory was
still stronger. It wasn't only Davenport. Genus 0 exceptional
covers force an intellectual problem faced by the whole community. 

\begin{edesc} \label{theParadigm} \item \label{theParadigma} 
RET  guides us to how to find exceptional
covers. 
\item \label{theParadigmb} Using exceptional covers demands an explicit presentation
of equations that \eql{theParadigm}{theParadigma} can't give directly. 
\end{edesc} 

\subsubsection{From Davenport pairs to the genus 0 problem} \label{expexcTowers}
I knew Harold Davenport from graduate school (University
of Michigan), my second year, 1965--1966. He lectured on analytic number theory and
diophantine approximation (my initial interest), though this  included related finite field topics.  
Discussions with Armand Brumer (algebraic number theory, from whom I learned Galois theory), Donald Lewis
(diophantine properties of forms; my PhD advisor) and Andrzej Schinzel (properties of one variable
polynomials) were part of seminars I attended. MacCluer attended these, too; we overlapped two years of
graduate school. Problems formulated by Schinzel
used the topics of these discussions. 

My understanding of the literature on finding variables
separated polynomials
$f(x)-g(y)$ that factor started with
\cite{DLSf-g} and
\cite{DSf-g}. At the writing of these papers, the authors did 
not realize the equivalence between this factorization problem and Davenport's value set
problem \cite{FrRedPol}.   Within two years
from that time, I had finished that project. This used small private
lectures from John McLaughlin on permutation representations. 

Years later, I returned to these topics while writing my lecture at Andrzej
Schinzel's  birthday conference
\cite{Fr-Schconf}. I  record some points here.  

\begin{edesc} \label{davMot} \item  \label{davMota} Davenport wished (Ohio State, Spring
1966) that  confusions among polynomial ranges over finite fields received 
greater attention. 
\item  \label{davMotb}  He insisted many used Weil's theorem on zeta
functions gratuitously. 
\item  \label{davMotc}  Groups and Galois theory frustrated him. 
\end{edesc}  
Small subsections below explain each point. 

\subsection{Arithmetic uniformization and exceptional covers} \label{cryptexc} 
Exceptional covers and cryptology go together 
(\S\ref{crypt} and \S\ref{excPerms}). We would now express Davenport's concern in 
\eql{davMot}{davMota} as this:  How to detect when one isovalent DP is formed
from another by composing with exceptional covers. 
\subsubsection{\eql{davMot}{davMota}: 
Davenport's problem led to studying exceptional covers} \label{DavProb} Davenport asked
whether two polynomials could (consequentially) have the same ranges modulo $p$ for almost all primes
$p$? By consequential we mean,  no linear change of variables, even over $\bar
\bQ$, equates them (an hypothesis that we intend from this point). 
\cite{FrRedPol} restricted to having one polynomial indecomposable 
(primitive as a covering map,  \S\ref{prehist}). A first step then says they have the same  
degree. Over an arbitrary number field, there may be consequential Davenport
pairs. Yet, only for a bounded set of degrees $\{7, 11, 13, 15, 21, 31\}$.
Further (again the indecomposable case) this can't  happen over
$\bQ$. The first
result uses the simple group  classification. The second does not. For it, we need only 
the {\sl Branch-Cycle Lemma\/}  (App.~\ref{WCocycle}). 

M\"uller made a practical contribution to the genus 0 problem by listing primitive
monodromy groups of tame polynomial covers. There are three nontrivial families of
indecomposable polynomial Davenport pairs. 
\S\ref{MullerDPs} explains how these Davenport families are {\sl exactly\/} the nontrivial
families of sporadic polynomial monodromy groups. 
 {\sl Nontrivial\/} in that the pairs have a
significant variation; some {\sl  reduced deformation\/} (\S\ref{nc2}). We recount points from the detailed analysis
of \cite[\S3 and \S5]{thompson}.  \S\ref{WHist}, for example, reminds of the historical relation between the
production of Abelian varieties whose field of moduli is not a field of definition \wsp an unsolved problem at the
time \wsp and these Davenport pairs.

\subsubsection{\eql{davMot}{davMotb}: The name exceptional and eigenvalues of
the Frobenius} \label{nameExc}  
For three of our topics, exceptional covers conjure up zeta functions and
Frobenius eigenvalues that support Davenport's desire in
\eql{davMot}{davMota}. 

First: Still with genus 0 exceptional covers, we use \S\ref{monConna} to tell from whence
came  the  phrase {\sl exceptional polynomial\/}. The start was a paper
 in the long collaboration of Davenport and Lewis. \cite{dlcong1} checked, in a
hyperelliptic curve pencil, if the Weil error accumulates significantly. When it
did not, they called that case exceptional. Later they guessed  an equivalence
between their exceptionality and the conclusion of Schur's conjecture
(Prop.~\ref{basicSchur}). The latter generalizes to what we now call exceptionality.  Katz
\cite{Kz81} used
\cite{dlcong1} to discover for the same pencils that exceptionality is equivalent to
irreducible monodromy action of the base's fundamental group on the pencil fibers
(\S\ref{monConnb}). There is, however, a surprise. Katz drew conclusions on exceptional
covers for values of $t$ where, over $\bF_{q^t}$, the polynomials were as far from
exceptional as possible. This motivates topics that are now haphazard in the
literature: To inspect exceptional polynomials outside their exceptional sets, and to
consider exceptional covers of higher genus.  

Second: If $\phi:Y \to \prP^1_z$ is exceptional, then $Y$  
is {\sl e-median\/}. 
\begin{itemize} \item It is {\sl median value}: $Y(\bF_{q^t})=q^t+1$ for 
$\infty$-ly many $t$. 
\item The median value 
 exceptional set of $t$ 
contains $t=1$ (Prop.~\ref{excFibProp}). \end{itemize} 
Exceptional correspondences with
$\prP^1_z$ are examples of e-median curves (\S\ref{excCor}) that are not a' priori given by
curves from an exceptional cover like $\phi$. We characterize Davenport pairs as having
a special  pr-exceptional correspondence between their curves. A fundamental
question arises: How can we characterize curves that have an exceptional correspondence with $\prP^1_z$?
\cite[\S3.5]{FrGCov} notes the genus 1 curves with this property are supersingular. It also checks  
examples (from \cite[Prop.~14.4]{GF}) of supersingular genus 1 curves and shows they are, indeed, exceptional covers
of
$\prP^1_z$. A next step is the program of Prob.~\ref{superSingExc}. The following remark starts our
continuation in
\cite{exceptTow}:  e-median is a pure zeta function property and  not all e-median curves will have supersingular
Jacobians. 
    
Third: Suppose we have a Poincar\'e series $W_{D,\bF_q}(u)=\sum_{t=1}^\infty N_D(t)u^t$ for
a diophantine problem
$D$ over a finite field $\bF_q$. We call these {\sl Weil vectors}. (Example: One 
from a zeta function of an algebraic variety.) Assume also: $\phi_i:X\to Y$,
$i=1,2$, is an isovalent Davenport pair over $\bF_q$. If $D$ has a map to $Y$, this
Davenport pair produces new Weil vectors 
$W_{D,\bF_q}^{\phi_i}$, $i=1,2$, and a {\sl relation\/}
between $W_{D,\bF_q}^{\phi_1}(u)$ and $W_{D,\bF_q}^{\phi_2}(u)$: an infinite
set of $t$, where the coefficients of $u^t$ in 
$W_{D,\bF_q}^{\phi_1}(u)-W_{D,\bF_q}^{\phi_2}(u)$ equal 0. Producing relations 
between Weil vectors  is characteristic of isovalent DPs. \cite{exceptTow} has an effectiveness result: For
any Weil vector,  the support set of $t\in \bZ$ of 0 coefficients  differs by a
finite set from a union of full Frobenius progressions (\S\ref{frobProg}).

\subsection{History of Davenport pairs} \label{BackDPs}
{\sl Davenport pair\/} first referred to pairs $(f,g)$ of polynomials, over a
number field $K$ (with ring of integers $\sO_K$), with the same ranges on almost all residue
class fields.  Now we call that a  {\sl strong\/}  Davenport
pair (of polynomials) over
$K$. An SDP over $(Y,K)$ is a pair of covers $\phi_i: X_i\to Y$, $i=1,2$,  over $K$
satisfying {\sl Range equality}: 
\begin{triv} \label{rangeSt} 
$\phi_1(X_1(\sO/\bp))=\phi_2(X_2(\sO/\bp))$ for almost all prime ideals
$\bp$ of $\sO_K$. 
\end{triv} 

\cite{AFH} reserves the acronym DP over $(Y,K)$ to mean
equality on ranges holds for infinitely many $\bp$. An {\sl iDP\/} is then an isovalent DP
(\S\ref{DavProb} and Prop.~\ref{iDPconsts}), iSDP means isovalent SDP, etc.  

\begin{prop} \label{DPK} If $(\phi_1,\phi_2,K)$ is an iSDP for 
$\infty$-ly many $\bp$, then it is an iSDP for almost all $\bp$.
\end{prop}

\begin{proof} Use notation of \S\ref{davPairs}, with extra
decoration indicating the  base field. For $|\bp|$ large, let  
$\sigma\in G(\hat K/K)$  be a choice of Frobenius for the prime $\bp$. Then, we can 
identify two  geometric-arithmetic  monodromy group pairs \cite[Lem.~19.27]{FrJ}: 
$$(G_{(\phi_1,\phi_2),\sO/\eu{p}},\hat G_{(\phi_1,\phi_2),\sO/\eu{p}})\text{ and }
(G_{(\phi_1,\phi_2),\hat K^\sigma},\hat G_{(\phi_1,\phi_2),\hat K^\sigma}).$$ Restrict
to such $\bp$. Then,
$E_{(\phi_1,\phi_2),\sO/\bp}=\bN^+$ if and only if
$(\phi_1,\phi_2,\sO/\bp)$ is an SDP. 

Lem.~\ref{LemGind} shows this is equivalent to the representation pair  $(T_1,T_2)$   
giving equivalent representations on $G_{(\phi_1,\phi_2),\hat K}$, a condition independent
of
$\bp$. So, excluding finitely many $\bp$, this holds either for all or none of the $\bp$.
\end{proof}

\renewcommand{\labelenumi}{{{\rm (\teql \alph{enumi})}}} 

\begin{appendix}

\section{Review of Nielsen classes} \label{nc1} When $Y=\prP^1_z$, a Nielsen class is
a combinatorial invariant attached to the cover. Suppose $\bz$ is the branch point
set of
$\phi$,
$U_\bz=\prP^1_z\setminus
\{\bz\}$ and $z_0\in U_\bz$. Consider analytic continuation of the points over $z_0$ along
paths based at $z_0$, of the form $\gamma\cdot \delta_i\cdot \gamma^{-1}$, $\gamma, \delta$
on
$U_\bz$ and $\delta_i$ a small clockwise circle around $z_i$. This gives  a
collection of conjugacy classes $\bfC=(\row {\text{C}} r\}$, one for each $z_i\in \bz$, in 
$G_\phi$. The associated  {\sl  Nielsen  
class\/}: \begin{equation} \label{NCDef} \ni=\ni(G,\bfC)=\{\bg= (\row g r) \mid g_1\cdots
g_r=1, 
\lrang{\bg}=G \text{\ 
and\ } \bg\in \bfC\}.\end{equation}  Writing $\bg\in\bfC$ means  the 
$g_i\,$s, in some order, define the same conjugacy classes in $G$  (with multiplicity) as 
those in 
$\bfC$. We call the respective conditions $g_1\cdots g_r=1$ and $\lrang{\bg}=G$, the {\sl
product-one\/} and {\sl generation\/} conditions. Each cover
$\phi: X\to\prP^1_z$ has a uniquely attached Nielsen class: 
$\phi$ is in the Nielsen class $\ni(G,\bfC)$. We give examples in \S\ref{DickMon}. The examples of the degree 7, 13
and 15 degree Davenport pairs in \cite[\S5]{thompson} can give a reader a full taste of why even
polynomial covers require RET. The point is that these three examples are the most significant of the
0-sporadic polynomial covers. The reduced spaces parametrizing these covers are each genus 0 curves defined over
$\bQ$. Each is a  (non-modular curve)  $j$-line cover \cite[Prop.~4.1]{thompson}. These
facts come directly from using Nielsen classes. 

\subsection{Inner and absolute Nielsen classes } \label{nc1}
Suppose we have  $r$ (branch) points $\bz$, and a corresponding choice $\bar \bg$ of {\sl classical 
generators\/} for 
$\pi_1(U_\bz,z_0)$ \cite[\S1.2]{BFr}. Then, $\ni(G,\bfC)$ lists all 
surjective homomorphisms $\pi_1(U_\bz,z_0)\to
G$ with local monodromy in $\bfC$ given by 
$\bar g_i\mapsto g_i$, $i=1,\dots,r$. Each gives a cover with branch points $\bz$ 
associated to
$(G,\bfC)$. The 
$\bg\in \ni(G,\bfC)$ are {\sl branch cycle descriptions\/} for these covers relative to
$\bar\bg$.   Equivalence classes of 
covers with fixed branch points $\bz$ correspond one-one to equivalence classes on 
$\ni(G,\bfC)$. Caution: Attaching a Nielsen class
representative to a cover requires picking one from many possible
$r$-tuples $\bar \bg$. It is not an algebraic process.

\cite[\S3.1]{BFr}  reviews
common  equivalences with examples and relevant definitions, such as the group
$\sQ''$ below. Let
$N_{S_n}(G,\bfC)$ be those $g\in S_n$ normalizing $G$ and permuting the
collection of conjugacy classes in $\bfC$. Absolute (resp.~inner) equivalence classes of
covers (with branch points at
$\bz$) correspond to the elements of
$\ni(G,\bfC)/N_{S_n}(G,\bfC))$ (resp.~$\ni(G,\bfC)/G$). 
\cite[\S5]{thompson} uses {\sl absolute} and {\sl inner\/} (and for each of these {\sl reduced\/})
equivalence. These show how to compute specific properties of manifolds
$\sH(G,\bfC)^\abs$, $\sH(G,\bfC)^\inn$ and their reduced versions, parametrizing the
equivalences classes of covers as $\bz$ varies. Orbits of the Hurwitz monodromy group $H_r$ 
on the respective absolute and inner Nielsen classes determine components of these spaces. Here is the $H_r$ action
using generators $\row q {r-1}$ on $\bg\in \ni(G,\bfC)$:  
\begin{equation} \label{HurMon} q_i: \bg=(\row g r) \mapsto (\row g {i-1},
g_ig_{i+1}g_i^{-1},g_i,g_{i+2},\dots,g_r).\end{equation} 

\subsection{Reduced Nielsen classes when $r=4$} \label{nc2}  {\sl Reduced equivalence\/}
of covers equivalences a cover of $\prP^1_z$, 
$\phi: X\to
\prP^1_z$, with any cover 
$\alpha\circ \phi: X\to \prP^1_z$ from composing $\phi$ with $\alpha\in \PGL_2(\bC)$. This makes sense for
covers with any number $r$ of branch points, though the case $r=4$ has classical motivation. Then, the
$\PGL_2$ action associates to the branch point set $\bz$ a $j$-invariant. You can think of it as the $j$-invariant
of the genus 1 curve mapping 2-to-1 to $\prP^1_z$ and branched at $\bz$. The branch point set
$\bz$ of a cover is {\sl elliptic\/} if it  equals that of an
elliptic curve with automorphism group of order larger than 2. 

We now review from \cite[\S2.6 and \S3.7.2]{BFr} how Nielsen classes describe the collection of reduced classes of
covers up to  inner or absolute equivalence that have a particular non-elliptic value of $j$ as their invariant.
Indeed, this set is just the inner or absolute Nielsen classes modulo an action of a quaternion group $\sQ\le H_4$
on the respective Nielsen classes. The action of $\sQ=\lrang{(q_1q_2q_3)^2,q_1q_3^{-1}}$ (using
\eqref{HurMon}) factors through a Klein group action
$\sQ''$. This arises from there always being a Klein 4-group ($\cong \bZ/2\times\bZ/2$) in
$\PGL_2(\bC)$ leaving the branch point set $\bz$ fixed. (An
even larger group leaves elliptic $\bz$ fixed.) Then, absolute reduced and
inner reduced equivalence have respective representatives  
$$\ni(G,\bfC)/\lrang{N_{S_n}(G,\bfC),\sQ''}\text{ and }
\ni(G,\bfC)/\lrang{N_{S_n}(G,\bfC),\sQ''}.$$   
When $r=4$,
these give formulas for branch cycles presenting
$\sH(G,\bfC)^{\abs,\rd}$ and $\sH(G,\bfC)^{\inn,\rd}$ as quotients of the upper half plane
by a finite index subgroup of $\PSL_2(\bZ)$ as a ramified cover of the classical $j$-line.
These  branch over the traditional places (normalized in \cite[Prop.~4.4]{BFr} to 
$j=0,1,\infty$) with the points over $\infty$ {\sl meaningfully\/}  called cusps.  

\cite[\S4]{thompson} has many examples of this. 
For example:  \cite[Prop.~4.1]{thompson}  uses these tools to 
produce a genus 0 $j$-line cover (dessins d'enfant) defined over $\bQ$ that parametrizes the pairs $(f,g)$ 
of  reduced classes of degree 7 Davenport polynomial pairs. As a parameter space for the 1st (resp.~2nd) coordinate
$f$ (resp.~$g$) the two families are defined and conjugate over 
$\bQ(\sqrt{-7})$. 

A cover (over $K$) in the Nielsen class 
$\ni(G,\bfC)$ with arithmetic monodromy group $\hat G$ is 
a $(G,\hat G,\bfC)$ {\sl realization\/} (over $K$). 

\subsection{Algebraist's branch cycles} \label{algBrCyc} Grothendieck's Theorem
\cite{Gr} gives us branch cycles for any tame cover, even in positive
characteristic. We state its meaning (\cite[Chap.~4, Prop.~2.11]{FB} has 
details). Consider a perfect algebraically closed field $\bar F$. For $z'\in
\prP^1_z(\bar F)$ and $e$ a positive integer prime to $\text{char}(\bar K)$, denote the
field of Laurent formal series 
$\bar F(((z-z')^{1/e}))$ by 
$\sP_{z,e}$. We choose a
compatible set $\{\zeta_e\}_{\{e | (e,\text{char}(\bar K)=1)\}}$ of roots of 1. Let
$\sigma_{z',e}:
\sP_{z,e}\to \sP_{z,e}$  be the automorphism (fixed on $\bar K((z-z'))$) that acts by 
$(z-z')^{1/e}\mapsto \zeta_e(z-z')^{1/e}$. Let $\bz=\{\row z r\}$ be $r$ distinct points of $\prP^1_z$. 

\begin{prop}[Algebraist branch cycles] \label{algRET} Assume  
$\hat L$ is the Galois closure of a tamely ramified extension $L/\bar
F(z)$ having branch points
$\bz$. Then there are embeddings $\psi_i: \hat L\to \sP_{z_i,e_i}$ with
$e_i$ the ramification index of $\hat L$ over $z_i$ satisfying this. The restrictions
$g_{z_i,\psi_i}\in G_f$ of $\sigma_{z_i,e_i}$ to $\hat L$, $i=1,\dots,r$, have the
generation and product-one  properties  \eqref{NCDef} 
\cite[Chap.2~\S7.5]{FB}. 
\end{prop}

Suppose given $r$ distinct points on $\prP^1_z$. Then, any set of classical generators
(as in \S\ref{nc1}) of $\pi_1(U_\bz,z_0)$  produces the collection $\bg=(\dots, g_{z_i,\psi_i}, \dots)$
for all covers in Prop.~\ref{algRET}. These are also compatible, in the following sense.
Given branch cycles for
$\phi:X\to \prP^1_z$ appearing in a chain $\psi: X\mapright{\phi'}X'\to \prP^1_z$, this uniquely
gives branch cycles of $\phi'$ (dependent on $\psi$). 

Also, we explain how fiber products
alone give a notion of compatibility without any appeal to paths.  Let $\phi_i: X_i\to \prP^1_z$ and assume $\phi: X
\to \prP^1_z$ is a cover defined by a $\bar F$ component $X$ of
$X_1\times_{\prP^1_z}X_2$. Suppose $\bg_i$ is a branch cycle description for $\phi_i$,
$i=1,2$. We say $\bg_1$ and $\bg_2$ are compatible if there are branch cycles 
$\bg$ for $\phi$ that restrict to $\bg_i$ on $\phi_i$, $i=1,2$, as in Prop.~\ref{algRET}.
Note: Referencing branch cycles gives meaning to the Nielsen class (any type) of
a tame cover in any characteristic. If we want to compare branch cycle descriptions of
a finite set of tamely ramified covers over $\prP^1_z$, we may take their fiber products
and a branch cycle description of a cover that dominates them all.

Suppose $\ni(G,\bfC)$ defines some Nielsen class (say absolute or inner; $r$ conjugacy classes). The
rest of Grothendieck's theorem requires 
$(|G|,\text{char}(\bar K))=1$. Then we interpret it as follows.  Given
$\bz$,
$r$ distinct points on
$\prP^1_z(\bar F)$, equivalence classes of covers in the Nielsen class with branch
points $\bz$ have a compatible set of branch cycle descriptions that correspond one-one 
with the Nielsen class representatives.

\section{Weil's cocycle condition and the Branch Cycle Lemma} \label{WCocycle} 
Often we apply Nielsen classes
to problems asking about the realization of covers over $\bQ$ or some variant like
 $(G,\hat G, \bfC)$ realization problems (\S\ref{nc2}). 

\subsection{The Branch Cycle Lemma story} \label{BCL} Realization problems, according to the  \BCL, 
require  $\bfC$, conjugacy classes in $G\le N_{S_n}(G,\bfC)\le S_n$, to be {\sl rational}. It
is now a staple of the theory of covers. 

\begin{defn} \label{ratUnion} Let $G^*$ be a group between $G$ and $N_{S_n}(G,\bfC)$.
Suppose for each integer
$k$ prime to the orders of elements in
$\bfC$, there is $h=h_k\in G^*$ and $\pi\in S_r$ so that we have the identity
$h\C_{(i)\pi}h^{-1}=\C_{i}^k$, $i=1,\dots,r$, in conjugacy classes. Then, $\bfC$ is a rational
union of conjugacy classes $\mod G^*$. \end{defn}

For this special case of \cite[Thm.~5.1]{FrHFGG}, the {\sl Branch Cycle Lemma\/} (BCL)
says $\bfC$ is a rational union of conjugacy classes $\mod G'$ is a necessary condition for
a $(G,G'',\bfC)$ realization with $G\le G''\le G'$. 

Some version of the BCL and Weil's cocycle
condition is now standard to determine when 
equivalence classes of covers have equations over the smallest possible field
one could expect for that. Though standard, getting it there required
getting researchers to master the notion of Nielsen class. For
example, in the special case mentioned above of Davenport pairs,  the BCL was the main tool in
\cite[\S3]{FrRedPol}. 
\cite{FrHFGG} proved converses of the conclusion of the BCL, by formulating
{\sl Braid rigidity\/} (though not calling it that). In 
 \cite{thompson} examples \wsp giving complete details on  the parameter spaces of Davenport
pairs of indecomposable polynomials over number fields \wsp the Braid Rigidity hypothesis holds and we apply
the converse. 

\subsection{Weil's cocycle condition and its place in the literature}  \S\ref{WWorks} explains how Weil's cocycle
condition works for families of covers, then \S\ref{WHist} tells some history behind it. 

\subsubsection{How the co-cycle condition works} \label{WWorks} 
Suppose
$\phi: X\to Y$ is a cover with  $Y$ embedded in some ambient projective space over a perfect
field $F$ and $X$ similarly embedded in a projective space over $\bar F$. Then, consider
$G_\phi$: $$\sigma\in G(\bar F/F)\text{ for which there exists }\psi_\sigma:
X^\sigma\to X\text{ so  }\phi\circ \psi_\sigma
=\phi^\sigma.$$ Denote the fixed field of
$G_\phi$ in $\bar F$ by $L_\phi$.

\begin{prop} Assume also, there is no isomorphism $\psi: X\to X$ that commutes with
$\phi$. Then, there is a cover
$\phi': X'\to Y$ with $L_\phi$ a field of definition of $X'$ and $\phi'$, and an
isomorphism  $\psi': X'\to X$ with $\phi\circ \psi'=\phi'$. \end{prop} 

\begin{proof} Regard the pairs $\{(X^\sigma, \phi^\sigma)\}_{\sigma\in G(\bar F/F)}$
as a subvariety of some ambient projective space. 
Then, $\psi_\sigma$ induces an isomorphism $(X^\sigma, \phi^\sigma)\to (X,\phi)$, and this
gives an isomorphism $\psi_\tau\circ \psi_\sigma^{-1}=\psi_{\sigma,\tau}: (X^\sigma,
\phi^\sigma)\to (X^\tau,\phi^\tau)$. That there is  no automorphism $\psi:
X\to X$ that commutes with
$\phi$ implies that for $\sigma, \tau, \gamma\in G_\phi$, $$\psi_{\tau,\gamma}\circ
\psi_{\sigma,\tau}=\psi_{\sigma,\gamma}.$$ This is the co-cycle condition attached to our
situation. 

The conclusion is the existence of an actual pair $(X',\phi')$ over $L_\phi$ by applying \cite{weil}. Examples with
the covers represented by polynomials appear in \cite[\S4 and \S5]{thompson} with, typical of its use, 
a much stronger conclusion: The whole family of covers in a Nielsen class has definition field 
$\bQ$. 
\end{proof}

\subsubsection{Some history of applying the co-cycle condition to families of covers} \label{WHist} 
I learned the Weil cocycle condition from the 1961 version of \cite[p.~27]{shimura}
when  I learned complex multiplication
studying with Shimura during my years '67--'69 at IAS. I showed Shimura the
BCL, and the effect of applying the Weil cocycle condition to the arithmetic
of covers. In particular, I showed its application to Davenport pairs. This produced
curves with field of moduli $\bQ$ not equal to their field of definition. Those first
curves   were the Galois closures of Davenport pairs $(f,g)$, such as those
of degree 7  over $\bQ(\sqrt{-7})$. 

As in \cite[Prop.~3]{FrRedPol}, the  arithmetic Galois
closures $\hat X$  of the covers from $f$ and $g$ 
are the same, and the BCL showed 
$f$ and $g$ are conjugate. So, the field of moduli of $\hat X$ as a Galois extension of
$\prP^1_z$ is $\bQ$ (an inner equivalence class as in \S\ref{nc1}): The field of moduli of
the cover together with its automorphisms. If, however, $\bQ$ were 
its field of definition, then the subgroups corresponding to the covers given by $f$ and
$g$ would also be over $\bQ$. So, the field of definition for this equivalence of covers is
not $\bQ$. It is easy to show the full automorphism group of $\hat X$ in
this case is $\PGL_3(\bZ/2)$ together with its diagram automorphism, and from that to
conclude the field of moduli of
$\hat X$ is not a field of definition for it. 

Shih's paper \cite{shih}, with
some version of the BCL, was in print before
\cite{FrHFGG} (though not before \cite{FrRedPol}). Some authors have revised the
situation of its priority, saying the results were done independently. 

\cite{FrHFGG} was half of an original paper that was in Shimura's hands by Fall of 1971. It was  broken apart in
Spring of 1972 when I was again at IAS.  Shimura  sent Shih to visit me when I was
at MIT, fall 1971, on a Sloan. This resulted from Shimura
asking me to give an elementary approach to canonical fields of definition.  My answer
was the Hurwitz space approach, using the BCL, and applying it in
particular to modular curves in \cite{FrGGCM} (the other half of the 1971 preprint). I said I would quote
\cite{shih}, and he could use the BCL if he said from where he got it. I did my part. He
did not. 

%

\section{DPs and the genus 0 problem} \label{MullerDPs}  Davenport phrased his problem starting over $\bQ$ and at
least for indecomposable polynomials, \cite[Thm.~2]{FrRedPol} showed it was true: Two polynomials $f,g\in \bQ[x]$
with the same ranges modulo almost all primes $p$ are linearly related: $f(ax+b)=g(x)$ for some $a,b\in \bar \bQ$. 
Because of indecomposability, we actually may take $a,b\in \bQ$ (Rem.~\ref{QvsbQ}). \S\ref{MuelList}  is a
complement to
\cite[\S4 and \S5]{thompson}. 

We consider indecomposable polynomial DPs over a number
field
$K$. These are essential cases in the genus 0 problem. The polynomials that arise in serious arithmetic problems
aren't generic. So, in continuing \S\ref{nonSporadics} we show how Davenport's Problem relates to  0-sporadic
polynomials. M\"uller's Theorem in this direction is a gem from my view for two reasons. It shows how truly
significant Davenport Pairs were to this direction, and it is easy to understand.

\subsection{M\"uller's list of primitive polynomial monodromy and DPs}
\label{MuelList} Suppose $(f,g)$ is a DP over a number field  $K$
($f,g\in K[x])$.    We always assume $(f,g)$ are not affine equivalent.  Lem.~\ref{basicSchur} says
that $f$ indecomposable translates to $f:\prP^1_x\to\prP^1_z$ having doubly transitive
geometric monodromy. In particular it says $f$ is not exceptional. 
\cite[Cor.~7.30]{AFH} showed $g=g_1(g_2(x))$ is a decomposition (over $K$) with $(f,g_1)$ an
iSDP.

\subsubsection{The three 1-dimensional reduced spaces of 0-sporadic polynomial covers} \label{3-1dim}  You don't have
to be a group theorist to read the list from \cite{primPol} of  primitive
polonomial groups that are not cyclic, dihedral,
$A_n$ or
$S_n$.  

Our version of M\"uller's list shows how pertinent was Davenport's  problem.   
All appearing groups are almost simple (\S\ref{enthCrypt}). Exclude  those (finitely many)
that normalize  
$\PSL_2(\bF_q)$ (for very small $q$) and the degree 11 and 23 Matthieu groups. Then,
all  remaining  $G$ are from 
\cite{FrRedPol} and they have these objects.       
\begin{edesc} \item Two inequivalent doubly transitive representations, 
equivalent as  (degree $n$) group representations; and    
\item  an 
$n$-cycle (for these  representations). \end{edesc} We know such groups. There is one   
of degree 11. The others  are Chevalley groups that normalize 
$\PSL_{u+1}(\bF_q)$ (acting on points and hyperplanes of $\prP^u$).
\cite[\S9]{Fr-Schconf} reviews and completes this. All six  (with corresponding
Nielsen classes) give  Davenport pairs. 
We concentrate on those three with one extra
property:   
\begin{trivl} \label{dim1mod} Modulo
$\PGL_2(\bC)$ (reduced equivalence as in \S\ref{nc2})  action, the space of these
polynomials has dimension at least (in all cases, equal) 1.  \end{trivl}

These properties hold for sporadic 
polynomial maps with $r\ge 4$ branch points. 
\begin{itemize} \item They have degrees  
from
$\{7,13,15\}$ and $r=4$.
\item All  $r\ge 4$ branch point indecomposable polynomial maps in an iDP pair are in one
of the respectively, 2, 4 or 2 Nielsen classes corresponding to the respective degrees 7,
13 and  15.
\end{itemize} 
\cite{FrRedPol} outlines this. 

\cite[\S8]{Fr-Schconf}, 
\cite{MeuDavI} and
\cite[\S2.7]{primPol} say much on the group theory of the indecomposable
polynomial SDPs over number fields. Yet, we now say something new on the definition
field of these families, a subtlety on dessins d'enfant, presented as genus 0
$j$-line covers.   Let 
$\sH_7^\D$, $\sH^\D_{13}$ and $\sH^\D_{15}$ denote the spaces of polynomial covers that are
one from a Davenport pair having four branch points (counting
$\infty$). The subscript decoration corresponds to the respective
degrees. We assume absolute, reduced equivalence (as in \S\ref{nc2}). 
Then, all these spaces are irreducible and defined over $\bQ$ as covers of the $j$-line. Each $\sH_n^\D$
is labeled by a difference set modulo $n$, $n=7,13,15$, and there is an action of $G_{\bQ}$ on the the difference
sets (modulo translation) \cite[\S2.3]{thompson}.

In these cases, analytic families of respective degree $n$ polynomials fall into
several components ($\sH_7^\D$ are those of degree 7). Yet, each component corresponds to
a unique Nielsen class and a particular value of $D$.  We understand these Nielsen classes and the definition fields
of these components from the BCL. 

\begin{rem}[Linearly related over $\bQ$ versus over $\bar \bQ$] \label{QvsbQ} The comments on
proof in Prop.~\ref{ordSchurT} note the degree $n$ Chebychev polynomial $T_n$  gives all Dickson polynomials by
composing with linear fractional transformations in the form  $l_u\circ T_n\circ l_{u^{-1}}$. All Dickson
polynomials of degree $n$ over a given finite field have the same exceptional polynomial behavior and branch cycle
descriptions placing them in one family. Whether you see them as significantly different depends on your
perspective. I tend to downplay this, though there are times it is worthy to consider.  

\cite[Thm.~2]{FrRedPol} {\sl does\/} have the  conclusion that indecomposable DPs over $\bQ$ are linearly related
over $\bQ$. Still, there are elementary examples of (composable) Davenport pairs, linearly related over $\bar \bQ$
and not over
$\bQ$. Davenport likely knew those for he used the same examples elsewhere: $(h(x^8),h(16x^8))$ with $h\in \bQ[x]$
are a Davenport pair, linearly related over $\bar \bQ$ \cite[Rem.~21.6.1]{FrJ4}. 
\end{rem}

\subsubsection{Masking} Consider the statement in the paragraph starting \S\ref{MuelList}. 
 One possibility not yet excluded for $(f,g_1)$ from
\cite[Cor.~7.30]{AFH} is that
$g_1$ is affine equivalent to $f$, and yet $g_2$ is not exceptional. 

This  has
an analog over a finite field. Possibly
$g$ and
$g\circ g_1$ have precisely the same range for
$\infty$-ly many residue classes of a number field (or extensions $\bF_{q^t}$) even though
$g_1$ is not exceptional. (\cite{FrRedPol}, for example, shows this can't be if $f$
and $g_1$ have the same ranges on almost all residue class fields, or on all
extensions of
$\bF_q$.) 
 
\cite[Def.~1.3]{AFH} calls this possibility an example of  {\sl masking}. \cite[\S
4]{MeuDavI} found a version of it, motivating our name.  

\subsection{Print version miscues in \cite{thompson}} Here are several typographical difficulties in the
final version of \cite{thompson}, though not in the files I sent the publishers.  

\newcommand{\Prob}[1]{\text{Problem$_{#1}^{\!\scriptscriptstyle{g=0\!}}$}}
\begin{itemize} \item Expressions  $\Prob n$ (for $n=1$ and 2 representing two
distinct problems John Thompson considered) appear  as $\Prob 0 n$. 

\item Throughout the manuscript, whenever a reference is made to an expression in a section or subsection, the
reference came out to be a meaningless number. So \S3.2 titled: Difference sets give properties
(3.1a) and (3.2b), had those last two references appear as (91) and (92). We follow this pattern in the other
cases, labeling the sections and giving the changes in the form  (91) $\mapsto$ (3.1a) and (92) $\mapsto$ (3.2b). 

\S3.3: (92) $\mapsto$ (3.1b).

\S5.2.1: (171) $\mapsto$ (5.3a)

\S 5.2.2 (172)  $\mapsto$ (5.3b)

\S 5.2.3 (172)  $\mapsto$ (5.3b)
\end{itemize} 

\end{appendix}

\providecommand{\bysame}{\leavevmode\hbox to3em{\hrulefill}\thinspace}

\end{document}